\documentclass[11pt]{amsart}
\usepackage{amssymb,amsmath,amsthm}
\usepackage{amsfonts}
\usepackage{enumerate}
\usepackage{dsfont}
\usepackage{cite}
\usepackage{color} 
\usepackage[colorlinks,citecolor=red]{hyperref}
\usepackage{mathrsfs}
\usepackage{epsfig}
\usepackage{lscape}
\usepackage{subfigure}
\usepackage{epstopdf}
\usepackage{caption}
\usepackage{algorithm}
\usepackage{algpseudocode}
\usepackage{multirow}
\usepackage{geometry}
\usepackage{paralist}
\usepackage{enumerate}
\usepackage{url}
\usepackage{graphicx,cite}
\usepackage{booktabs}   
\usepackage{makecell}   
\usepackage{array}      
\usepackage{makecell}
\usepackage{longtable}
\usepackage{tabularx} 
\usepackage{tikz}
\usepackage{bbm}

\newcommand{\argmin}[1]{\mathop{\rm argmin}\limits_{#1}}

\newtheorem{definition}{Definition}[section]

\newtheorem{prop}[definition]{Proposition}
\newtheorem{theorem}[definition]{Theorem}
\newtheorem{lemma}[definition]{Lemma}

\newtheorem{remark}[definition]{Remark}

\newtheorem{assumption}{Assumption}[section]
\date{}
\begin{document}
\newpage
\baselineskip 18pt
\bibliographystyle{plain}
\title[Connecting  RIM and Krylov subspaces]
{Connecting randomized iterative methods with Krylov subspaces}

\author{Yonghan Sun}
\address{School of Mathematical Sciences, Beihang University, Beijing, 100191, China. }
\email{sunyonghan@buaa.edu.cn}

\author{Deren Han}
\address{LMIB of the Ministry of Education, School of Mathematical Sciences, Beihang University, Beijing, 100191, China. }
\email{handr@buaa.edu.cn}

\author{Jiaxin Xie}
\address{LMIB of the Ministry of Education, School of Mathematical Sciences, Beihang University, Beijing, 100191, China. }
\email{xiejx@buaa.edu.cn}

\begin{abstract}
Randomized iterative methods, such as the randomized Kaczmarz method, have gained significant attention for solving large-scale linear systems due to their simplicity and efficiency. Meanwhile, Krylov subspace methods have emerged as a powerful class of algorithms, known for their robust theoretical foundations and rapid convergence properties. 
Despite the individual successes of these two paradigms,  their underlying connection has remained largely unexplored. 
In this paper, we develop a unified framework that bridges randomized iterative methods and Krylov subspace techniques, supported by both rigorous theoretical analysis and practical implementation. The core idea is to formulate each iteration as an adaptively weighted linear combination of the sketched normal vector and previous iterates, with the weights optimally determined via a projection-based mechanism. This formulation not only reveals how subspace techniques can enhance the efficiency of randomized iterative methods, but also enables the design of a new class of iterative-sketching-based Krylov subspace algorithms. We  prove that our method converges linearly in expectation and validate our findings with numerical experiments. 

\end{abstract}

\maketitle

\let\thefootnote\relax

\let\thefootnote\relax

\section{Introduction}

Solving linear systems of the form 
\begin{equation}\label{main-prob}
	Ax=b, \ A\in\mathbb{R}^{m\times n}, \ b\in\mathbb{R}^m
\end{equation}
is a fundamental task in various applications, including statistics \cite{hastie2009elements}, scientific computing \cite{golub2013matrix}, and machine learning  \cite{boyd2018introduction,bishop2006pattern}. A primary challenge is the increasing size of these systems, driven by the advent of the big data era. Recent efforts have focused on developing randomized iterative methods \cite{needell2014stochasticMP,necoara2019faster,strohmer2009randomized,han2024randomized,zeng2024adaptive,Gow15} characterized by low per-iteration costs, optimal complexity, ease of implementation, and effectiveness for large-scale linear systems.

In addition to randomized iterative methods, Krylov subspace methods \cite{saad2003iterative,golub2013matrix,liesen2013krylov} have long been a cornerstone of numerical linear algebra for solving linear systems. Methods such as the conjugate gradient (CG) and generalized minimal residual (GMRES) are renowned for their strong theoretical guarantees and fast convergence properties. However, the connection between randomized iterative methods and Krylov subspace methods has yet to be established. This paper investigates their connections and proposes new hybrid algorithms that combine the strengths of both approaches.

Let $A^\dagger$ denote the Moore-Penrose pseudoinverse of $A$. Starting from any initial point \(x^0\), at the \(k\)-th iterate, we generate \(x^{k+1}\) via the projection
\begin{equation}\label{RIM_subspace}
x^{k+1} = \mathop{\arg\min}_{x \in \Pi_{k}} \lVert x - A^\dagger b \rVert_2^2, 
\end{equation}
where $\Pi_{k}$ is the affine subspace defined by
$$
\Pi_{k} := \operatorname{aff}\left\{x^{j_{k}}, \ldots, x^k, x^k - A^\top S_k S^\top_k (Ax^k - b)\right\}.
$$
Here, \(j_{k} := \max\{k - \ell + 1, 0\}\) tracks the most recent \(\ell\) iterates, with \(\ell\) being a positive integer. The matrix \(S_k \in \mathbb{R}^{m \times q}\) is an iterative sketching matrix sampled from a user-defined probability space \((\Omega, \mathcal{F}, \mathbf{P})\). The term \(A^\top S_k S^\top_k (Ax^k - b)\) represents the sketched normal vector, and \(\top\) denotes the transpose of a matrix or vector.
 We note that when $\ell=1$, the iteration scheme \eqref{RIM_subspace} can simplify to  randomized iterative methods for solving linear systems, such as the randomized Kaczmarz (RK) method \cite{strohmer2009randomized}; see Remark \ref{remark-xie-0325-2}. On the other hand, when $\ell=\infty$, this iteration scheme  can recover Krylov subspace methods; see Section \ref{section-ISK}. This establishes a bridge between randomized iterative methods and Krylov subspace methods, where $\ell$ serving as the truncated parameter.

\subsection{Our contribution}
The main contributions of this work are as follows.
\begin{itemize}
\item [1.] The efficiency of our algorithm relies on maintaining the affine independence of the affine subspace \(\Pi_{k}\), i.e., \(\dim(\Pi_{k}) = k - j_{k} + 1\).
We  demonstrate  that ensuring \(S^\top_k (Ax^k - b) \neq 0\) at every step is sufficient to guarantee that the affine subspace remains affinely independent throughout the process.
\item [2.] Although problem \eqref{RIM_subspace} involves the unknown vector \( A^\dagger b \), we show that \( x^{k+1} \) can be obtained without explicitly computing \( A^\dagger b \) by leveraging the orthogonal projection design. In addition, we propose an efficient approach to solve problem \eqref{RIM_subspace} with a computational complexity of \( O(n(k-j_k+1)) \), which is linear in the problem dimension.  
\item [3.] We derive alternative equivalent representations of the affine subspace \(\Pi_{k}\) and demonstrate its connection to Krylov subspaces. Notably,  when \(\ell = \infty\) and \(\Omega = \{I\}\), \(\Pi_{k}\) reduces to a Krylov subspace. Building on this insight, we analyze our method within the Krylov subspace framework and develop a novel iterative-sketching-based Krylov subspace approach for solving linear systems. 
\item [4.] We prove that our method converges linearly in expectation, with the convergence factor influenced by the number of previous iterates \(\ell\). In particular, as \(\ell\) increases, the convergence upper bound decreases, resulting in faster convergence. Furthermore, if the  parameter \(\ell\) is greater than or equal to the rank of the coefficient matrix \(A\), the algorithm exhibits finite termination in the absence of round-off errors.
The conclusions  are validated by the results of our numerical experiments.
\end{itemize}

\subsection{Related work}
\subsubsection{Randomized iterative methods}

Randomized iterative methods form a class of algorithms that integrate two significant paradigms in numerical computation. These methods can be seen as both a natural extension of the RK method \cite{strohmer2009randomized} and a specialized form of stochastic gradient descent (SGD) \cite{robbins1951stochastic,Han2022-xh,needell2014stochasticMP,ma2017stochastic,zeng2023randomized} for solving least-squares problems. The foundation of this approach traces back to the classical Kaczmarz method \cite{Kac37}, also known as the algebraic reconstruction technique (ART) \cite{herman1993algebraic,gordon1970algebraic}, which provides an efficient row-action method for solving large-scale linear systems. At each iteration, the method alternates between choosing a row of the system and updating the solution based on the projection onto the hyperplane defined by that row. Strohmer and Vershynin \cite{strohmer2009randomized} introduced its randomized variant, where rows are sampled with probabilities proportional to their row norms, and proved that RK converges linearly in expectation for consistent linear systems.  Due to their simplicity and efficiency,
Kaczmarz-type methods have been widely applied in various fields,
including compressed sensing \cite{schopfer2019linear,yuan2022adaptively,jeong2023linear}, phase retrieval \cite{tan2019phase,huang2022linear},  tensor recovery \cite{chen2021regularized,ma2022randomized}, ridge regression \cite{hefny2017rows,gazagnadou2022ridgesketch}, adversarial optimization \cite{huang2024randomized}, and absolute value equations \cite{xie2024randomized}.

The SGD framework minimizes functions of the form $f(x)=\frac{1}{m}\sum_{i=1}^mf_i(x)$ through the iterative update
\begin{equation}
	\label{SGD-iteration}
	x^{k+1}=x^k-\alpha_k\nabla f_{i_k}(x^k),
\end{equation}
where  $\nabla$ is the gradient operator, $\alpha_k$ is the step-size and $i_k$ is selected randomly. This stochastic approach provides significant computational advantages for large-scale problems by utilizing partial gradient information at each iteration. From an optimization perspective, these methods can be precisely characterized as SGD applied to the least-squares problem $f(x)=\frac{1}{2m}\|Ax-b\|^2_2=\frac{1}{2m}\sum_{i=1}^{m}(A_{i,:}x-b_i)^2$, where the classical SGD formulation reduces precisely to the RK method. Here, $A_{i,:}$ denotes the $i$-th row of the matrix $A$.
In fact, when \(\ell=1\), our method follows the iteration scheme \(x^{k+1}=x^k-\alpha_k A^\top S_kS_k^\top (Ax^k-b)\), which can be interpreted as applying SGD to solve the following stochastic optimization problem
\begin{equation}
	\label{s-reformulation}
	\min_{x \in \mathbb{R}^n} f(x):=\mathbb{E} \left[f_{S}(x)\right],
\end{equation}
where the expectation is taken over randomized sketching matrices \(S\) drawn from a probability space \((\Omega, \mathcal{F}, \mathbf{P})\). The function \(f_{S}(x)\) is defined as
$f_{S}(x):=\frac{1}{2} \left\|S^\top(Ax-b)\right\|^2_2$.

\subsubsection{Acceleration techniques}

The development of acceleration techniques for randomized iterative methods has been extensively studied in the literature.   Current research primarily investigates various acceleration strategies, including Polyak's heavy ball momentum (HBM) \cite{polyak1964some,loizou2020momentum,han2022pseudoinverse,Han2022-xh,xiang2025randomized,zeng2024adaptive}, Nesterov's momentum \cite{nesterov1983method,nesterov2003introductory,lan2020first,liu2016accelerated}, Chebyshev-based stepsize selection \cite{necoara2019faster}, inertial extrapolation \cite{he2024inertial,su2024greedy}, 
subspace techniques \cite{liu2021subspace,yuan2014review}, 
and Gearhart-Koshy acceleration \cite{rieger2023generalized,tam2021gearhart,gearhart1989acceleration,hegland2023generalized}, to enhance algorithmic performance.  Below, we discuss several works that are particularly relevant to our research.


Our work is closely related to the work of  \cite{zeng2024adaptive}, where an adaptive stochastic HBM (ASHBM) was developed to  address the stochastic problem \eqref{s-reformulation}. 
The ASHBM framework incorporates Polyak's HBM into SGD \eqref{SGD-iteration} through the iterative scheme
$ x^{k+1} = x^k - \alpha_k \nabla f_{i_k}(x^k) + \beta_k (x^k - x^{k-1})$, 
where $\alpha_k$ and $\beta_k$ are adaptively determined by solving the constrained optimization problem
$$	
	\min\limits_{ \alpha,\beta\in\mathbb{R}} \ \ \|x-A^\dagger b\|_2^2 \ \ 
\text{subject to} \ \ x=x^{k}-\alpha \nabla f_{S_k}(x^{k})+\beta(x^{k}-x^{k-1}).
$$
Note that $\nabla f_{S_k}(x^{k})=A^\top SS^\top (Ax^k-b)$ and $\operatorname{aff}\{x^{k-1},x^k, x^k-A^\top SS^\top (Ax^k-b)\}=x^k+\operatorname{span}\{x^{k-1}-x^{k}, -A^\top SS^\top (Ax^k-b)\}$. Therefore, the ASHBM method computes $x^{k+1}$ as
$$	
\begin{aligned}	
	x^{k+1}=\arg\min \ \ \|x-A^\dagger b\|_2^2 \ \
	\text{subject to} \ \  x\in \operatorname{aff}\{x^{k-1},x^k, x^k-A^\top SS^\top (Ax^k-b)\}.
\end{aligned}
$$
This approach aligns with the proposed method \eqref{RIM_subspace} with \(\ell=2\).  In fact, the proposed method \eqref{RIM_subspace} considers an affine subspace of dimension \(\ell\), which enables the algorithm to leverage information from more previous iterates if \(\ell > 2\), thereby achieving superior convergence efficiency. Moreover, our theoretical analysis differs from that of \cite{zeng2024adaptive}, particularly in the treatment of efficient implementation.

Another closely related approach is the generalized Gearhart-Koshy accelerated Kaczmarz method proposed in \cite{rieger2023generalized}. 
The method operates through an epoch-based iteration scheme with parameter \(\ell\), where \(j_{k}= \max\{k - \ell + 1, 0\}\) tracks the historical iterates. Starting from initial point \(x^0\), each epoch $k$ begins with $x^{k,0} = x^k$ and executes $m$ sequential Kaczmarz updates
\[
x^{k,i} = x^{k,i-1} - \frac{  A_{i,:}x^{k,i-1}-b_{i}}{\|A_{i,:}\|^2_2}A_{i,:}^\top, \quad i = 1, \ldots, m.
\]
The subsequent iterate is obtained via affine subspace projection
\[
x^{k+1} = \argmin \ \|x-A^\dagger b\|^2 \ \ \text{subject to} \ \ x\in\operatorname{aff}\{x^{j_k},\ldots,x^k,x^{k,m}\}.
\]
  Although this work shares conceptual similarities with our approach in utilizing affine subspace projections, there are several key differences. First, we employ a randomized framework compared to their deterministic approach.  When using the probability space \(\Omega = \{e_i\}_{i=1}^m\), our algorithm requires sampling only one row from matrix \(A\) for each computation, while their algorithm needs to perform \(m\) sequential updates. Second, their algorithm lacks convergence rate analysis, whereas our algorithm achieves linear convergence in expectation. Finally, by manipulating the probability spaces, we can design more versatile hybrid algorithms with improved performance, accelerated convergence, and better scalability.

\subsubsection{Randomized-sketching-based methods}


Randomized sketching-based methods \cite{raskutti2016statistical,guttel2024sketch,burke2025gmres}, also known as sketch-and-solve methods \cite{nakatsukasa2024fast,derezinski2024recent,martinsson2020randomized}, offer an efficient framework for solving large-scale problems through dimensionality reduction. These methods construct low-dimensional subproblems using randomly generated sketching matrices that satisfy the subspace embedding condition \cite{balabanov2022randomized,guttel2024sketch}, a property can be ensured by the Johnson-Lindenstrauss lemma \cite{johnson1984extensions}. Notable variants include sketching GMRES (SGMRES) \cite{nakatsukasa2024fast} and randomized GMRES (RGMRES) \cite{balabanov2022randomized}, which can be viewed as randomized versions  of Krylov subspace methods for solving linear systems. Although effective in reducing problem size, these approaches typically use a fixed sketching matrix throughout the computation process.



 Our work employs an iterative-sketching paradigm \cite{derezinski2024recent,zeng2024adaptive,Gow15,pilanci2016iterative}, which is another important method in randomized numerical linear algebra for solving large-scale problems. 
The core idea is to iteratively refine solutions by solving a sequence of compressed subproblems generated through sketching, rather than directly addressing the original large-scale matrix or dataset. This approach relaxes the stringent requirements of randomized sketching methods by eliminating the need for subspace embedding conditions and dimensional constraints on the sketching matrices. We provide a detailed comparison in Subsection \ref{subsec 4.4}. 
\subsection{Organization}

The remainder of the paper is organized as follows.
After introducing some notations and preliminaries in Section $2$, we present and analyze the proposed method in Section $3$. The connection between the randomized iterative methods and the Krylov subspace methods is established in Section $4$. In Section $5$, 
we perform some numerical experiments to show the effectiveness of the proposed method. We conclude the paper in Section $6$. Proofs of the theorems with lengthy derivations are provided in the appendices.

	\section{Notation and Preliminaries}\label{sec-pre}	
	\subsection{Notations}

	For vector $x\in\mathbb{R}^n$, we use $x_i,x^\top$, and $\|x\|_2$ to denote the $i$-th entry, the transpose, and the Euclidean norm of $x$, respectively. 
	 We adopt a special decomposition for vector
	 \[
	 x = \begin{pmatrix} \overline{x} \\ \underline{x} \end{pmatrix}, \quad \text{where } \overline{x} := (x_1,\ldots,x_{n-1})^\top \in \mathbb{R}^{n-1} \text{ and } \underline{x} := x_n \in \mathbb{R}
	 \]
	The $n$-order identity matrix and the $i$-th unit vector are denoted by $I_n$ and $e_n^i\in\mathbb{R}^n$, respectively. 
	For any $j\in[n]$, we define 
	 $$\mathbbm{1}_n^j:=\sum_{i=1}^je_n^i.$$ 
	 We denote the affine subspace of vectors $x^1,\ldots, x^k \in \mathbb{R}^n$ as
	 \[
	 \operatorname{aff}\{ x^1,\ldots, x^k\} := \left\{\sum_{i=1}^k \alpha_i x^i \, \ \left |  \ \sum_{i=1}^k \alpha_i = 1, \alpha_i \in\mathbb{R} \right.\right\},
	 \]
	and their linear span as \( \operatorname{span}\{ x^1,\ldots, x^k \} \).
	 For any matrix $A \in \mathbb{R}^{m \times n}$, we use $A_{i,:}$, $A^\top$, $A^\dagger$, $\|A\|_F$, $\text{rank}(A)$, and $\text{Range}(A)$ to denote the $i$-th row, the transpose, the Moore-Penrose pseudoinverse, the Frobenius norm, the rank, and the column space, respectively. We use  $\sigma_{\min}(A)$ to denote the smallest nonzero singular value of $A$.
	 Given $\mathcal{J} \subseteq [m]:=\{1,\ldots,m\}$, the complementary set of $\mathcal{J} $ is denoted by $\mathcal{J} ^c$, i.e. $\mathcal{J} ^c=[m] \setminus \mathcal{J}$.
	 We use $A_{\mathcal{J},:}$ to denote the row submatrix indexed by $\mathcal{J}$.
	For any random variables $\xi_1$ and $\xi_2$, we use $\mathbb{E}[\xi_1]$ and $\mathbb{E}[\xi_1|\xi_2]$ to denote the expectation of $\xi_1$ and the conditional expectation of $\xi_1$ given $\xi_2$. 


\subsection{Some useful lemmas}
In this subsection, we recall some known results that we will need later on. 
In order to proceed, we shall propose a basic assumption on the probability space $(\Omega, \mathcal{F}, \mathbf{P})$ used in this paper.

\begin{assumption}
	\label{Assumption1}
	Let $(\Omega, \mathcal{F}, \mathbf{P})$  be  the probability space from which the iterative sketching matrices are drawn. We assume that  $\mathop{\mathbb{E}}_{S\in\Omega} \left[S S^\top\right]$ is a positive definite matrix.
\end{assumption}

The following lemma is crucial for our convergence analysis.

\begin{lemma}[\cite{lorenz2023minimal}, Lemma 2.3 ]
	\label{positive}
	Let $S\in\mathbb{R}^{m\times q}$ be a real-valued random variable defined on a probability space $(\Omega,\mathcal{F},\mathbf{P})$. Suppose that
	$
	\mathbb{E}\left[SS^\top\right]
	$
	is a positive definite matrix and $A\in\mathbb{R}^{m\times n}$ with $A\neq 0$.  Then
	$$
	H:=\mathbb{E}\left[\frac{SS^\top}{\|S^\top A\|^2_2}\right]
	$$
	is well-defined and positive definite, here we define $\frac{0}{0}=0$.
\end{lemma}

\begin{lemma}[\cite{zeng2024adaptive}, Lemma 2.3]
	\label{lemma-meanwhile}
	Assume that the linear system $Ax=b$ is consistent. Then for any matrix $S \in \mathbb{R}^{m\times q}$ and any vector $\tilde{x} \in \mathbb{R}^{n}$, it holds that $A^\top SS^\top(A\tilde{x}-b) \neq 0$ if and only if $S^\top(A\tilde{x}-b) \neq 0$.
\end{lemma}

\begin{lemma}[\cite{zeng2024adaptive}, Lemma 2.4]
	\label{xie-empty}
	Let matrix $S$ be a random variable such that $
	\mathbb{E}\left[SS^\top\right]
	$ is positive definite.
	Then $S^\top(Ax-b)=0$ for all $S$ holds if and only if $Ax=b$.
\end{lemma}

The following useful lemma characterizes the inverse of matrices with a  special structure.

\begin{lemma}[\cite{rieger2023generalized}, Lemma 13]
	\label{lemma-inv-sol}
Let \( B \in \mathbb{R}^{n \times n} \) be an invertible matrix, \( p \in \mathbb{R}^n \), and let \(\delta > 0\) and \(\gamma \in \mathbb{R}\). If the matrix 
\[
G := \begin{pmatrix} B & p \\ p^\top & \delta \end{pmatrix}
\]
is invertible, then \( p^\top B^{-1} p \neq \delta \). The solution to the linear system \( Gx = \gamma e_n^n \) is given by
\[
G^{-1} \gamma e_n^n = \frac{\gamma}{p^\top B^{-1} p - \delta} \begin{pmatrix} B^{-1} p \\ -1 \end{pmatrix}.
\]
\end{lemma}
    
    \section{Randomized iterative methods with affine subspace search}
    \label{sec-affine-search}

   In this section, we provide an in-depth examination of the iteration scheme \eqref{RIM_subspace} and analyze  how we ensure the efficient execution of the algorithm.
    Recall that the affine subspace in \eqref{RIM_subspace} is defined as
   $$
   \begin{aligned}
   	\Pi_{k} &= \operatorname{aff}\left\{x^{j_{k}}, \ldots, x^k, x^k - A^\top S_k S^\top_k (Ax^k - b)\right\}\\
   	&=x^k+\operatorname{span}\left\{x^{j_{k}}-x^k,\cdots,x^{k-1}-x^k, -A^\top S_k S^\top_k (Ax^k - b)\right\},
   \end{aligned}
   $$ 
   where $j_{k} = \max\{k - \ell + 1, 0\}$, \(\ell\) is a positive integer, and $S_k\in\mathbb{R}^{m\times q}$ is an iterative sketching matrix sampled from a user-defined probability space $(\Omega, \mathcal{F}, \mathbf{P})$.
   For any $k\geq0$,  we define
   \begin{equation}
   	\label{def-Mk}
   	M_{k}=\left(x^{j_{k}}-x^k,\cdots,x^{k-1}-x^k,-A^\top S_k S^\top_k (Ax^k - b)\right)\in\mathbb{R}^{n\times(k-j_{k}+1)}.
   \end{equation}
   Then, the affine subspace $\Pi_{k}$ is equivalent to $\Pi_{k}=x^k+\text{Range}(M_{k})$ and the iteration scheme  \eqref{RIM_subspace} reduces  to find $s_k$ such that
  $$ 		s_k=\mathop{\arg\min}\limits_{s\in\mathbb{R}^{k-j_k+1}}\lVert x^k+M_ks-A^\dagger b\rVert_2^2,$$
   which is equivalent to 
   \begin{equation}
    \label{xie-5-26-1}
   		M_k^\top M_ks_k=-M_k^\top(x^k-A^\dagger b).
   \end{equation}
   Analyzing the right-hand side vector $-M_k^\top(x^k-A^\dagger b)$, we obtain
   \begin{equation}\label{equality-opt}
   	\begin{cases}
   		-\langle x^{j_k+i-1}-x^k,x^k-A^\dagger b\rangle=0,\quad1\leq i\leq k-j_k,\\
   		-\langle -A^\top S_kS_k^\top(Ax^k-b),x^k-A^\dagger b\rangle=\lVert S_{k}^\top(Ax^k-b)\rVert_2^2,
   	\end{cases}
   		%
   \end{equation}
   where the equalities above follow from that $x^k$ is the orthogonal projection of $A^\dagger b$ onto the affine set $\Pi_{k-1}=\text{aff}\{x^{j_{k-1}},\ldots,x^{k-1},x^{k-1}-A^\top S_{k-1}S_{k-1}^\top(Ax^{k-1}-b)\}$. Let $\gamma_k:=\lVert S_{k}^\top(Ax^k-b)\rVert_2^2$, then the iteration scheme  \eqref{RIM_subspace} is to find $s_k$ such that
  \begin{equation}\label{eq-s_k}
  	M_k^\top M_ks_k=\gamma_ke_{k-j_k+1}^{k-j_k+1},
  \end{equation}	
   where $e_{k-j_k+1}^{k-j_k+1}\in\mathbb{R}^{k-j_k+1}$ is the $(k-j_k+1)$-th unit vector.

    The following proposition indicates that if the iterative sketching matrices $S_{k}$ are chosen such $S_{k}^\top(Ax^k-b)\neq0$ for $k\geq0$, then \(M_k\) is full column rank, i.e., dim$(\Pi_k)=k-j_k+1$.
    \begin{prop}\label{prob-full}
    	Let \(x^0 \in \text{Range}(A^\top)\) be an initial point, and let \(\{x^k\}_{k \geq 1}\) be the sequence obtained by solving the optimization problem \eqref{RIM_subspace}. If \(S_i^\top(Ax^i-b) \neq 0\) for \(i = 0, \ldots, k\), then \(M_k\) defined in \eqref{def-Mk} has full column rank.
    \end{prop}
    \begin{proof}

    	For \(k=0\), we have \(M_0 = -A^\top S_{0} S_{0}^\top (Ax^{0} - b)\). Thus, \(M_0\) is of full column rank provided that \(A^\top S_{0} S_{0}^\top (Ax^{0} - b) \neq 0\). From Lemma \ref{lemma-meanwhile}, we know \(S_0^\top (Ax^0 - b) \neq 0\) if and only if \(A^\top S_0 S_0^\top (Ax^0 - b) \neq 0\). Hence, if \(S_0^\top (Ax^0 - b) \neq 0\), \(M_0\) is of full column rank.
    	
    	Next, we prove by induction that \(M_k\) (for \(k \geq 1\)) has full column rank. Assume \(M_{k-1}\) has full column rank, we show that \(M_k\) is of full column rank provided \(S_{k}^\top (Ax^{k} - b) \neq 0\).
    	Suppose, for contradiction, that \(x^{j_{k}} - x^k, \ldots, x^{k-1} - x^k, -A^\top S_k S^\top_k (Ax^k - b)\) are linearly dependent. Then there exist scalars \(\{\lambda_i\}_{i=1}^{k-j_{k}}\), not all zero, satisfying
    	\[
    	A^\top S_{k} S_{k}^\top (Ax^{k} - b) = \sum_{i=1}^{k-j_{k}} \lambda_i (x^{i+j_{k}-1} - x^k).
    	\]
    	This leads to
    	\[
    	\| S_{k}^\top (Ax^k - b) \|_2^2 = \langle x^k - A^\dagger b, A^\top S_{k} S_{k}^\top (Ax^{k} - b) \rangle = \sum_{i=1}^{k-j_{k}} \lambda_i \langle x^k - A^\dagger b, x^{i+j_{k}-1} - x^k \rangle = 0,
    	\]
    	where the last equality follows from \(\langle x^k - A^\dagger b, x^{i+j_{k}-1} - x^k \rangle = 0\) for \(1 \leq i \leq k-j_{k}\). This contradicts the assumption that \(S_k^\top (Ax^k - b) \neq 0\). Therefore, \(M_k\) is of full column rank. This completes the proof of the proposition.
    \end{proof}
    
     	We present our randomized iterative methods with  affine subspace search in Algorithm \ref{Algo-1}. It can be seen that the requirement \(S_{k}^\top(Ax^k-b) \neq 0\) not only guarantees that \(M_k\) maintains full column rank, but also prevents \emph{null steps}.  When \(S_{k}^\top(Ax^k-b) = 0\), we have \(s_k = 0\), resulting in \(x^{k+1} = x^k\), which is a null step. Conversely, if \(S_{k}^\top(Ax^k-b) \neq 0\), then \(M_ks_k \neq 0\), leading to \(x^{k+1} \neq x^k\). Furthermore, Lemma \ref{xie-empty} implies that if \(x^k\) is not a solution to the linear system, there exists \(S_k \in \Omega\) such that \(S_k^\top(Ax^k-b) \neq 0\), indicating that our requirement on the iterative sketching matrices is well-defined. In addition, solving the linear system \eqref{eq-s_k} at each iteration can be computationally expensive. Strategies for efficiently solving this system will be discussed in Subsection \ref{subsection-eu}.

     	
     	  \begin{algorithm}[htpb]
     		\caption{Randomized iterative methods with affine subspace search (RIM-AS)}
     		\label{Algo-1}
     		\begin{algorithmic}
     			\Require $A\in\mathbb{R}^{m\times n}$, $b\in\mathbb{R}^m$, probability space ${(\Omega,\mathcal{F},\mathbf{P})}$, positive integer $\ell$, initial point $x^0\in\text{Range}(A^\top)$ and set $k=0$.
     			\begin{enumerate}
     				\item [1:] Randomly select an iterative sketching matrix $S_{k}\in\Omega$ until $S_{k}^\top(Ax^k-b)\neq0$.
     				\item [2:] Update  $j_k=\max\{k-\ell+1,0\}$ and 
     				$$M_{k}=\left(x^{j_{k}}-x^k,\cdots,x^{k-1}-x^k,-A^\top S_k S^\top_k (Ax^k - b)\right)\in\mathbb{R}^{n\times(k-j_{k}+1)}.$$
     				\item [3:] Compute $\gamma_k=\lVert S_{k}^\top (Ax^k-b)\rVert_2^2$, and find \( s_k \) as the solution to \eqref{eq-s_k}, i.e.,
     				 $$	M_k^\top M_ks_k=\gamma_ke_{k-j_k+1}^{k-j_k+1}.$$
     				\item [4:] Update $x^{k+1}=x^k+M_ks_k$.
     				\item [5:] If the stopping rule is satisfied, stop and go to output. Otherwise, set $k=k+1$ and return to Step 1.
     			\end{enumerate}
     			\Ensure
     			The approximate solution $x^k$.
     		\end{algorithmic}
     	\end{algorithm}

\begin{remark}\label{remark-xie-0325-2}
	When \(\ell=1\), the linear system \eqref{eq-s_k} simplifies to \(\lVert A^\top S_k S_{k}^\top (Ax^k-b)\rVert_2^2 s_k = \lVert S_{k}^\top (Ax^k-b)\rVert_2^2\), leading to
	\[
	s_k = \frac{\lVert S_{k}^\top (Ax^k-b)\rVert_2^2}{\lVert A^\top S_k S_{k}^\top (Ax^k-b)\rVert_2^2}.
	\]
	Consequently, Algorithm \ref{Algo-1} updates \(x^{k+1}\) as
	\begin{equation}
		\label{ell1-iteration}
		x^{k+1} = x^k - \frac{\lVert S_{k}^\top (Ax^k-b)\rVert_2^2}{\lVert A^\top S_k S_{k}^\top (Ax^k-b)\rVert_2^2} A^\top S_k S_{k}^\top (Ax^k-b).
	\end{equation}
	This is indeed the randomized iterative algorithmic framework with stochastic Polyak step-sizes for solving linear systems \cite{zeng2024adaptive,xie2024randomized,necoara2019faster}. A typical case of the iteration scheme \eqref{ell1-iteration} is the RK method \cite{strohmer2009randomized}. Suppose the sampling spaces are \( \Omega = \left\{e_i\right\}_{i=1}^m \), where \(e_i\) is sampled with probability \(\frac{\|A_{i, :}\|^2_2}{\|A\|^2_F}\) at the \(k\)-th iteration. We then obtain the following iteration scheme:
	\[
	x^{k+1} = x^k - \frac{A_{i_k, :}x^k-b_{i_k}}{\|A_{i_k, :}\|^2_2} A_{i_k, :}^\top,
	\]
	which is the RK method. For further discussion on the general applicability of \eqref{ell1-iteration}, see \cite{zeng2024adaptive,xie2024randomized}.
\end{remark}
   
Finally, we note that Algorithm 1 utilizes a fixed probability space \((\Omega, \mathcal{F}, \mathbf{P})\). Alternatively, a class of probability spaces \(\{(\Omega_k, \mathcal{F}_k, \mathbf{P}_k)\}_{k \geq 0}\) can be adopted to generate the iterative sketching matrix \(S_k\) at each iteration \cite{zeng2024adaptive}. This approach can enable the design of more versatile hybrid algorithms, offering enhanced performance, faster convergence, and improved scalability.

      \subsection{Convergence analysis}
      
      	We first introduce some auxiliary variables.
       We define 
      $$
      V_0:=\emptyset \ \ \text{and} \ \
      V_{k}:=\left(x^{j_{k}}-x^k,\cdots,x^{k-1}-x^k\right)\in\mathbb{R}^{k-j_{k}} \ \ \text{if} \ \ k\geq1,
      $$ 
       From Proposition \ref{prob-full}, we know that $V_k^\top V_k(k\geq1)$ and $M_k^\top M_k$ are positive definite matrices. 
     We split the unique solution $s_k$ of the linear system \eqref{eq-s_k} into the vector $\overline{s_k}\in\mathbb{R}^{k-j_k}$ of the first components and the last component $\underline{s_k}\in\mathbb{R}$. 
     Let $$d_k:=-A^\top S_k S_{k}^\top (Ax^k-b),$$ utilizing Cramer's rule, we can express $\underline{s_k}$ as
      \begin{equation*}
      	\begin{aligned}
      		\underline{s_k}=\frac{\text{det}\begin{pmatrix}
      				V_k^\top V_k&0\\
      				d_k^\top V_k&\gamma_k
      		\end{pmatrix}}{\text{det}(M_k^\top M_k)}=\gamma_k\frac{\text{det}(V_k^\top V_k)}{\text{det}(M_k^\top M_k)},
      	\end{aligned}
      \end{equation*}
      where $\gamma_k=\lVert S_{k}^\top (Ax^k-b)\rVert_2^2$.
      Applying the Schur complement formula, we have
      \begin{equation*}
      	\begin{aligned}
      		\text{det}(M_k^\top M_k)&=\text{det}\begin{pmatrix}
      			V_k^\top V_k& V_k^\top d_k\\
      			d_k^\top V_k&\lVert d_k\rVert_2^2
      		\end{pmatrix}=\text{det}(V_k^\top V_k)(\lVert d_k\rVert_2^2-d_k^\top V_k(V_k^\top V_k)^{-1}V_k^\top d_k).\\
      	\end{aligned}
      \end{equation*}
      Thus, we obtain 
      \begin{equation}
      	\label{def-sk}
      	\underline{s_k}=\frac{\gamma_k}{\lVert d_k\rVert_2^2-d_k^\top V_k(V_k^\top V_k)^{-1} V_k^\top d_k}.
      \end{equation} 
      We define 
    \begin{equation}\label{eq-q_k}
     \mathcal{Q}_k:=\{S\in\Omega:S^\top(Ax^k-b)\neq0\} \ \
    \text{and} \ \	\begin{aligned}
    		q_k:=\inf_{S_{k}\in\mathcal{Q}_k}\left\{\left(1-\frac{\lVert V_kV_k^\dagger d_k\rVert_2^2}{\lVert d_k\rVert_2^2}\right)^{-1}\right\}.
    	\end{aligned}
    \end{equation}
     Given that $M_k$ has full column rank, we know  $d_k\notin\text{Range}(V_k)$. Since $V_kV_k^\dagger$ is the
     orthogonal projection onto $\text{Range}(V_k)$, we can get
    $$
    0\leq\frac{\lVert V_kV_k^\dagger d_k\rVert_2^2}{\lVert d_k\rVert_2^2}<1,
    $$ 
    which implies that $q_k\geq 1$. We note that for certain probability space $(\Omega, \mathcal{F}, \mathbf{P})$, the parameter $q_k$  can  be strictly greater than $1$ (see Remark \ref{remark-strictly}).
    The convergence result for Algorithm \ref{Algo-1} is as follows.
    \begin{theorem}\label{Thm-Convergence-SCRIM-AS}
    	Suppose that the linear system $Ax=b$ is consistent and the probability space $(\Omega,\mathcal{F},\mathbf{P})$ satisfies Assumption \ref{Assumption1}.  Let $H=\mathbb{E}\left[\frac{SS^\top}{\|S^\top A\|^2_2}\right]$ and  $\{x^k\}_{k\geq0}$ be the iteration sequence generated by Algorithm {\rm\ref{Algo-1}}. Then
    	\begin{equation*}
    		\begin{aligned}
    			\mathop{\mathbb{E}}\limits_{S_{k}\in\Omega}[\lVert x^{k+1}-A^\dagger b\rVert_2^2 \mid x^k]\leq(1-q_k\sigma_{\min}^2(H^{\frac{1}{2}}A))\lVert x^k-A^\dagger b\rVert_2^2,
    		\end{aligned}
    	\end{equation*}
    	where $q_k$ are given by \eqref{eq-q_k}.
    \end{theorem}

    \begin{proof}
    	By the iteration scheme  of $x^{k+1}$, we have
    	\begin{equation*}
    		\begin{aligned}
    			\lVert x^{k+1}-A^\dagger b\rVert_2^2&=\lVert x^k+M_ks_k-A^\dagger b\rVert_2^2\\
    			&=\lVert x^k-A^\dagger b\rVert_2^2+2s_k^\top M_k^\top (x^k-A^\dagger b)+s_k^\top M_k^\top M_ks_k\\
    			&=\lVert x^k-A^\dagger b\rVert_2^2-s_k^\top M_k^\top M_ks_k
    			\\
    			&=\lVert x^k-A^\dagger b\rVert_2^2-\gamma_k\underline{s_k},
    		\end{aligned}
    	\end{equation*}
    	where the third equality follows from \eqref{xie-5-26-1} and the last equality follows from \eqref{eq-s_k},  $\gamma_k=\lVert S_{k}^\top (Ax^k-b)\rVert_2^2$, and $\underline{s_k}$ is defined as \eqref{def-sk}. For convenience, we use $\mathbb{E}[\hspace{2pt}\cdot\hspace{2pt}]$ to denote $\mathbb{E}[\hspace{2pt}\cdot\hspace{2pt}|x^k]$. Taking the condition expectation, we have
    	\begin{equation}\label{proof-xie-03-25-1}
    		\begin{aligned}
    			\mathop{\mathbb{E}}\limits_{S_{k}\in\Omega}[\lVert x^{k+1}-A^\dagger b\rVert_2^2]&=\mathop{\mathbb{E}}\limits_{S_{k}\in\mathcal{Q}_k}[\lVert x^{k+1}-A^\dagger b\rVert_2^2]\\
    			&=\lVert x^k-A^\dagger b\rVert_2^2-\mathop{\mathbb{E}}\limits_{S_{k}\in\mathcal{Q}_k}\left[\frac{\gamma_k^2}{\lVert d_k\rVert_2^2-d_k^\top V_k(V_k^\top V_k)^{-1} V_k^\top d_k}\right]\\
    			&=\lVert x^k-A^\dagger b\rVert_2^2-\mathop{\mathbb{E}}\limits_{S_{k}\in\mathcal{Q}_k}\left[\frac{\gamma_k^2}{\lVert d_k\rVert_2^2}\left(1-\frac{\lVert V_kV_k^\dagger d_k\rVert_2^2}{\lVert d_k\rVert_2^2}\right)^{-1}\right]\\
    			&=\lVert x^k-A^\dagger b\rVert_2^2-\mathop{\mathbb{E}}\limits_{S_{k}\in\mathcal{Q}_k}\left[\frac{\lVert S_{k}^\top (Ax^k-b)\rVert_2^4}{\lVert A^\top S_{k} S_{k}^\top (Ax^k-b)\rVert_2^2}(1-\cos^2 \theta_k)^{-1}\right]\\
    			&\leq\lVert x^k-A^\dagger b\rVert_2^2-\mathop{\mathbb{E}}\limits_{S_{k}\in\mathcal{Q}_k}\left[\frac{\lVert S_{k}^\top (Ax^k-b)\rVert_2^2}{\lVert A^\top S_{k}\rVert_2^2}(1-\cos^2 \theta_k)^{-1}\right]\\
    			&\leq\lVert x^k-A^\dagger b\rVert_2^2-q_k\mathop{\mathbb{E}}\limits_{S_{k}\in\mathcal{Q}_k}\left[\frac{\lVert S_{k}^\top (Ax^k-b)\rVert_2^2}{\lVert A^\top S_{k}\rVert_2^2}\right]\\
    			&\leq\lVert x^k-A^\dagger b\rVert_2^2-q_k\mathbb{P}(S_k\in\mathcal{Q}_k)\mathop{\mathbb{E}}\limits_{S_{k}\in\mathcal{Q}_k}\left[\frac{\lVert S_{k}^\top (Ax^k-b)\rVert_2^2}{\lVert A^\top S_{k}\rVert_2^2}\right].
    		\end{aligned}
    	\end{equation}
    	Since
    	\begin{equation*}
    		\begin{aligned}
    			\mathop{\mathbb{E}}\limits_{S_{k}\in\Omega}\left[\frac{\lVert S_{k}^\top (Ax^k-b)\rVert_2^2}{\lVert S_{k}^\top A\rVert_2^2}\right]=&\mathbb{P}(S_k\in\mathcal{Q}_k)\mathop{\mathbb{E}}\limits_{S_{k}\in\mathcal{Q}_k}\left[\frac{\lVert S_{k}^\top (Ax^k-b)\rVert_2^2}{\lVert S_{k}^\top A\rVert_2^2}\right]\\
    			&+\mathbb{P}(S_k\in\mathcal{Q}_k^c)\mathop{\mathbb{E}}\limits_{S_{k}\in\mathcal{Q}_k^c}\left[\frac{\lVert S_{k}^\top (Ax^k-b)\rVert_2^2}{\lVert S_{k}^\top A\rVert_2^2}\right],
    		\end{aligned}
    	\end{equation*}
    	and 
    	$$\mathop{\mathbb{E}}\limits_{S_{k}\in\mathcal{Q}_k^c}\left[\frac{\lVert S_{k}^\top (Ax^k-b)\rVert_2^2}{\lVert S_{k}^\top A\rVert_2^2}\right]=0$$
    	as $S_{k}^\top (Ax^k-b)=0$ for $S_k\in\mathcal{Q}_k^c$.
    	We know that
    	$$
    	\mathbb{P}(S_k\in\mathcal{Q}_k)\mathop{\mathbb{E}}\limits_{S_{k}\in\mathcal{Q}_k}\left[\frac{\lVert S_{k}^\top (Ax^k-b)\rVert_2^2}{\lVert A^\top S_{k}\rVert_2^2}\right]=\mathop{\mathbb{E}}\limits_{S_{k}\in\Omega}\left[\frac{\lVert S_{k}^\top (Ax^k-b)\rVert_2^2}{\lVert A^\top S_{k}\rVert_2^2}\right].
    	$$
    	Substituting it into \eqref{proof-xie-03-25-1}, we can get
    	\[
    	\begin{aligned}
    		\mathop{\mathbb{E}}\limits_{S_{k}\in\Omega}[\lVert x^{k+1}-A^\dagger b\rVert_2^2]&=\lVert x^k-A^\dagger b\rVert_2^2-q_k\mathop{\mathbb{E}}\limits_{S_{k}\in\Omega}\left[\frac{\lVert S_{k}^\top (Ax^k-b)\rVert_2^2}{\lVert A^\top S_{k}\rVert_2^2}\right]\\
    		&\leq(1-q_k\sigma_{\min}^2(H^{\frac{1}{2}}A))\lVert x^k-A^\dagger b\rVert_2^2,
    	\end{aligned}
    	\]
    	where the last inequality follows because \(x^k - A^{\dag}b \in \text{Range}(A^\top)\) and Lemma \ref{positive} implies that \(H\) is positive definite. This completes the proof of the theorem.
    \end{proof}

  Let us provide some remarks on the convergence results. The first result indicates that we can select any \(x^0 \in \mathbb{R}^n\) as the initial vector.
  
  \begin{remark}
  	One can choose any \(x^0 \in \mathbb{R}^n\) as the initial vector. In this case, it can be shown that
  	\[
  	\mathop{\mathbb{E}}_{S_{k} \in \Omega}[\lVert x^{k+1} - x^0_{*} \rVert_2^2 \mid x^k] \leq (1 - q_k \sigma_{\min}^2(H^{\frac{1}{2}}A)) \lVert x^k - x^0_{*} \rVert_2^2,
  	\]
  	where \(x^0_* := A^\dagger b + (I - A^\dagger A)x^0\). Here, \(x^0_*\) is actually the orthogonal projection of \(x^0\) onto the set \(\{ x \in \mathbb{R}^n : Ax = b\}\). 
  We refer the reader to \cite{han2024randomized, han2022pseudoinverse} for more details.
  \end{remark}
    
   The following remarks demonstrate that incorporating more previous iterations can enhance the performance of Algorithm \ref{Algo-1}.
   
    \begin{remark}
    	Consider the case where \(\tilde{\ell}\) previous iteration points are used, with \(\tilde{\ell} < \ell\). Define
    	\[
    	\tilde{V}_k = \begin{pmatrix}
    		x^{\tilde{j}_k} - x^k, \cdots, x^{k-1} - x^k
    	\end{pmatrix},
    	\]
    	where \(\tilde{j}_{k} = \max\{k - \tilde{\ell} + 1, 0\}\). Since \(\tilde{\ell} < \ell\), it follows that \(\operatorname{Range}(\tilde{V}_k) \subseteq \operatorname{Range}(V_k)\). This implies
    	\[
    	\frac{\lVert \tilde{V}_k \tilde{V}_k^\dagger d_k \rVert_2^2}{\lVert d_k \rVert_2^2} \leq \frac{\lVert V_k V_k^\dagger d_k \rVert_2^2}{\lVert d_k \rVert_2^2}.
    	\]
    	Therefore,
    	\[
    	\tilde{q}_k := \inf_{S_{k} \in \mathcal{Q}_k} \left\{ \left(1 - \frac{\lVert \tilde{V}_k \tilde{V}_k^\dagger d_k \rVert_2^2}{\lVert d_k \rVert_2^2}\right)^{-1} \right\} \leq q_k,
    	\]
    	which implies
    	\[
    	1 - \tilde{q}_k \sigma_{\min}^2(M^{\frac{1}{2}}A) \geq 1 - q_k \sigma_{\min}^2(M^{\frac{1}{2}}A).
    	\]
    	Hence, using more previous iterations can enhance the performance of Algorithm \ref{Algo-1}.
    	 \end{remark}
    	
    	\begin{remark}
    When \(\ell \geq r:= \operatorname{rank}(A)\), Algorithm \ref{Algo-1} can exhibit finite-time convergence. Indeed, we know that now
    \[ 
    M_{r-1}=\left(x^{0}-x^{r-1},\cdots,x^{r-2}-x^{r-1},-A^\top S_{r-1} S^\top_{r-1} (Ax^{r-1} - b)\right).
    \]
    For any \(i = 0, \ldots, r-2\), we have \(x^{i} - x^{r-1} \in \operatorname{Range}(A^\top)\), which, together with \(A^\top S_{r-1} S^\top_{r-1} (Ax^{r-1} - b) \in \operatorname{Range}(A^\top)\), implies that \(\operatorname{span}(M_{r-1})\subseteq \operatorname{Range}(A^\top)\). Since \(M_{r-1}\) has full column rank, we know that
    \[
    \operatorname{Range}(M_{r-1}) = \operatorname{Range}(A^\top).
    \]
    Because \(A^\dagger b \in \operatorname{Range}(A^\top)\), it follows that \(A^\dagger b \in \operatorname{Range}(M_{r-1})\). Therefore, by our design, the next iterate satisfies \(x^{r} = x^{r-1} + M_{r-1} s_{r-1} = A^\dagger b\). This implies that Algorithm \ref{Algo-1} terminates after exactly \(r\) iterations in the absence of round-off error.
    \end{remark}
    
     Finally, we note that while our method can theoretically terminate in a finite number of steps, it is more practical to use a fixed-memory strategy by setting \(\ell\) to a constant value. This approach aligns with the philosophy of limited-memory methods, optimizing the trade-off between memory efficiency, computational speed, and numerical accuracy \cite{nocedal1999numerical}.

    \subsection{Efficient implementation}
    \label{subsection-eu}

    In this subsection, we introduce an efficient update strategy for solving the linear system \eqref{eq-s_k} and subsequently present an efficient implementation of Algorithm \ref{Algo-1}.
   Given that 
   \[
   M^\top_kM_k = \begin{pmatrix}
   	V_k^\top V_k & V_k^\top d_k \\
   	d_k^\top V_k & \lVert d_k \rVert_2^2
   \end{pmatrix},
   \]
   where \(d_k = -A^\top S_k S^\top_k(Ax^k-b)\), the unique solution to the linear system \eqref{eq-s_k}, according to Lemma \ref{lemma-inv-sol}, can be expressed as
   \begin{equation}\label{definitionsk}
   s_k = \begin{pmatrix}
   	\overline{s_k} \\
   	\underline{s_k}
   \end{pmatrix} = \frac{\gamma_k}{d_k^\top V_k (V_k^\top V_k)^{-1} V_k^\top d_k - \lVert d_k \rVert_2^2} \begin{pmatrix}
   	(V_k^\top V_k)^{-1} V_k^\top d_k \\
   	-1
   \end{pmatrix}.
   \end{equation}
  This implies that we need only consider the inverse of the matrix \(V_k^\top V_k\), rather than \(M_k^\top M_k\). The following lemma provides a simple formula for the inverse of \(V_k^\top V_k\). We note that the proof of this result, as well as the proofs of other results with lengthy derivations, are all moved to Appendix \ref{sec:appd}.  For convenience, we define the matrix \(C(\alpha_1, \ldots, \alpha_n)\) for any \(\alpha_1, \ldots, \alpha_n \in \mathbb{R}\) with $\alpha_i\neq 0$ for all $i\in[n]$ as follows
  	{\small
  	\begin{equation}\label{matrixC}
  		\begin{aligned}
  			C(\alpha_1,\ldots,\alpha_n):=
  			\begin{pmatrix}
  				\alpha_1^{-1}&-\alpha_1^{-1}&&&&\\
  				-\alpha_1^{-1}&\alpha_1^{-1}+\alpha_2^{-1}&-\alpha_2^{-1}&&&\\
  				&-\alpha_2^{-1}&\alpha_2^{-1}+\alpha_3^{-1}&\cdots&&\\
  				&&\ddots&\ddots&\ddots&\\
  				&&&\cdots&\alpha_{n-2}^{-1}+\alpha_{n-1}^{-1}&-\alpha_{n-1}^{-1}\\
  				&&&&-\alpha_{n-1}^{-1}&\alpha_{n-1}^{-1}+\alpha_n^{-1}\\
  			\end{pmatrix}
  		\end{aligned}
  	\end{equation}
  }

  \begin{lemma}\label{lemma-effup}
  Suppose that \(\{s_k\}_{k \geq 0}\) and \(\{\gamma_k\}_{k \geq 0}\) are the sequences generated by Algorithm \ref{Algo-1}. Then the inverse matrix of \(V_k^\top V_k\) can be expressed as
  \[
  C_k = C(\gamma_{j_k}\underline{s_{j_k}}, \ldots, \gamma_{k-1}\underline{s_{k-1}}) \in \mathbb{R}^{(k-j_k) \times (k-j_k)},
  \]
  where the matrix \(C(\gamma_{j_k}\underline{s_{j_k}}, \ldots, \gamma_{k-1}\underline{s_{k-1}})\) is defined by \eqref{matrixC}.
  \end{lemma}
    
   From \eqref{definitionsk}, we know that
   \[
   \overline{s_k} = -\underline{s_k}(V_k^\top V_k)^{-1}V_k^\top d_k = -\underline{s_k}C_kV_k^\top d_k,
   \]
   where the last equality follows from Lemma \ref{lemma-effup}. Therefore, we have
   \begin{equation}\label{xk1-scheme}
   x^{k+1} = x^k + M_ks_k = x^k + V_k\overline{s_k} + \underline{s_k}d_k = x^k - \underline{s_k}V_kC_kV_k^\top d_k + \underline{s_k}d_k.
   \end{equation}
   Now, we have already presented an efficient implementation of Algorithm \ref{Algo-1}, as described in Algorithm \ref{Algo-2}.

    \begin{algorithm}[htpb]
    	\caption{Efficient RIM-AS}
    	\label{Algo-2}
    	\begin{algorithmic}
    		\Require $A\in\mathbb{R}^{m\times n},b\in\mathbb{R}^m$, probability space $(\Omega,\mathcal{F},\mathbf{P})$, positive integer $\ell$, initial point $x^0\in\text{Range}(A^\top)$, $V_0=C_0=\emptyset$, and set $k=0$.
    		\begin{enumerate}
    			\item [1:] Randomly select an iterative sketching matrix $S_{k}\in\Omega$ until $S_{k}^\top(Ax^k-b)\neq0$.
    			\item [2:] Update  $j_k=\max\{k-\ell+1,0\}$ and 
    			$$V_{k}=\left(x^{j_{k}}-x^k,\cdots,x^{k-1}-x^k\right)\in\mathbb{R}^{n\times (k-j_{k})}.$$
    			\item [3:] Construct the matrix $C_k=C(\gamma_{j_k}\underline{s_{j_k}},\cdots,\gamma_{k-1}\underline{s_{k-1}})\in\mathbb{R}^{{(k-j_k)}\times{(k-j_k)}}$ as defined in \eqref{matrixC}.
    			\item [4:] Compute 
    			$$\begin{cases}
    				\gamma_k=\lVert S_{k}^\top (Ax^k-b)\rVert_2^2,\\
    				d_k=-A^\top S_{k}S_{k}^\top(Ax^k-b),\\
    				c_k=\lVert d_k\rVert_2^2,\\
    				w_k=V_k^\top d_k,\\
    				h_k=C_kw_k,\\
    				\underline{s_k}=\frac{\gamma_k}{c_k-w_k^\top h_k}.
    			\end{cases}
    			$$
    			\item [5:] Update $x^{k+1}=x^k-\underline{s_k}(V_kh_k-d_k)$.
    			\item [6:] If the stopping rule is satisfied, stop and go to output. Otherwise, set $k=k+1$ and return to Step 1.
    		\end{enumerate}
    		\Ensure
    		The approximate solution $x^k$.
    	\end{algorithmic}
    \end{algorithm}
    
Finally, we provide a detailed discussion of the computational complexity of Algorithm \ref{Algo-2}. For convenience, we assume that the randomized sketching matrix $S\in\mathbb{R}^{m\times q}$, and we use $\mathcal{W}(S_k^\top A)$ and $\mathcal{W}(S_k^\top b)$ to denote the workload of computing $S_k^\top A$ and $S_k^\top b$, respectively. The computational complexity of the $k$-th iteration of Algorithm \ref{Algo-2} consists of the following components:
\rm(\romannumeral1) Updating $V_k$ requires $(k-j_k)n$ flops, assuming the vectors $x^{j_k},\ldots,x^{k-1}$ are stored;
  		(\romannumeral2) Computing $\gamma_k$ requires $\mathcal{W}(S_k^\top A)+\mathcal{W}(S_k^\top b)+2qn+2q-1$ flops, where the computation of $S_k^\top A$ and $S_k^\top b$ are performed first;
  		 (\romannumeral3) Computing $d_k$ requires $(2q-1)n$ flops, based on the precomputed residual $S_k^\top(b-Ax^k)$ from step (\romannumeral2);
		 (\romannumeral4) Computing $c_k,w_k,h_k$ and $\underline{s_k}$ requires a total of $2(k-j_k+1)n+2(k-j_k)^2$ flops;
  		 (\romannumeral5) Updating $x^{k+1}$ requires $2(k-j_k+1)n$ flops.
  	Hence, the total computational complexity of  Algorithm \ref{Algo-2} is
\begin{equation}
\label{cc-alg2}
\mathcal{W}(S_k^\top A)+\mathcal{W}(S_k^\top b)+(4q+5(k-j_k)+3)n+2q+2(k-j_k)^2-1.
\end{equation}
In particular, when $j_k=0$, i.e. $k\leq \ell-1$, the algorithm has the following computational complexity
$$
\mathcal{W}(S_k^\top A)+\mathcal{W}(S_k^\top b)+(4q+5k+3)n+2q+2k^2-1.
$$
When $j_k=k-\ell+1>0$, the algorithm has the following computational complexity
$$
\mathcal{W}(S_k^\top A)+\mathcal{W}(S_k^\top b)+(4q+5\ell-2)n+2q+2(\ell-1)^2-1.
$$


    \section{Iterative-sketching-based Krylov subspace methods}\label{sec-Krylov}
    
    In this section, we first investigate alternative representations of the affine subspace \(\Pi_k\). By analyzing these forms, we establish that the Krylov subspace is actually a  special case of \(\Pi_k\). Building on this insight, we then propose a new framework for iterative-sketching-based Krylov (IS-Krylov) subspace methods to solve linear systems, which allows for the derivation of various IS-Krylov subspace methods by selecting different probability spaces \((\Omega, \mathcal{F}, \mathbf{P})\).
    We note that all proofs in this section are collected in Appendix \ref{sec:appd}.
    
    \subsection{Analysis of the affine subspace}
    In this subsection, we establish two lemmas that are useful for developing our IS-Krylov subspace methods. 
    The first lemma provides equivalent  forms of \(\Pi_k\). 

    \begin{lemma}\label{Thm-Pi_k-Krylov}
    Let  $\{x^k\}_{k\geq0}$ be the iteration sequence generated by Algorithm \ref{Algo-1}, and define $r^k := Ax^k - b$. Then the affine subspaces
    $$
    \Pi_k = \text{\rm aff}\{x^{j_k},x^{j_k+1},\cdots,x^k,x^k-A^\top S_k S_k^\top r^k\}
    $$
    and
    $$
    \Theta_k := \text{\rm aff}\{x^{j_k},x^{j_k+1},\cdots,x^k,x^{k+1}\}
    $$
    are identical. Moreover, when $\ell=\infty$, both affine subspaces coincide with the affine subspace
    \[
    \Lambda_k :=x^0+ \operatorname{span}\{A^\top S_0S_0^\top r^0,\cdots,A^\top S_kS_k^\top r^k\}.
    \]
    \end{lemma}

    Recall that, given a matrix $B\in\mathbb{R}^{n\times n}$ and a vector $r\in\mathbb{R}^n$, the Krylov subspace  of order $k$ is defined as \cite[Section 10.1.1]{golub2013matrix} 
   \[
   \mathcal{K}_k(B, r) = \text{span}\{r, Br, B^2r, \ldots, B^{k-1}r\}.
   \] 
   The following lemma  demonstrates that \(x^0 + \mathcal{K}_{k+1}(A^\top A, A^\top r^0)\) is a specific instance of \(\Pi_k\).
    
    \begin{lemma}
    	\label{xie-equ-krylov}
    	Suppose that $\{x^k\}_{k\geq0}$ is the iteration sequence generated by Algorithm \ref{Algo-1} with setting $\ell=\infty$ and $\Omega=\{I\}$. Let $r^k = Ax^k - b$. Then
    the subspaces  $\text{span}\{A^\top r^0,\cdots,A^\top r^k\}$ and $\mathcal{K}_{k+1}(A^\top A,A^\top r^0)$ are identical.
    \end{lemma}

    \subsection{The IS-Krylov subspace method }\label{section-ISK}
    
    From Lemmas \ref{Thm-Pi_k-Krylov} and \ref{xie-equ-krylov}, we can observe that when \(\ell=\infty\) and \(\Omega=\{I\}\), our RIM-AS framework \eqref{RIM_subspace} reduces to a Krylov subspace method. Specifically, the iteration scheme \eqref{RIM_subspace}   simplifies to
    \begin{equation}\label{RIM_subspace-eq}
    	\begin{aligned}
    		&x^{k+1} = \mathop{\arg\min}_{x \in \mathbb{R}^n} \lVert x - A^\dagger b \rVert_2^2 \\
    		 \text{subject to} &\ \ x \in \Pi_{k}=\Lambda_k= x^0+\mathcal{K}_{k+1}\left(A^\top A,A^\top \left(Ax^0-b\right)\right).
    	\end{aligned}
    \end{equation}
    This suggests that for general choices of \(\ell\) and probability spaces \((\Omega, \mathcal{F}, \mathbf{P})\), our RIM-AS approach \eqref{RIM_subspace} can be extended to construct a novel iterative-sketching-based Krylov (IS-Krylov) subspace method. 
    
    We present our IS-Krylov subspace method in Algorithm \ref{Algo-3}. This algorithm can be regarded as an \(\ell\)-truncated-type algorithm, where \(\ell\) is the truncated parameter. Moreover, the iteration scheme \eqref{iter-krylov} in Algorithm \ref{Algo-3} functions as a Gram-Schmidt orthogonalization process applied to the vector set
    $
    \left\{x^{j_k} - x^k, \cdots, x^{k-1} - x^k, -A^\top S_k S^\top_k (Ax^k - b)\right\},
    $
    as detailed in Remark \ref{remark-ob0329}. In fact, Algorithm \ref{Algo-3} can be interpreted as an another efficient update strategy variant of Algorithm \ref{Algo-1}. In this form, we can further investigate some properties of the algorithm and establish connections between it and the conjugate gradient method.

    

     \begin{algorithm}[htpb]
    	\caption{The iterative-sketching-based Krylov (IS-Krylov) subspace method}
    	\label{Algo-3}
    	\begin{algorithmic}
    		\Require $A\in\mathbb{R}^{m\times n},b\in\mathbb{R}^m$, probability space $(\Omega,\mathcal{F},\mathbf{P})$, positive integer $\ell$, initial point $x^0\in\text{Range}(A^\top)$, and set $k=0$.
    		\begin{enumerate}
    			\item[1:] Randomly select an iterative sketching matrix $S_0 \in \Omega$ until $S_0^\top (Ax^0-b)\neq 0$.
    			\item[2:] Set $p_0=d_0=- A^\top S_0 S_0^\top (Ax^0-b)$.
    			\item[3:] Set $\delta_k =\|S_k^\top (Ax^k-b)\|_2^2 / \|p_k\|_2^2$.
    			\item[4:] Update $x^{k+1}=x^k+\delta_k p_k$.
    			\item[5:] Randomly select an iterative sketching $S_{k+1} \in \Omega$ until $S_{k+1}^\top (Ax^{k+1}-b)\neq 0$.
    			\item [6:] Update  $j_{k+1}=\max\{k-\ell+2,0\}$ and compute
    			\begin{equation}\label{iter-krylov}
    				\left\{
    				\begin{array}{l}
    					d_{k+1}=-A^\top S_{k+1}S_{k+1}^\top (Ax^{k+1}-b),\\
    					\eta_{k+1}^i=\langle d_{k+1},p_i\rangle/\lVert p_i\rVert_2^2, \ i=j_{k+1},\ldots,k,\\
    					p_{k+1}=d_{k+1}-\sum_{i=j_{k+1}}^{k}\eta_{k+1}^ip_i.
    				\end{array}\right.
    			\end{equation}
    			\item [7:] If the stopping rule is satisfied, stop and go to output. Otherwise, set $k=k+1$ and return to Step 3.
    		\end{enumerate}
    		\Ensure
    		The approximate solution $x^k$.
    	\end{algorithmic}
    \end{algorithm}


    The following theorem demonstrates the equivalence between Algorithms \ref{Algo-1} and \ref{Algo-3}.

    \begin{theorem}\label{thm-eq-0330}
    	If Algorithms \ref{Algo-1} and \ref{Algo-3} share the same parameters \(\{S_k\}_{k\geq0}\) and initial point \(x^0\), their generated sequences \(\{x^k\}_{k\geq0}\) coincide exactly.
    \end{theorem}

One may wonder why we present two equivalent implementations of Algorithm \ref{Algo-1} (Algorithms \ref{Algo-2} and \ref{Algo-3}) rather than just Algorithm \ref{Algo-3}. 
  This is because the derivation of Algorithm \ref{Algo-3} relies on Algorithm \ref{Algo-2}; specifically, the proof of Theorem \ref{thm-eq-0330} requires conclusions established from Algorithm \ref{Algo-2}. We will demonstrate that, while both algorithms share the same computational complexity, Algorithm \ref{Algo-3} is more efficient in practice, as it avoids matrix column extraction and concatenation operations.


  \subsubsection{Implementation comparison between Algorithms \ref{Algo-2} and \ref{Algo-3}}
  \label{subsection-com-2-3}
  
 First, we analyze the computational complexity of Algorithm \ref{Algo-3}. Similarly, we assume that the iterative sketching matrix $S\in\mathbb{R}^{m\times q}$. The computational complexity of the $k$-th iteration of Algorithm \ref{Algo-3} consists of the following components: (\romannumeral1) Computing $d_k$ requires $\mathcal{W}(S_k^\top A)+\mathcal{W}(S_k^\top b)+(4q-1)n$ flops; (\romannumeral2) Computing $\eta_k^{j_k}\ldots,\eta_k^{k-1}$ requires $2(k-j_k)n$ flops, as $\lVert p_{j_k}\rVert_2^2,\ldots,\lVert p_{k-1}\rVert_2^2$ are compute in advance; (\romannumeral3) Computing $p_k$ requires $2(k-j_k)n$ flops; (\romannumeral4) Computing $\delta_k$ requires $2n+2q-1$ flops; (\romannumeral5) Updating $x^{k+1}$ requires $2n$ flops. Hence, the total computational complexity of Algorithm \ref{Algo-3} is
  \begin{equation}\label{cc-alg3}
	\begin{aligned}
		\mathcal{W}(S_k^\top A)+\mathcal{W}(S_k^\top b)+(4q+4(k-j_k)+1)n+2q-1.
	\end{aligned}
	\end{equation}
%
It can be observed from \eqref{cc-alg2} and \eqref{cc-alg3} that Algorithms \ref{Algo-2} and \ref{Algo-3} have almost the same computational complexity. However, in practical computation, Algorithm \ref{Algo-3} can be more advantageous to implement than Algorithm \ref{Algo-2}.

We now provide a detailed comparison of the implementations between Algorithms \ref{Algo-2} and \ref{Algo-3}. For convenience, let \( V(:,i) \) denote the \( i \)-th column of the matrix \( V \), and \( V(:,[i:j]) \) denote the submatrix consisting of columns indexed by the set \(\{i,i+1,\ldots,j\}\). The matrix \( V_k \) in Algorithm \ref{Algo-2}  can be computed as:
\[
V_k = \overline{V_k}(:,[1:(k-j_k)]) - \overline{V_k}(:,k-j_k+1)(\mathbbm{1}_{k-j_k}^{k-j_k})^\top,
\]
where \(\overline{V}_k = \begin{pmatrix} x^{j_k}, \ldots, x^{k-1}, x^k \end{pmatrix} \in \mathbb{R}^{n \times (k-j_k+1)}\) and \(\mathbbm{1}^{k-j_k}_{k-j_k} \in \mathbb{R}^{k-j_k}\).
When \( j_k \geq 1 \), i.e., \( k \geq \ell \), the matrix \(\overline{V}_{k+1}\) is updated as:
\[
\overline{V}_{k+1} = \begin{pmatrix} x^{j_{k}+1}, \ldots, x^{k}, x^{k+1} \end{pmatrix} = \begin{pmatrix} \overline{V}_k(:,[2:\ell]), x^{k+1} \end{pmatrix} \in \mathbb{R}^{n \times \ell}.
\]
This indicates that when \( k \geq \ell \), Algorithm \ref{Algo-2} requires column extraction and concatenation at each iteration.

To efficiently implement Algorithm \ref{Algo-3}, during the initial steps where \( k = \ell-2 \), we construct the matrix
$
P_k = \begin{pmatrix}
\frac{p_0}{\lVert p_0 \rVert_2}, \ldots, \frac{p_k}{\lVert p_k \rVert_2}
\end{pmatrix}.
$
When \( k \geq \ell-1 \), the matrix \( P_k \in \mathbb{R}^{n \times (\ell-1)} \) is updated as follows
\[
P_k(:,i) = \begin{cases} 
\frac{p_k}{\lVert p_k \rVert_2}, & \text{if } i = i_k; \\
P_{k-1}(:,i), & \text{otherwise},
\end{cases}
\]
where \( i_k = \text{mod}(k+1, \ell) + 1 \). In fact, this update rule simply replaces the \( i_k \)-th column of \( P_{k-1} \) with \(\frac{p_k}{\lVert p_k \rVert_2}\) to obtain \( P_k \), eliminating the need for column extraction and concatenation operations. Now, \( p_{k+1} \) in equation (\ref{iter-krylov}) can be computed as
\[
p_{k+1} = d_{k+1} - \sum_{i=j_{k+1}}^{k} \frac{\langle d_{k+1}, p_i \rangle}{\lVert p_i \rVert_2^2} p_i = d_{k+1} - P_k P_k^\top d_{k+1}.
\]

    \subsection{Further properties }
    \label{subsection-FP}
    We have the following result for Algorithm \ref{Algo-3}.
    \begin{prop}\label{prop-property-Krylov}
    	Suppose that $\{x^k\}_{k\geq0}$ and $\{p_k\}_{k\geq0}$ are the sequences generated by Algorithm \ref{Algo-3}. Let $r^k=Ax^k-b$. Then for any  $j_k\leq t< v\leq k$, we have
    	\begin{itemize}
    		\item[(i)] $\langle p_v,p_t\rangle=0$;
    		\item[(ii)] $\langle S_t^\top r^{v},S_t^\top r^t\rangle=\left\{
    		\begin{array}{cc}
    			0,&\text{if }v=t+1\text{ or } j_k= j_t,\\
    			-\sum_{w=t}^{v-1}\sum_{i=j_{t}}^{j_k-1}\delta_w\eta_{t}^i\langle p_w,p_i\rangle,&\text{otherwise}.
    		\end{array}
    		\right..
    	$
    	\end{itemize}
    \end{prop}

     \begin{remark}\label{remark-strictly}
  Based on Proposition \ref{prop-property-Krylov}, we can demonstrate that the parameter \(q_k\) in Theorem \ref{Thm-Convergence-SCRIM-AS} can be strictly greater than $1$ in certain cases. For instance, if \(\Omega = \{S\}\), where \(S\) is a fixed matrix such that \(S^\top S\) is positive definite, then we have
  \[
  \begin{aligned}
  	(V_k^\top d_k)_{k-j_k} &= \langle x^{k-1} - x^k, d_k \rangle \\
  	&= \langle S^\top A(x^k - x^{k-1}), S^\top(Ax^k - b) \rangle \\
  	&= \lVert S^\top(Ax^k - b) \rVert_2^2 - \langle S^\top(Ax^{k-1} - b), S^\top(Ax^k - b) \rangle \\
  	&= \lVert S^\top(Ax^k - b) \rVert_2^2,
  \end{aligned}
  \]
  where the last equality follows from Proposition \ref{prop-property-Krylov}. This implies that \(V_k^\top d_k \neq 0\) as long as \(Ax^k \neq b\). Consequently, we have \(\lVert V_k V_k^\dagger d_k \rVert_2^2 >0 \) and thus,
  \[
  q_k = \left(1 - \frac{\lVert V_k V_k^\dagger d_k \rVert_2^2}{\lVert d_k \rVert_2^2}\right)^{-1} > 1.
  \] 	
    \end{remark}

    \begin{remark}
    	\label{remark-ob0329}
    From Proposition \ref{prop-property-Krylov}, it can be seen that \(p_{j_{k}}, \ldots, p_{k}\) are orthogonal. For any \(i = j_{k}, \ldots, k\), we have
    \[
    p_i = \frac{x^{i+1} - x^i}{\delta_i} = \frac{x^{i+1} - x^{k}}{\delta_i} - \frac{x^i - x^{k}}{\delta_i},
    \]
    which implies that \(p_{j_{k}}, \ldots, p_{k}\) form an orthogonal basis for the subspace \(\operatorname{span}\{x^{j_{k}} - x^{k}, \ldots, x^{k+1} - x^{k}\}\). Now, the iteration scheme \eqref{iter-krylov} implements a Gram-Schmidt orthogonalization \cite[Section 5.2.7]{golub2013matrix}   of \(d_{k+1}\) with respect to the existing basis \(\{p_{j_{k+1}}, \ldots, p_{k}\}\).
    \end{remark}
    
    If \(\ell=2\), then \(j_{k+1}=k\) for \(k \geq 0\). For any initial point \(x^0 \in \text{Range}(A^\top)\), let \(p_0 = -A^\top S_0 S_0^\top (Ax^0-b)\). In this case, Algorithm \ref{Algo-3} simplifies to
    	\begin{equation*}
    		\left\{
    		\begin{array}{l}
    			\delta_k=\lVert S_{k}^\top (Ax^{k}-b)\rVert_2^2/\lVert p_k\rVert_2^2,\\
    			x^{k+1}=x^k+\delta_kp_k,\\
    			d_{k+1}=-A^\top S_{k+1}S_{k+1}^\top (Ax^{k+1}-b),\\
    			\eta_{k+1}=\langle d_{k+1},p_k\rangle/\lVert p_k\rVert_2^2,\\
    			p_{k+1}=d_{k+1}-\eta_{k+1}p_k.
    		\end{array}\right..
    	\end{equation*}
    	This iteration scheme aligns with the stochastic conjugate gradient (SCG) method proposed in \cite{zeng2024adaptive}.
	If $\ell=\infty$ and $\Omega=\{I\}$, we have $j_k=0$. Then
	$$
	\begin{aligned}
		\eta_{k+1}^i&=\frac{\langle d_{k+1},p_i \rangle}{\|p_i\|^2_2}=\frac{\langle -A^\top (Ax^{k+1}-b),x^{i+1}-x^i\rangle}{\delta_i\|p_i\|^2_2}\\
		&=\frac{\langle -Ax^{k+1}+b,Ax^{i+1}-Ax^i\rangle}{\|Ax^i-b\|^2_2}=\frac{\langle -r^{k+1},r^{i+1}-r^i\rangle}{\|r^i\|^2_2}\\
		&=\begin{cases}
				-\frac{\|r^{k+1}\|^2_2}{\|r^{k}\|^2_2}, \ &\text{if} \ i=k,\\
			0, & \text{otherwise},
		\end{cases}
	\end{aligned}
	$$
	where the last equality follows from Proposition \ref{prop-property-Krylov}.
	Hence, in this case, Algorithm \ref{Algo-3} simplifies to
    	\begin{equation*}
    		\left\{
    		\begin{array}{l}
    			\delta_k=\lVert Ax^{k}-b\rVert_2^2/\lVert p_k\rVert_2^2,\\
    			x^{k+1}=x^k+\delta_kp_k,\\
    			d_{k+1}=-A^\top (Ax^{k+1}-b),\\
    			\eta_{k+1}=-\|Ax^{k+1}-b\|^2_2/\lVert Ax^k-b\rVert_2^2,\\
    			p_{k+1}=d_{k+1}-\eta_{k+1}p_k.
    		\end{array}\right.  \Leftrightarrow \left\{
    		\begin{array}{ll}
    		\delta_k=\|r^k\|^2_2/\|p_k\|^2_2,\\
    		x^{k+1}=x^k+\delta_k p_k,  \\
    		r^{k+1}=r^k+\delta_k Ap_k,  \\
    		\tau_k=\|r^{k+1}\|^2_2/\|r^k\|^2_2,
    		\\
    		p_{k+1}=-A^\top r^{k+1}+\tau_k p_k.
    		\end{array}
    		\right.
    	\end{equation*}
    	This iteration scheme is precisely the conjugate gradient normal equation (CGNE) method \cite[Section 11.3.9]{golub2013matrix}.
    	Furthermore, the vectors $\left\{\frac{p_0}{\lVert p_0\rVert_2},\ldots,\frac{p_k}{\lVert p_k\rVert_2}\right\}$ form an orthogonal basis for the Krylov subspace $\mathcal{K}_{k+1}(A^\top A,A^\top r^0)$.  Indeed, let $P_k=\begin{pmatrix}
    		\frac{p_0}{\lVert p_0\rVert_2},\ldots,\frac{p_k}{\lVert p_k\rVert_2}
    	\end{pmatrix}$. The $(i,j)$-th $(i\leq j)$ entry of the symmetric matrix $P_k^\top A^\top AP_k$ satisfies
    	\begin{equation*}
    		\begin{aligned}
    			\frac{\langle Ap_i,Ap_j\rangle}{\lVert p_i\rVert_2\lVert p_j\rVert_2}=&\frac{\langle A(x^{i+1}-x^i),A(x^{j+1}-x^{j})\rangle}{\delta_{i}\delta_{j}\lVert p_i\rVert_2\lVert p_j\rVert_2}=\frac{\langle r^{i+1}-r^i,r^{j+1}-r^{j}\rangle}{\delta_{i}\delta_{j}\lVert p_i\rVert_2\lVert p_j\rVert_2}\\
    			=&\frac{\langle r^{i+1},r^{j+1}\rangle-\langle r^{i+1},r^{j}\rangle}{\delta_i\delta_{j}\lVert p_i\rVert_2\lVert p_j\rVert_2}
    			+\frac{-\langle r^{i},r^{j+1}\rangle+\langle r^{i},r^{j}\rangle}{\delta_i\delta_{j}\lVert p_i\rVert_2\lVert p_j\rVert_2}   
    			\\
    				=&\left\{
    			\begin{array}{cc}
    				\frac{1-\eta_{i+1}}{\delta_i}, \ &\text{if }  i=j,\\
    				\sqrt{\frac{-\eta_{i+1}}{\delta_i\delta_{i+1}}},& \ \text{if } i=j-1  \\
    				0,&\text{otherwise},
    			\end{array}
    			\right.
    		\end{aligned}
    	\end{equation*}
    	where the last equality follows from Proposition \ref{prop-property-Krylov} and $\delta_i\lVert p_i\rVert^2_2=\|r^i\|^2_2$. Hence, the iteration scheme can be viewed as the Lanczos process \cite[Section 10.1.2]{golub2013matrix}  on the subspace  $\mathcal{K}_{k+1}(A^\top A,A^\top r^0)$ in some sense.
    	
    	
    	

  \subsection{Comparison to randomized-sketching-based methods}\label{subsec 4.4}
In this subsection, we compare our proposed IS-Krylov method with two typical classes of randomized-sketching-based approaches: the sketched GMRES (SGMRES) method \cite{nakatsukasa2024fast,derezinski2024recent} and the randomized GMRES (RGMRES) \cite{balabanov2022randomized}.   We consider the linear system \(Ax = b\), where \(A \in \mathbb{R}^{n \times n}\) is a \emph{nonsingular} matrix and \(b \in \mathbb{R}^n\).


  \textbf{The SGMRES method.}
 Let \(x^0\) be the initial guess, \(r^0 = Ax^0 - b\) the initial residual, and \(\mathcal{K}_k(A, r^0)\) the associated Krylov subspace. Suppose that $S\in\mathbb{R}^{n\times q}$ is a randomized sketching matrix. At each iteration, the SGMRES method \cite{nakatsukasa2024fast} solves the sketched least-squares problem
$$\min \frac{1}{2}\lVert S^\top(Ax-b)\rVert_2^2\quad\text{ s.t. }x\in x^0+\mathcal{K}_k(A,r^0).$$
When the sketching matrix $S$ satisfies $SS^\top =I_n$, this method reduces to the classical GMRES method.

%
  \textbf{The RGMRES method.}
The RGMRES method builds upon the randomized Gram-Schmidt (RGS) orthogonalization process. To establish context, we first briefly review the classical Gram-Schmidt (CGS) procedure. Given vectors \( w_1,\ldots,w_m \in \mathbb{R}^n \) to be orthogonalized, the CGS process initializes with \( Q_1 = q_1 = w_1/\|w_1\|_2 \) and iterates as follows
\begin{equation*}
\left\{
\begin{array}{rl}
 q_{i+1}&=w_{i+1}-Q_{i}Q_{i}^\dagger w_{i+1},\\
 q_{i+1}&=q_{i+1}/\lVert q_{i+1}\rVert_2,\\
 Q_{i+1}&=[ Q_{i}, q_{i+1}].
\end{array}\right..
\end{equation*}
Note that, under infinite precision arithmetic, the pseudoinverse $Q_i^\dagger$ can be replaced by $Q_i^\top$ due to the column orthonormality of $Q_{i}$. However, in finite precision, rounding errors may destroy this orthogonality. To address this, for $i\geq1$, the update of $q_{i+1}$ is usually written as
$$q_{i+1}=w_{i+1}-Q_{i}y_{i+1} \ \text{with} \ y_{i+1}\in\arg\min \ \lVert Q_{i}y-w_{i+1}\rVert_2^2.$$
The RGS process \cite{balabanov2022randomized} reduces computational cost by solving sketched least-squares problems. Using a random sketching \( S^\top \in \mathbb{R}^{q \times n} \), it initializes with \( Q_1^r = q_1^r = w_1/\|S^\top w_1\|_2 \) and iterates
\begin{equation*}
\left\{
\begin{array}{rl}
y_{i+1}^r&\in\arg\min\lVert S^\top(Q_{i}^ry-w_{i+1})\rVert_2^2,\\
 q_{i+1}^r&=w_{i+1}-Q_{i}^ry_{i+1}^r,\\
 q_{i+1}^r&=q_{i+1}^r/\lVert S^\top q_{i+1}^r\rVert_2,\\
 Q_{i+1}^r&=[ Q_{i}^r, q^r_{i+1}].
\end{array}\right..
\end{equation*}
When $SS^\top=I_n$, the RGS process reduces to the CGS process. Although RGS solves a lower-dimensional subproblem at each step, if the sketching matrix $S$ satisfies certain properties, it still guarantees the construction of a complete orthonormal basis for the original subspace, i.e., $\text{span}\{q_1^r,\ldots,q_m^r\}=\text{span}\{w_1,\ldots,w_m\}$. Building on the RGS process, the RGS-Arnoldi algorithm has been proposed in \cite{balabanov2022randomized} for the efficient construction of Krylov subspace bases. For example, to construct an orthogonal basis for the Krylov subspace $\mathcal{K}_k(A,r^0)$, the RGS-Arnoldi process initializes with $w_1=r^0$ and $q_1^r=w_1/\lVert S^\top w_1\rVert_2$. For $i\geq1$, the RGS-Arnoldi process iterates as follows
\begin{equation*}
\left\{
\begin{array}{l}
 w_{i+1}=Aq^r_{i},\\
 q_{i+1}^r\text{ is obtained by using the RGS process}.\\ 
\end{array}\right..
\end{equation*}
By replacing the classical Arnoldi process with the RGS-Arnoldi process, we get the RGMRES method.

 It can be seen that a key distinction between randomized sketching-based methods and the proposed IS-Krylov subspace framework lies in the treatment of the sketching matrix. In both SGMRES and RGMRES methods, the randomized sketching matrix $S$ is randomly initialized but remains fixed throughout the entire iterative process. In contrast, the IS-Krylov subspace method allows the iterative sketching matrices $\{S_k\}_{k\geq0}$ to be generated dynamically from user-defined probability spaces.  Moreover, our IS-Krylov subspace framework can handle linear systems regardless of the dimensions \( m \) and \( n \), or the rank of \( A \).

\section{Numerical experiments}\label{sec-exp}

In this section, we implement the efficient RIM-AS (Algorithm \ref{Algo-2}) and IS-Krylov (Algorithm \ref{Algo-3}). We also compare the proposed methods with several state-of-the-art algorithms, including the randomized average block Kaczmarz (RABK) method \cite{necoara2019faster}, the stochastic conjugate gradient  (SCG) method \cite{zeng2024adaptive}, as well as the built-in {\sc MATLAB} functions \texttt{pinv} and \texttt{lsqminnorm}. All the methods are implemented in MATLAB R2024a for Windows 11 on a LAPTOP PC with an Intel Core Ultra 7 155H @ 1.40 GHz and 32 GB memory. The code to reproduce our results can be found at \href{https://github.com/xiejx-math/RIM-Krylov}{https://github.com/xiejx-math/RIM-Krylov}.

We begin by introducing several types of sketching matrices that can be employed in our algorithm. These techniques are widely used in randomized numerical linear algebra and optimization, each offering unique advantages that make them particularly suitable for specific problems. For a more comprehensive discussion of additional sketching methods, we refer the reader to, e.g., \cite[Sections 8 and 9]{martinsson2020randomized}.

\textbf{Partition sampling:} Given the size of block $q$, we consider the index parition of $[m]$:
\begin{equation*}
	\begin{aligned}
		&\mathcal{I}_i = \{\varpi(k) : k = (i-1)q+1, (i-1)q+2, \dots, iq\}, \quad i = 1, 2, \dots, t-1,\\
		&\mathcal{I}_t = \{\varpi(k) : k = (t-1)q+1, (t-1)q+2, \dots, m\}, \quad |\mathcal{I}_t| \leq q,
	\end{aligned}
\end{equation*}
where $\varpi$ is a uniformly random permutation of $[m]$. We select the sampling matrix $S^\top = I_{ \mathcal{I}_i}$ with the probability $\lVert A_{\mathcal{I}_i}\rVert_F^2 / \lVert A\rVert_{F}^2$. Partition sampling sketches are very cheap to compute; indeed, we need not even compute \(S^\top A\) and \(S^\top b\) since it is simply equivalent to fetching the rows of \(A\) indexed by a random subset. We refer to Algorithm \ref{Algo-2} and Algorithm \ref{Algo-3} with partition sampling as RIM-AS with  partition sampling (RIM-AS-PS) and IS-Krylov with  partition sampling (IS-Krylov-PS), respectively.

\textbf{Uniform sampling and CountSketch:} 
We consider the uniform sampling of $q$ unique indices forming the set \(\mathcal{J}\),  where \(\mathcal{J} \subset [m]\) and the cardinality  of \(\mathcal{J}\) is \(q\). The number of possible choices for \(\mathcal{J}\) is \(\binom{m}{q}\), and the probability of selecting any specific \(\mathcal{J}\) is \(\operatorname{Prob}(\mathcal{J}) = 1/\binom{m}{q}\).  The uniform sketching matrix is defined as \(S^\top_{US} = I_{\mathcal{J},:}\).

The CountSketch, originally from the streaming data literature \cite{charikar2004finding,cormode2005improved} and popularized for matrix sketching by \cite{clarkson2017low}, is defined as
$
S_{CS}^\top=DI_{\mathcal{J},:},
$
where \(D \in \mathbb{R}^{q \times q}\) is a diagonal matrix with elements sampled uniformly from \(\{-1, 1\}\), and \(\mathcal{J}\) is chosen by uniformly sampling \(q\) elements from \([m]\). Since \(S_{US} S_{US}^\top = S_{CS} S_{CS}^\top\) and \(\|S_{US}^\top (Ax-b)\|_2 = \|S_{CS}^\top (Ax-b)\|_2\), the iteration schemes in Algorithms \ref{Algo-2} and \ref{Algo-3} for both uniform sampling and CountSketch are identical. In these cases, we refer to Algorithms \ref{Algo-2} and \ref{Algo-3} as RIM-AS with CountSketch (RIM-AS-CS) and IS-Krylov with CountSketch (IS-Krylov-CS), respectively.




\textbf{Gaussian sketch:} 
A Gaussian sketch is defined as a random matrix \(S\in\mathbb{R}^{m\times q}\), where each element is independently and identically distributed (i.i.d.) according to the standard Gaussian distribution. For a dense matrix \( A \in \mathbb{R}^{m \times n} \), the computational cost of forming the sketched product \( S^\top A \) typically requires \( O(mnq) \) floating-point operations, which can become prohibitive for large-scale problems. However, in practical implementations, this computational burden can be substantially mitigated through parallel computing techniques or by exploiting sparsity in \( A \) when applicable. For a comprehensive analysis of these practical benefits and implementation considerations, we refer readers to \cite{meng2014lsrn}. 
In this case, Algorithm \ref{Algo-3} yields  IS-Krylov with Gaussian sketch (IS-Krylov-GS).

\textbf{Subsampled randomized Hadamard transform (SRHT):} 
The SRHT \cite{woolfe2008fast,tropp2011improved} is defined as a sketching matrix  \(S\in\mathbb{R}^{m\times q}\) with the structure
$
S^\top=\sqrt{\frac{m}{q}}I_{\mathcal{J},:}H_m D,
$
where \(D\in\mathbb{R}^{m\times m}\) is a diagonal matrix with elements sampled uniformly from \(\{-1,1\}\), \(H_m\in\mathbb{R}^{m\times m}\) is the Hadamard matrix\footnote{The Hadamard transform is defined for \(m=2^p\) for some positive integer \(p\). If \(m\) is not a power of $2$, a standard practice is to pad the data matrix with \(2^{\lceil\log_2(m)\rceil}-m\) additional rows of zeros.} of order \(m\), and the set \(\mathcal{J}\) is  chosen by uniformly sampling \(q\) elements from $[m]$. The SRHT can be efficiently computed using \(O(mn \log q)\) operations via the fast Walsh-Hadamard transform \cite{fino1976unified}. In this context, Algorithm \ref{Algo-3} yields  IS-Krylov with SRHT (IS-Krylov-SRHT).

\subsection{Numerical setup}

We consider two types of coefficient matrices in our experiments. The first type consists of synthetic Gaussian matrices generated using the {\sc MATLAB} function \texttt{randn}. Specifically, for given parameters \( m \), \( n \), \( r \), and a target condition number \( \kappa > 1 \), we construct a dense matrix \( A \in \mathbb{R}^{m \times n} \) as \( A = U D V^\top \), where \( U \in \mathbb{R}^{m \times r} \), \( D \in \mathbb{R}^{r \times r} \), and \( V \in \mathbb{R}^{n \times r} \). In {\sc MATLAB} notation, these matrices are generated via \texttt{[U, $\sim$] = qr(randn(m, r), 0)}, \texttt{[V, $\sim$] = qr(randn(n, r), 0)}, and \texttt{D = diag(1 + ($\kappa$ - 1) .* rand(r, 1))}. As a result, the condition number and the rank of \( A \) are upper bounded by \( \kappa \) and \( r \), respectively.  The second type is drawn from real-world datasets, where only the coefficient matrices \( A \) are used in the experiments. These matrices are obtained from the SuiteSparse Matrix Collection \cite{kolodziej2019suitesparse} and LIBSVM \cite{chang2011libsvm} repositories.

In our implementations, to ensure the consistency of the linear system, we first generate the ground-truth solution as \( x^* = \texttt{randn(n,1)} \), and subsequently compute the right-hand side as \( b = A x^* \). All algorithms are initialized with \( x^0 = 0 \). The iterative process is terminated when the relative solution error (RSE), defined as  
$
 \text{RSE} = \frac{\lVert x^k - A^\dagger b \rVert_2^2}{\lVert x^0 - A^\dagger b \rVert_2^2},
$  
 falls below a prescribed tolerance or when the maximum number of iterations is reached. In practice, the quantity \( \|S_k^\top (A x^k - b)\|_2 \) is considered to be zero if it is smaller than \texttt{eps}. Each experiment is repeated over $20$ independent trials to ensure statistical reliability.

Finally, we note that, as established in Remark~\ref{remark-xie-0325-2}, the RABK method emerges as a special case of our framework when partition sampling is employed with \( \ell = 1 \). Similarly, as discussed in Section~\ref{subsection-FP}, the SCG method with partition sampling (SCGP) also falls within our framework, corresponding to the choice \( \ell = 2 \). In particular, we will later provide a detailed comparison between the proposed IS-Krylov-PS method and these two methods.

\subsection{Comparison between Algorithms \ref{Algo-2} and \ref{Algo-3}}

In this subsection, we evaluate the performance of Algorithms \ref{Algo-2} and \ref{Algo-3}, focusing solely on partition sampling among the various available sketching strategies, specifically RIM-AS-PS and IS-Krylov-PS.  The results are presented in Figure \ref{fig:1}, where algorithm performance is measured in terms of iteration count and CPU time.    
The bold line illustrates the median value derived from $20$ trials. The lightly shaded area signifies the  range from the minimum to the maximum values, while the darker shaded one indicates the data lying between the $25$-th and $75$-th quantiles.

Figure~\ref{fig:1} demonstrates that while RIM-AS-PS and IS-Krylov-PS require a nearly identical number of iterations to converge, IS-Krylov-PS exhibits significantly better computational efficiency in terms of CPU time. This aligns with our theoretical complexity analysis in Subsection~\ref{subsection-com-2-3}, which predicts the superior efficiency of the IS-Krylov framework. Therefore, we focus the subsequent experiments on the algorithms derived from the IS-Krylov framework, i.e., Algorithm~\ref{Algo-3}.


\begin{figure}
	\centering
	\includegraphics[width=0.32\linewidth]{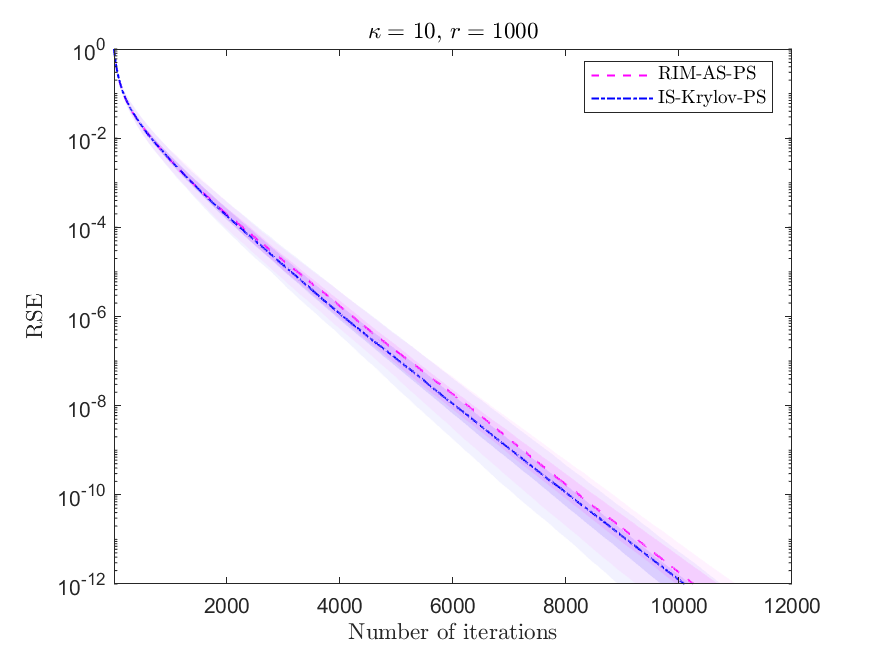}
	\includegraphics[width=0.32\linewidth]{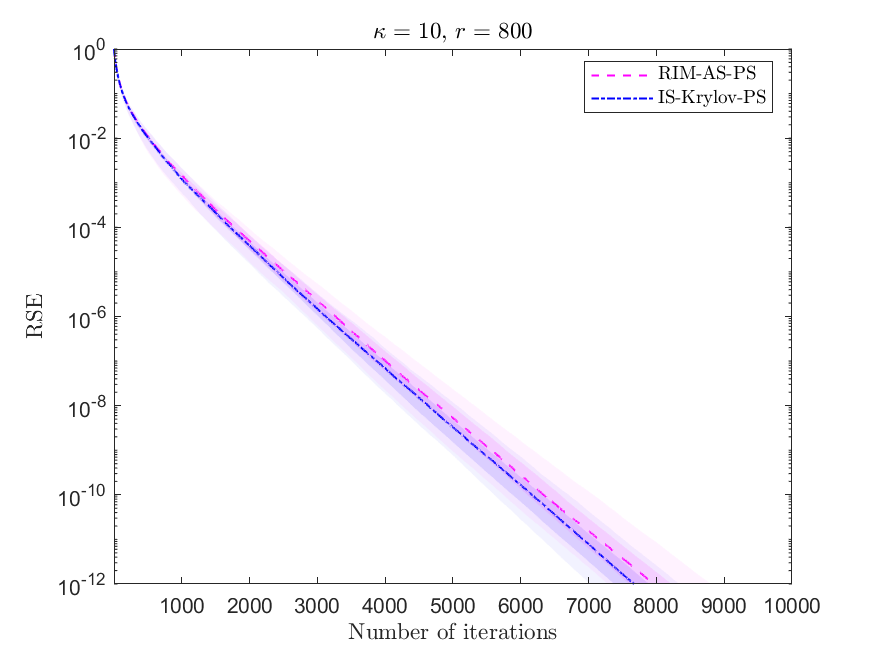}
	\includegraphics[width=0.32\linewidth]{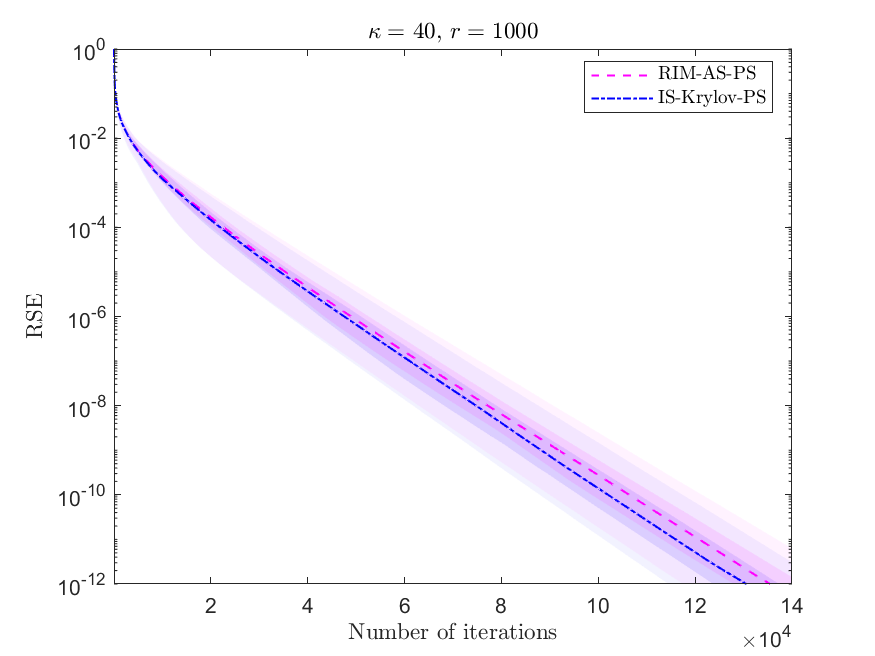}\\
	\includegraphics[width=0.32\linewidth]{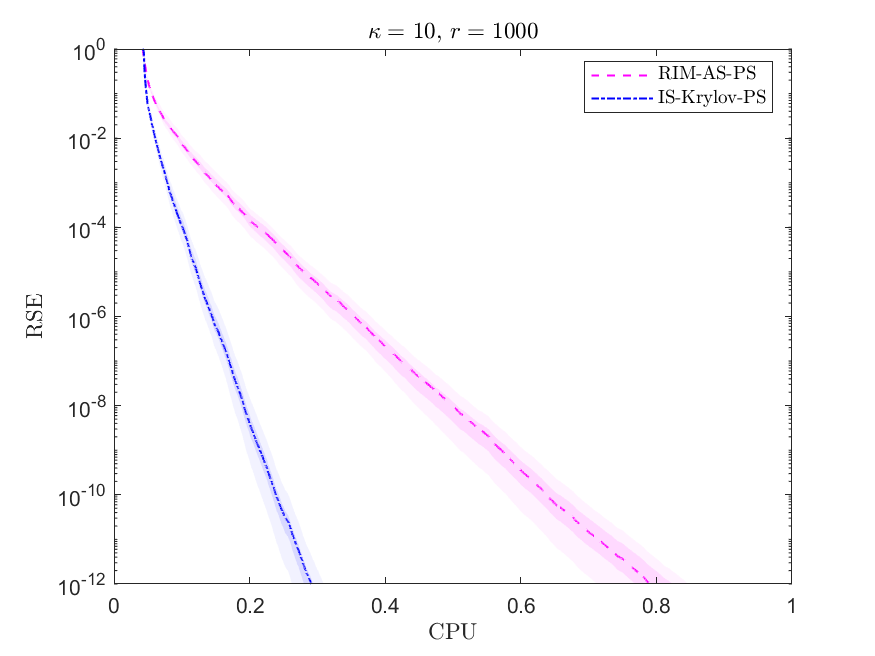}
	\includegraphics[width=0.32\linewidth]{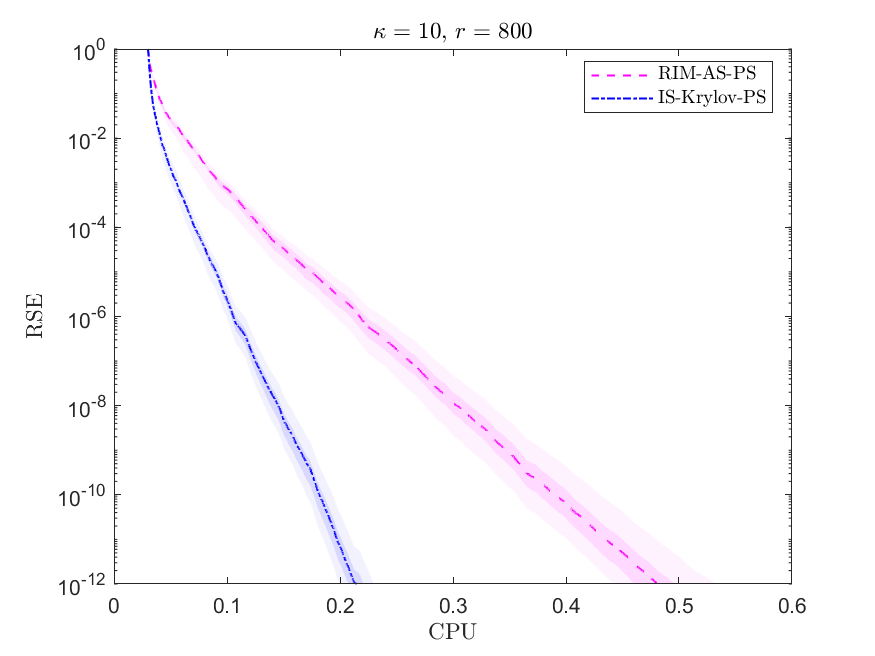}
	\includegraphics[width=0.32\linewidth]{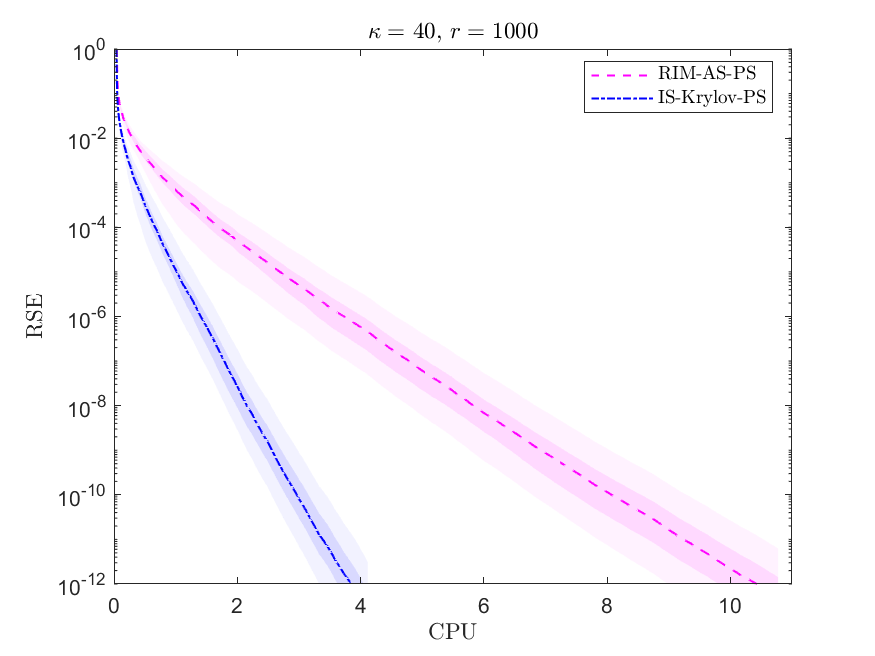}
	\caption{The figures illustrate the  evolution of RSE with respect to the number of iterations (top) and the CPU time (bottom). The title of each subplot indicates the corresponding values of $\kappa$ and $r$.  The other parameters are fixed as $m = 5000$, $n = 1000$, $q = 30$, and $\ell = 10$.  All computations are terminated once RSE$<{10}^{-12}$.}
	\label{fig:1}
\end{figure}

\subsection{ Comparison of different types of sketches}

This subsection presents a comparative evaluation of the four sketching strategies discussed previously, with particular emphasis on analyzing the performance of the algorithms derived from Algorithm \ref{Algo-3}, namely IS-Krylov-PS, IS-Krylov-CS, IS-Krylov-GS, and IS-Krylov-SRHT.  
Figure~\ref{fig:2} reports both the CPU time and the number of iterations required by each method.  

It can be seen from Figure~\ref{fig:2} that all four algorithms converge in the same number of iterations. However, IS-Krylov-GS and IS-Krylov-SRHT incur higher CPU time compared to the other two methods, primarily due to their greater per-iteration computational cost. Although IS-Krylov-PS and IS-Krylov-CS share the same per-iteration complexity, IS-Krylov-PS outperforms IS-Krylov-CS in terms of CPU time. This performance gap is attributed to the additional overhead in IS-Krylov-CS caused by repeatedly extracting rows from the matrix \( A \) during runtime. In contrast, IS-Krylov-PS avoids this overhead by pre-storing the submatrices of \( A \) based on the initial partition, thereby eliminating the need for repeated row access operations. Indeed, IS-Krylov-PS consistently outperforms the other three methods in CPU time across our experiments. Therefore, we will focus our subsequent tests on the IS-Krylov-PS method.

\begin{figure}
	\centering
	\includegraphics[width=0.32\linewidth]{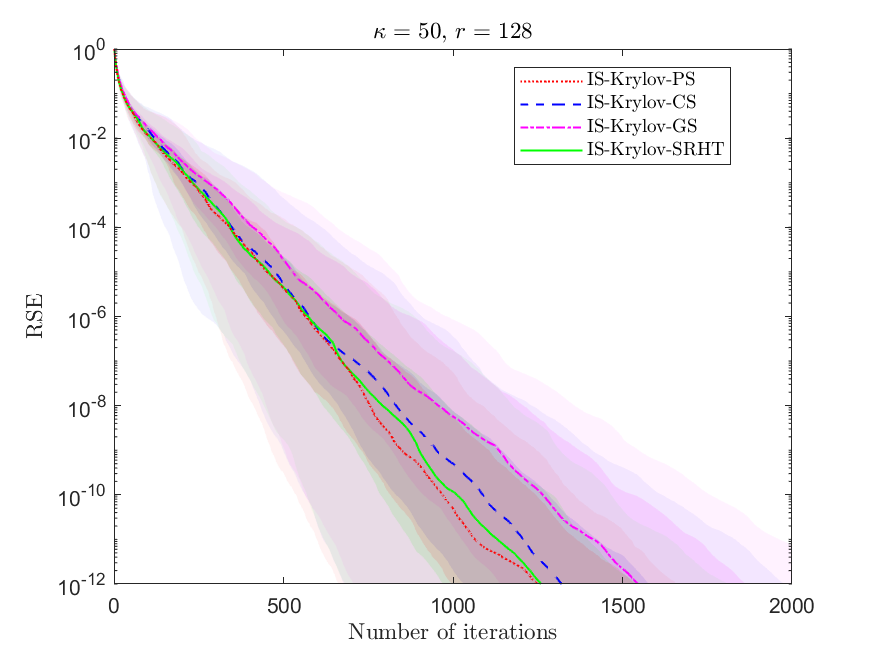}
	\includegraphics[width=0.32\linewidth]{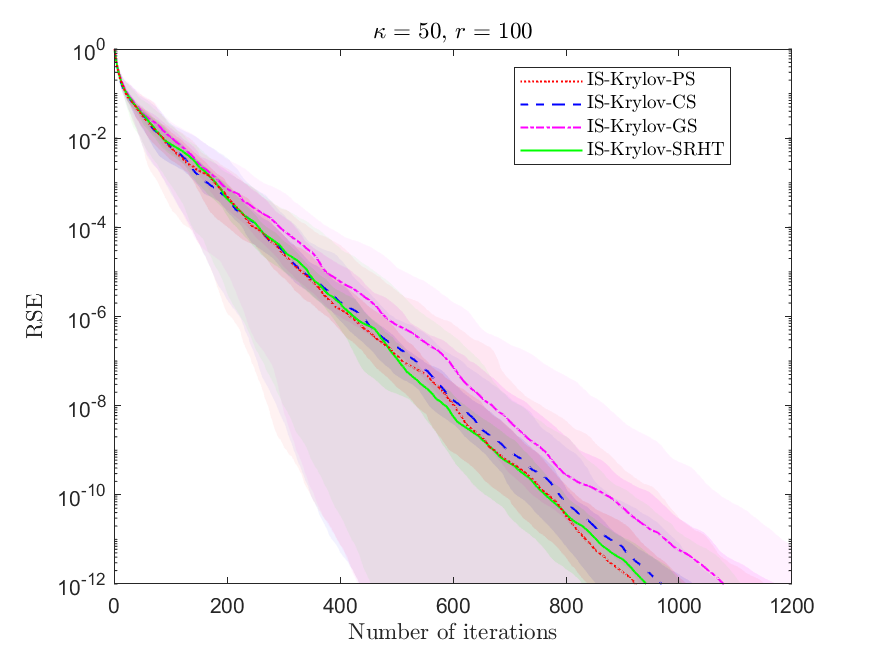}
	\includegraphics[width=0.32\linewidth]{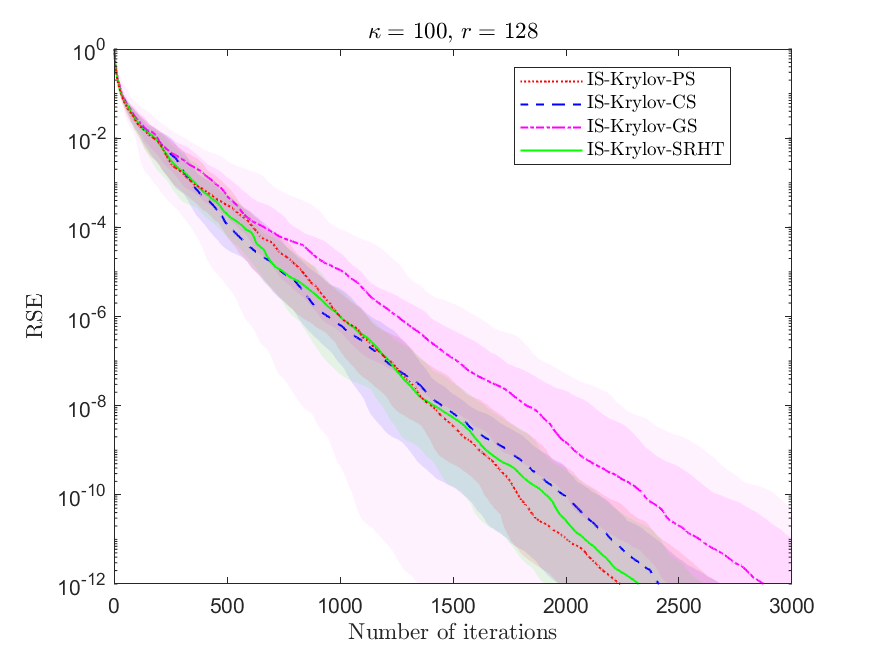}\\
	\includegraphics[width=0.32\linewidth]{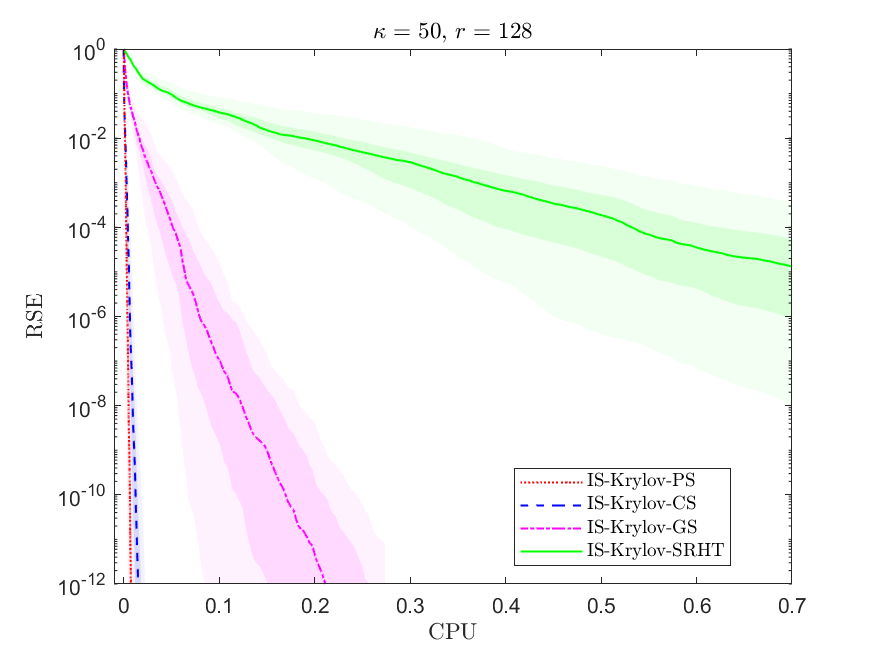}
	\includegraphics[width=0.32\linewidth]{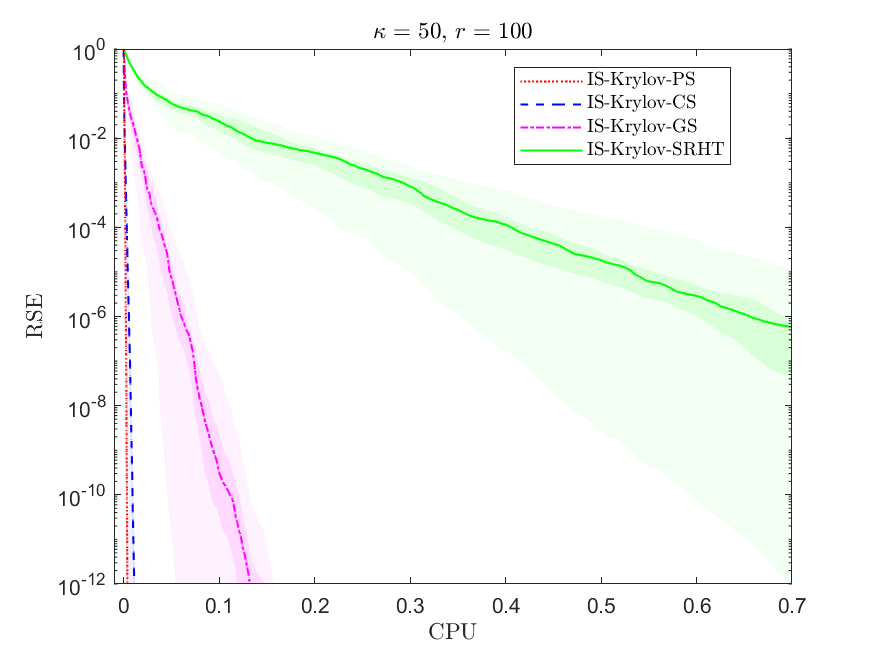}
	\includegraphics[width=0.32\linewidth]{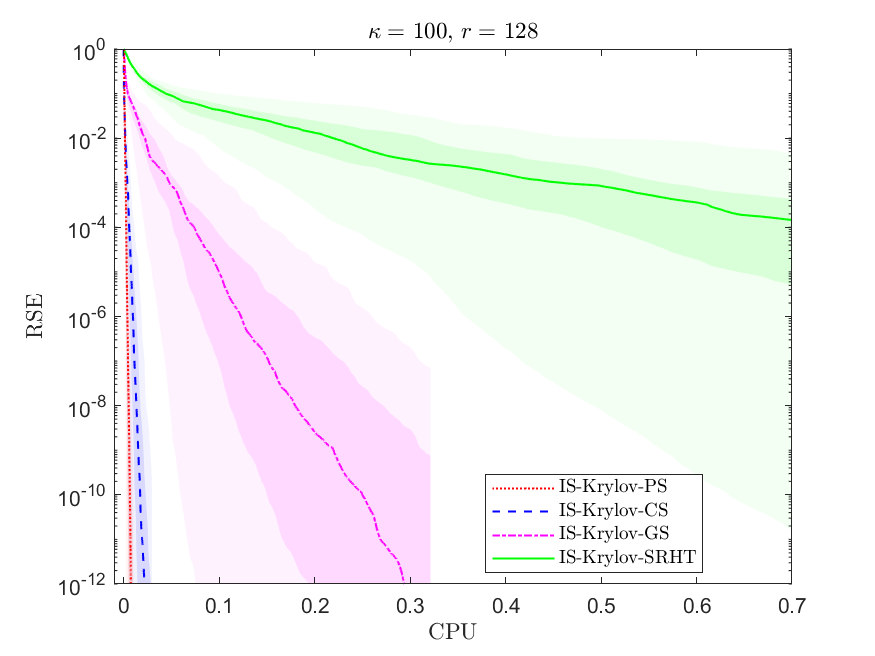}
	\caption{Figures depict the evolution of RSE with respect to the number of iterations (top) and the CPU time (bottom). The title of each plot indicates the values of $\kappa$ and $r$. We set $m=256,n=128,q=30$, and $\ell=10$. All computations are terminated once the number of iterations exceeds a certain limit. }
	\label{fig:2}
\end{figure}

\subsection{Choice of parameters $q$ and $\ell$}

In this subsection, we investigate the impact of the block size \( q \) and the number of previous iterations \( \ell \) on the convergence behavior of the IS-Krylov-PS method. The performance of the algorithms is measured in both the computing time (CPU) and the number of full iterations $(k\cdot\frac{q}{m})$, which  maintains a uniform count of operations for a single pass through the rows of \(A\) across all algorithms. The results are displayed in Figure \ref{fig:3}.

It can be seen from ~\ref{fig:3} that  for any fixed block size $q$, increasing the value of  $\ell$ consistently reduces the required number of full iterations, indicating that incorporating more previous iterations can improve algorithmic efficiency.
When \( \ell \) is fixed, the behavior with respect to \( q \) depends on the value of \( \ell \). Specifically, for relatively small values of \( \ell \) (e.g., \( \ell = 2, 5, 10 \)), the number of full iterations first decreases with increasing \( q \), reaches a minimum, and then increases. In contrast, for larger values of \( \ell \), the number of full iterations tends to increase monotonically with \( q \).
In terms of CPU time, the performance exhibits more nuanced behavior. For instance, when \( \ell = 50 \), although the number of full iterations increases monotonically with \( q \), the corresponding variation in wall-clock time remains relatively modest. This can be attributed to MATLAB’s automatic multithreading in matrix-vector product computations, which dominate the runtime in block sampling-based methods. When both CPU time and the number of full iterations are taken into account, configurations with \( \ell = 10 \) or \( 50 \) and \( q = 32 \) or \( 64 \) demonstrate favorable performance, achieving a desirable balance between convergence speed and computational cost.

\begin{figure}
	\centering
	\includegraphics[width=0.32\linewidth]{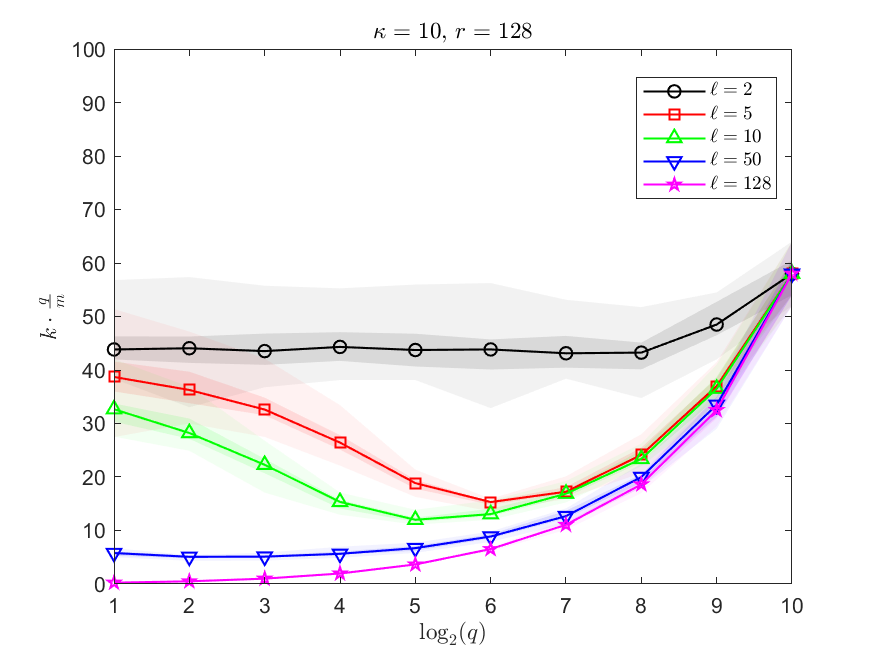}
	\includegraphics[width=0.32\linewidth]{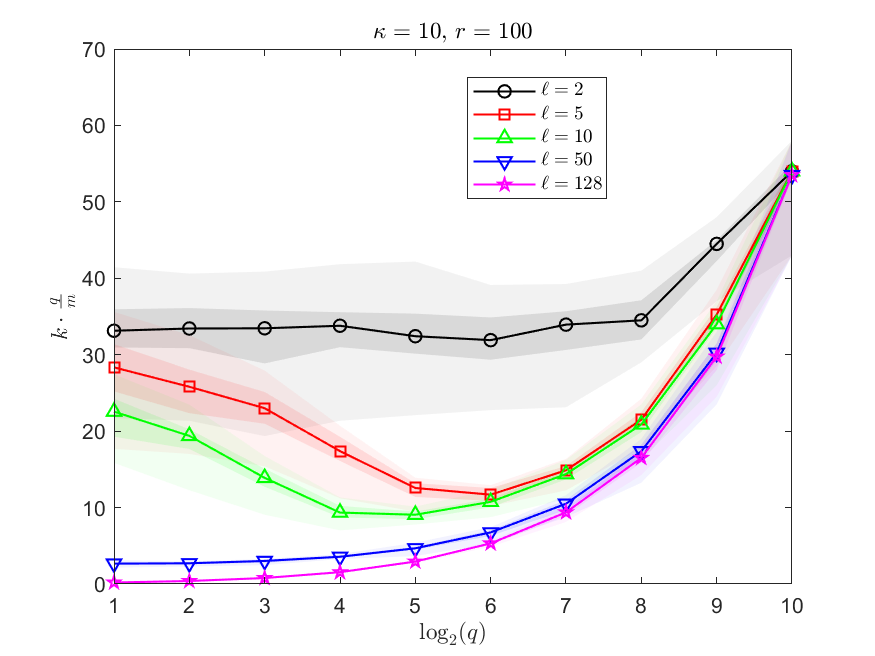}
	\includegraphics[width=0.32\linewidth]{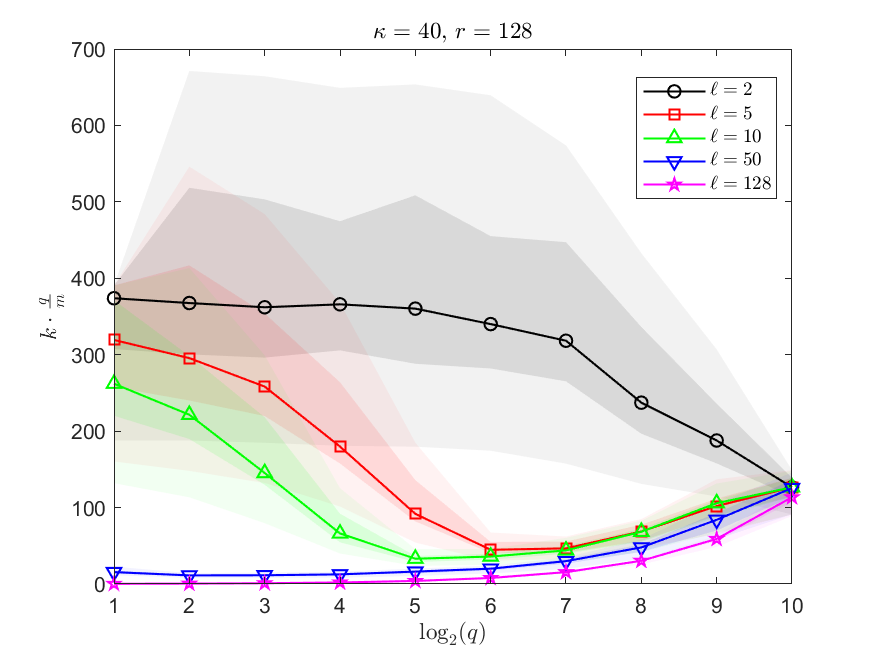}\\
	\includegraphics[width=0.32\linewidth]{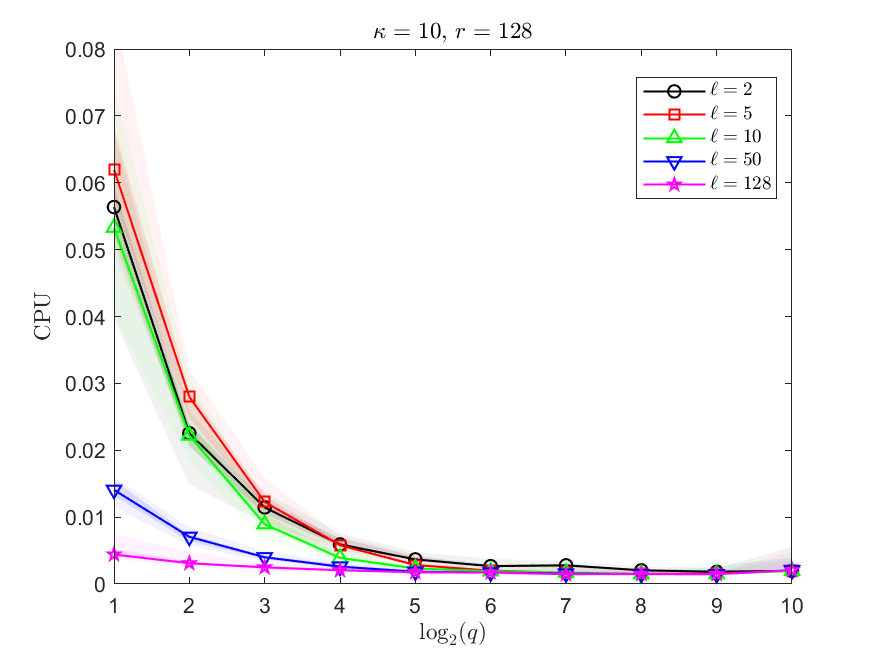}
	\includegraphics[width=0.32\linewidth]{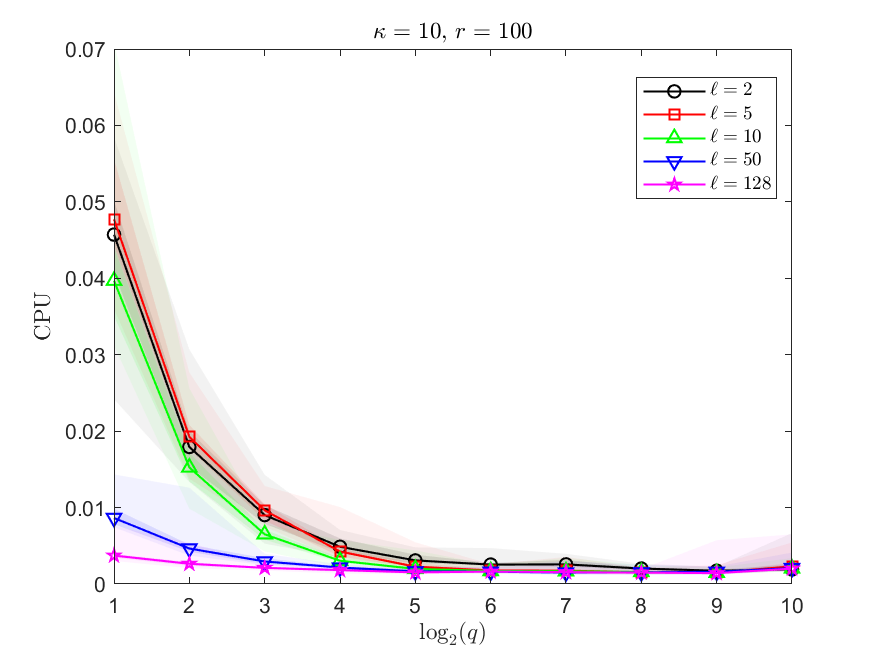}
	\includegraphics[width=0.32\linewidth]{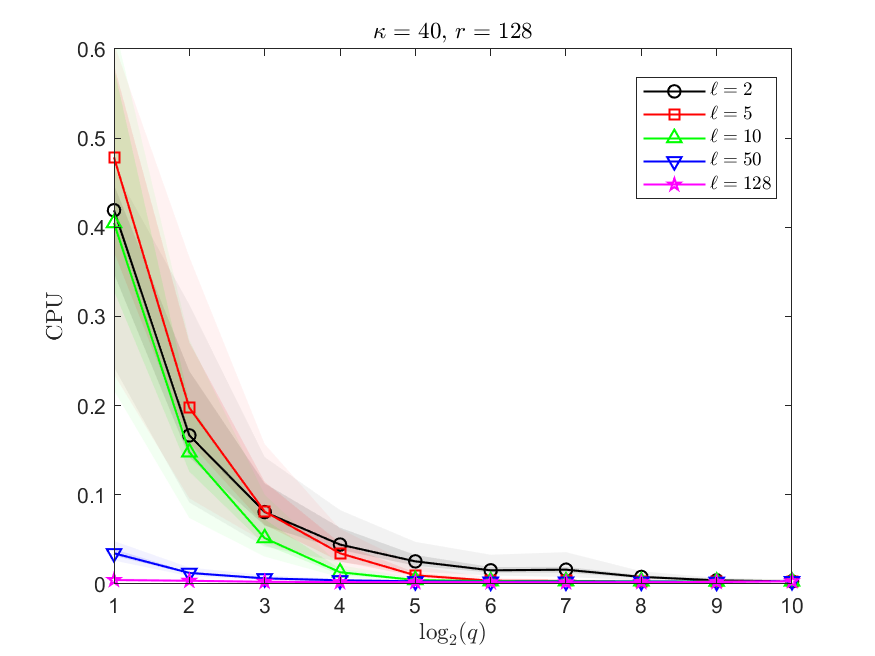}
	\caption{Figures depict the evolution of the number of full iterations (top) and the CPU time (bottom) with respect to the block size $q$ and the number of previous iterations $\ell$. The title of each plot indicates the values of $\kappa$ and $r$. We set $m=1024$ and $n=128$. All computations are terminated once RSE$<{10}^{-12}$.}
	\label{fig:3}
\end{figure}

\subsection{Comparison to RABK and SCGP}
In this section, we compare the performance of the proposed IS-Krylov-PS method with that of RABK and SCGP. Figure \ref{fig:21} presents the experimental results for the case where the coefficient matrices \( A \) are randomized Gaussian matrices.  The results demonstrate that IS-Krylov-PS  outperforms  RABK and SCGP in terms of both iteration counts and CPU time, regardless of the condition number of $A$ and whether it is full rank.

\begin{figure}
\centering
\includegraphics[width=0.32\linewidth]{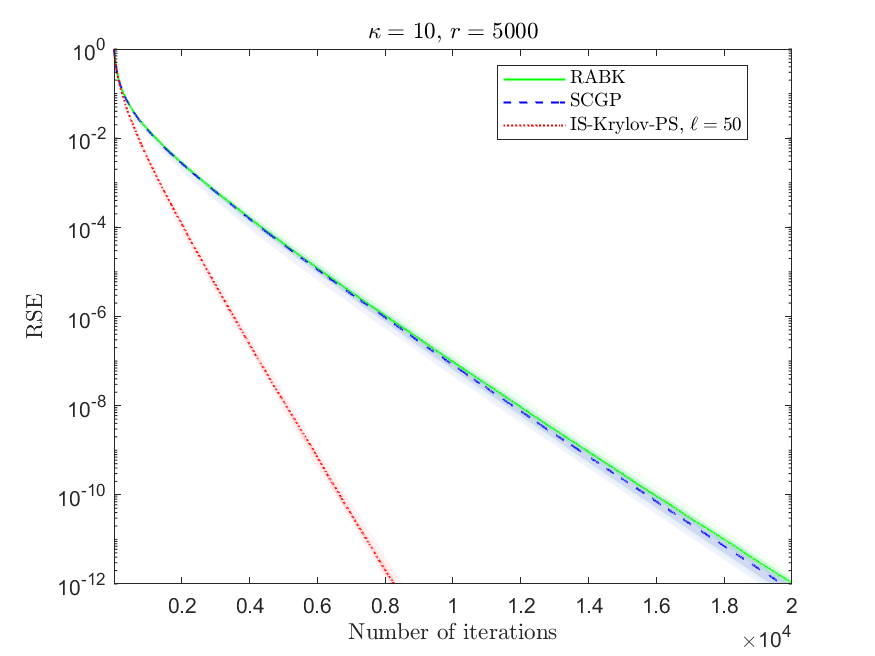}
\includegraphics[width=0.32\linewidth]{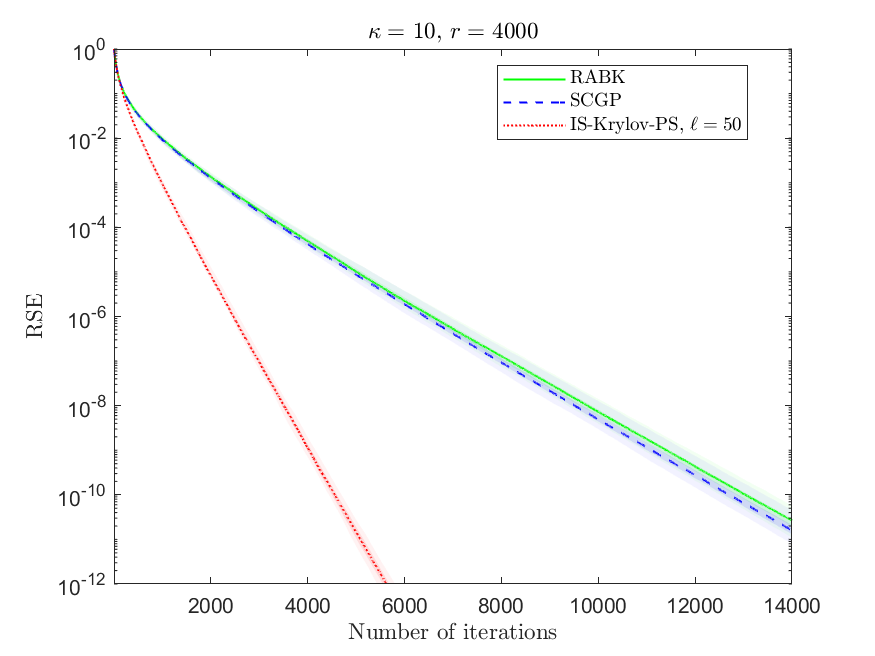}
\includegraphics[width=0.32\linewidth]{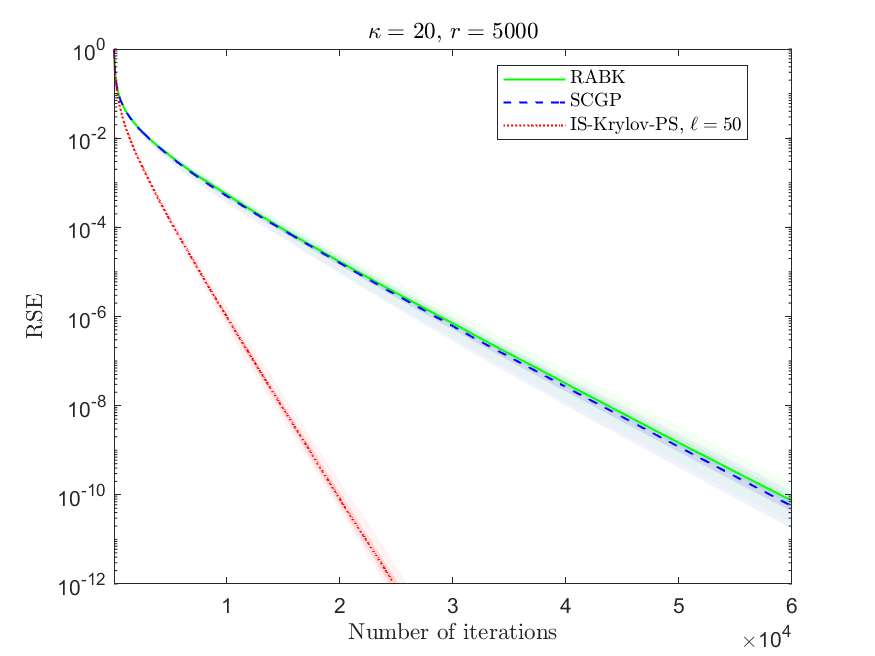}\\
\includegraphics[width=0.32\linewidth]{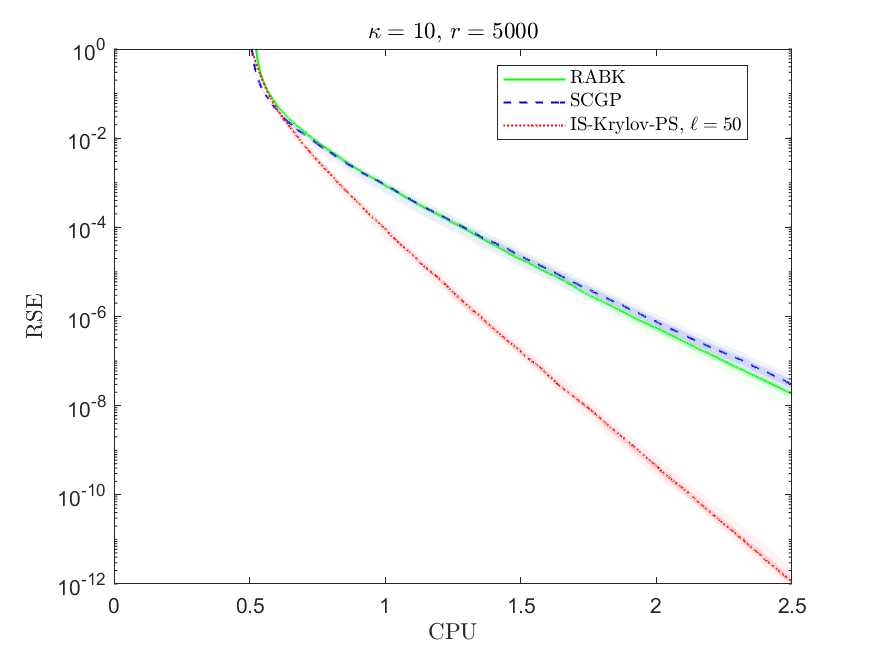}
\includegraphics[width=0.32\linewidth]{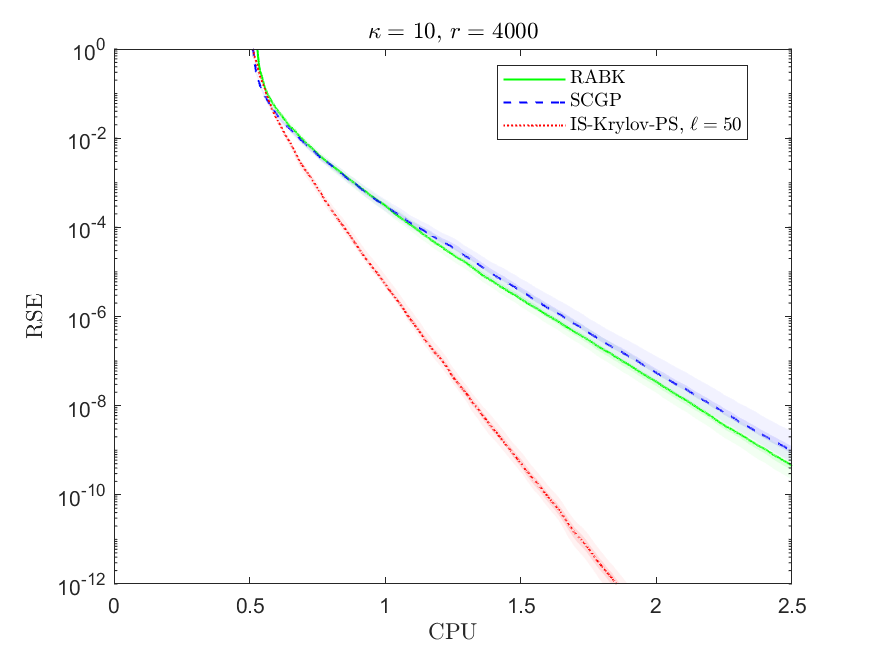}
\includegraphics[width=0.32\linewidth]{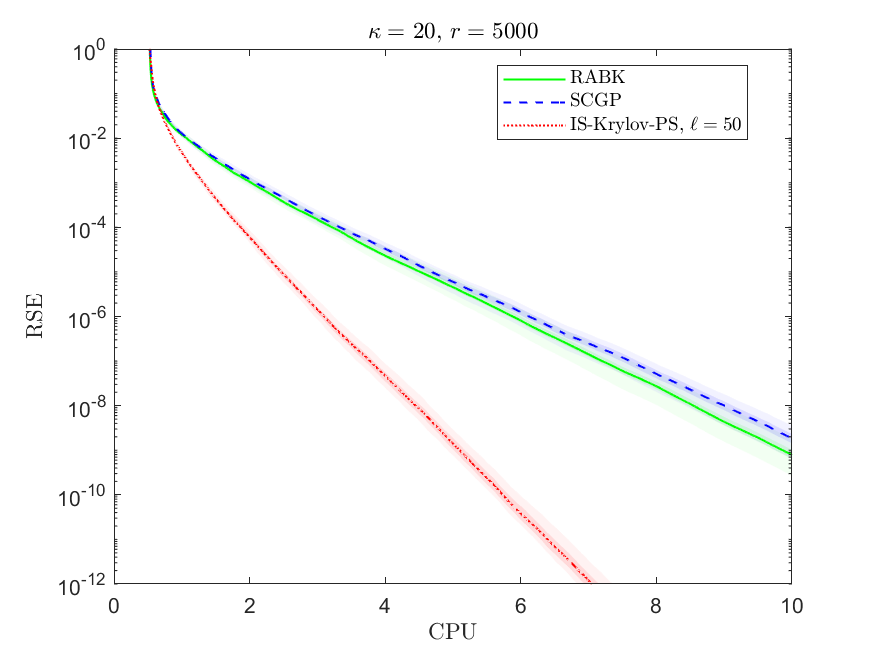} 
\caption{Figures depict the evolution of RSE with respect to the number of iterations (top) and the CPU time (bottom). The title of each plot indicates the values of $\kappa$ and $r$. We set $m=10000,n=5000$, and $q=100$, and for IS-Krylov-PS, we set $\ell=50$. All computations are terminated once the number of iterations exceeds a certain limit.}
\label{fig:21}
\end{figure}

	Table \ref{tab:1} and Figure \ref{fig:4} present the number of  iterations and the computing time for IS-Krylov-PS, RABK, and SCGP when applied to sparse matrices from SuiteSparse Matrix Collection \cite{kolodziej2019suitesparse} and  LIBSVM \cite{chang2011libsvm}. Specifically, the matrices are selected from the SuiteSparse collection:
	{\tt abtaha2}, \texttt{model1}, {\tt crew1}, \texttt{WorldCities}, \texttt{well1033}, {\tt cr42}, \texttt{Franz1}, {\tt GL7d11}, {\tt D\_6}, {\tt rel6}, and {\tt lp\_ship04s}. Additionally, four matrices are taken from LIBSVM:   {\tt a9a}, \texttt{aloi}, \texttt{cod-rna}, and \texttt{protein}. This selection covers a range of matrix properties, including full-rank and rank-deficient, well-conditioned and ill-conditioned,  overdetermined and underdetermined.

\begin{table}[h]
	\centering
	\caption{Average number of iterations and CPU time of RABK, SCGP, and IS-Krylov-PS for solving linear systems with coefficient matrices from the SuiteSparse Matrix Collection~\cite{kolodziej2019suitesparse}.  We set $q=30$ for all methods and $\ell=50$ for IS-Krylov-PS. All computations are terminated once RSE$<10^{-12}$.}
	\resizebox{\textwidth}{!}{
		\begin{tabular}{ccccccccccc}
			\toprule
			\multirow{2}{*}{Matrix} & \multirow{2}{*}{$m \times n$} & \multirow{2}{*}{rank} & \multirow{2}{*}{$\frac{\sigma_{\max}}{\sigma_{\min}}$} & \multicolumn{2}{c}{RABK} & \multicolumn{2}{c}{SCGP} & \multicolumn{2}{c}{IS-Krylov-PS} \\
			\cmidrule(r){5-6} \cmidrule(r){7-8} \cmidrule(r){9-10}
			& & & & Iter & CPU & Iter & CPU & Iter & CPU \\
			\midrule
			abtaha2 & 37932 $\times$ 331 & 331 & $12.22$ & 8505.15 & 0.5303 & 8076.75 & 0.5134 & 1148.35 & \textbf{0.4812} \\
			model1 & 362 $\times$ 798 & 362 & 17.57 & 4947.50 & 0.0304 & 3369.30 & 0.0225 & 787.35 & \textbf{0.0130} \\
			crew1 & 135 $\times$ 6469 & 135 & 18.20 & 1501.65 & 0.1261 & 794.85 & 0.0786 & 185.80 & \textbf{0.0354} \\
			WorldCities & 315 $\times$ 100 & 100 & 66.00 & 10939.90 & 0.0281 & 2607.65 & 0.0082 & 185.30 & \textbf{0.0025} \\
			well1033 & 1033 $\times$ 320 & 320 & 166.13 & 1191857 & 3.6360 & 1035160 & 3.5378 & 12688 & \textbf{0.1250} \\
			cr42 & 905 $\times$ 1513 & 905 & 688.12 & 614421.20 & 3.8735 & 230570.90 & 1.7035 & 12888.80 & \textbf{0.2614} \\
			Franz1 & 2240 $\times$ 768 & 755 & $2.74\text{e}+15$ & 2584.40 &\textbf{ 0.0157 }& 2592.80 & 0.0171 & 1206.75 &0.0187 \\
			GL7d11 & 1019 $\times$ 60 & 60 & $1.65\text{e}+16$ & 364.10 & 0.0021 & 346.85 & 0.0020 & 71.55 & \textbf{0.0019} \\
			D\_6 & 970 $\times$ 435 & 339 & $1.12\text{e}+17$ & 2416.10 & 0.0134 & 2003.85 & 0.0124 & 659.95 & \textbf{0.0101} \\			
			rel6 & 2340 $\times$ 157 & 155 & $1.31\text{e}+17$ & 2568.85 & 0.0101 & 2246.00 & 0.0101 & 373.75 & \textbf{0.0054} \\
			lp\_ship04s & 402 $\times$ 1506 & 360 & $2.99\text{e}+17$ & 243593.80 & 2.5658 & 187281.00 & 2.1739 & 4167.05 & \textbf{0.1039} \\
			\bottomrule
		\end{tabular}
	}
	\label{tab:1}
\end{table}

\begin{figure}
	\centering
	\includegraphics[width=0.32\linewidth]{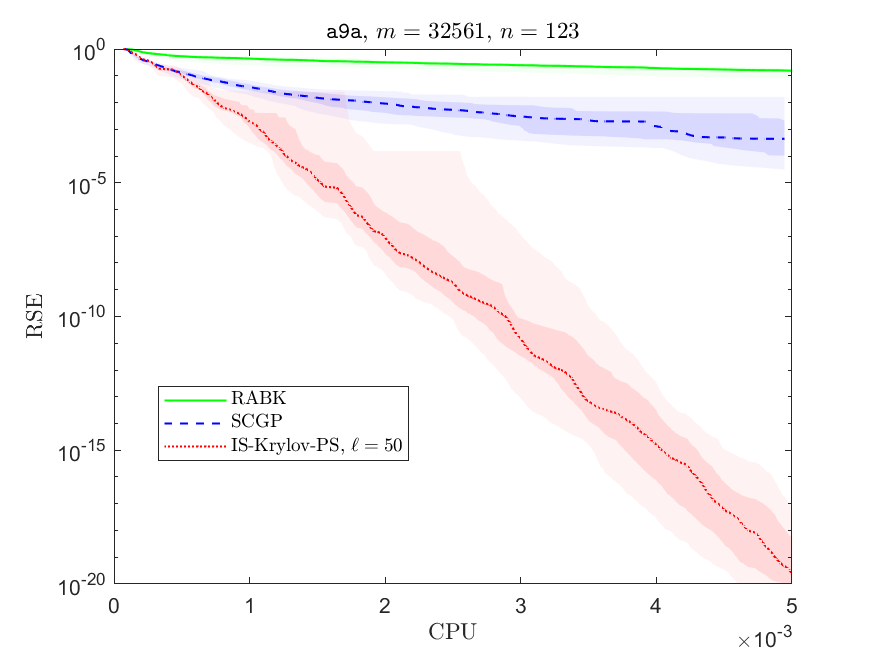}
	\includegraphics[width=0.32\linewidth]{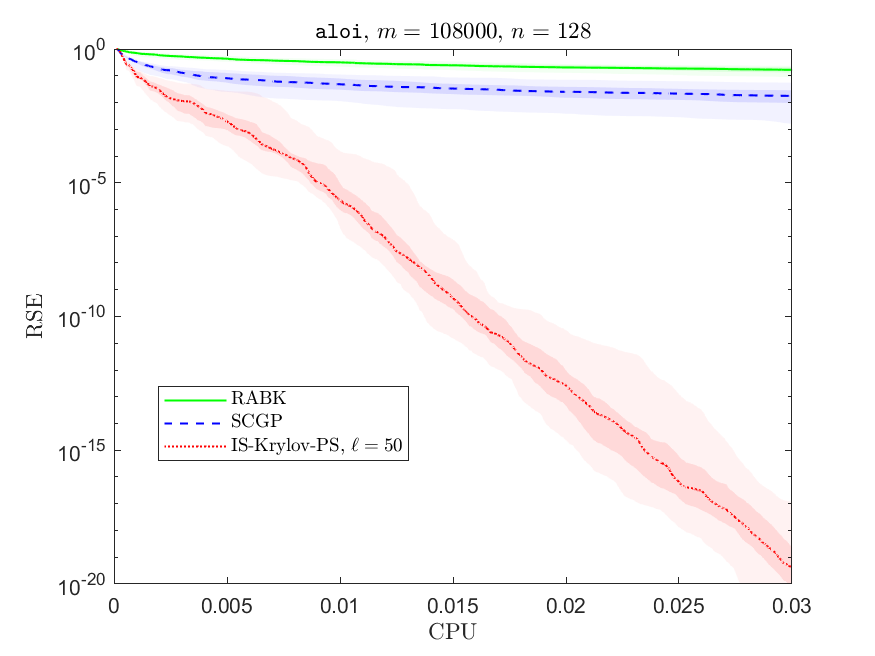}\\
	\includegraphics[width=0.32\linewidth]{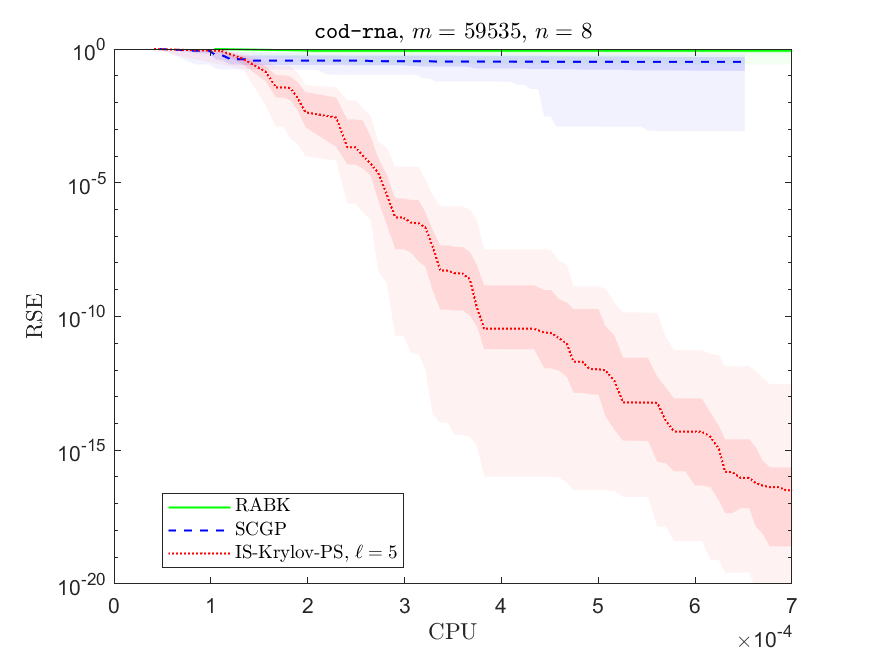}
	\includegraphics[width=0.32\linewidth]{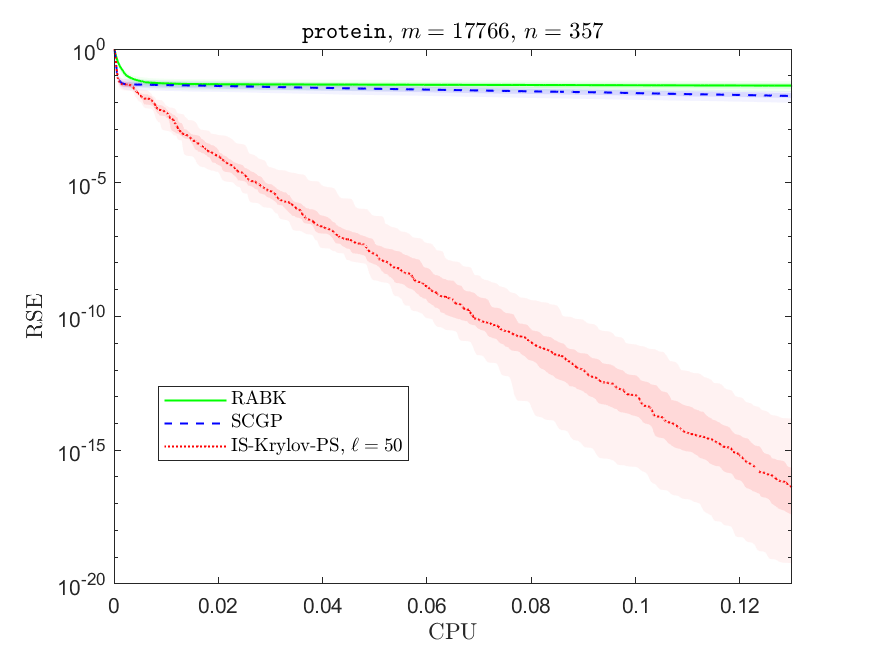}
	\caption{Performance of RABK, SCGP, and IS-Krylov-PS for linear systems with coefficient matrices from LIBSVM \cite{chang2011libsvm}. Figures depict the evolution of RSE with respect to the number of iterations and the CPU time. Each plot title indicates the dataset name and data dimensions. We set $q=300$ and stop the algorithms if the number of iterations exceeds a certain limit.}
	\label{fig:4}
\end{figure}

Table~\ref{tab:1} reports the numerical results on matrices from the SuiteSparse Matrix Collection. It can be seen that IS-Krylov-PS consistently outperforms both RABK and SCGP in terms of iteration count across all test cases. However, in certain instances, such as \texttt{Franz1}, SCGP achieves lower CPU time than IS-Krylov-PS. This can be attributed to the higher per-iteration computational cost of IS-Krylov-PS compared to SCGP.  Figure~\ref{fig:4} presents the results on matrices from LIBSVM. The results further demonstrate the efficiency of IS-Krylov-PS.

\subsection{Comparison to \texttt{pinv} and \texttt{lsqminorm}}

In this subsection, we compare the performance of IS-Krylov-PS with the MATLAB functions \texttt{pinv} and \texttt{lsqminnorm}. To facilitate the computation of the least-norm solution, we consider randomly generated matrices with full column rank, where \( m \geq n \) and the numerical rank \( r = n \). We construct the exact solution vector \( x^* \) by setting \( x^* = \texttt{randn(n,1)} \), and compute the right-hand side as \( b = A x^* \). This construction ensures that \( x^* \) is the unique minimum-norm solution of the resulting linear system.

We terminate IS-Krylov-PS once the accuracy of its approximate solution matches that of the solutions computed by \texttt{pinv} and \texttt{lsqminnorm}. Figure~\ref{fig:5} presents the CPU time as a function of the number of rows in the matrix \( A \), with the number of columns held fixed. The results indicate that for smaller problem sizes, \texttt{pinv} and \texttt{lsqminnorm} are more efficient than IS-Krylov-PS. However, as the size of \( A \) increases, IS-Krylov-PS demonstrates significantly better computational efficiency,
 highlighting  underscoring its scalability and suitability for large-scale problems.
\begin{figure}
	\centering
	 \includegraphics[width=0.32\linewidth]{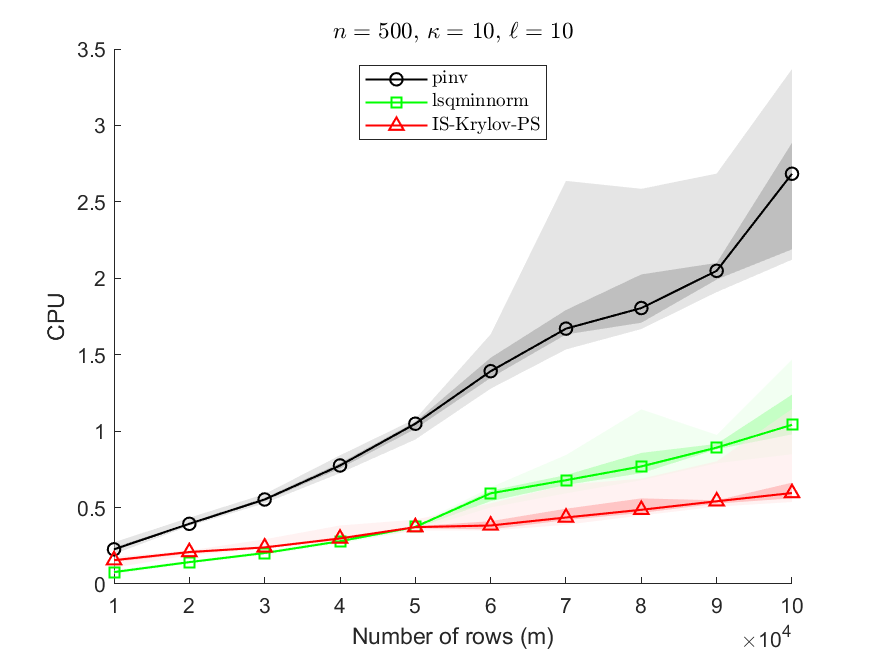}
	\includegraphics[width=0.32\linewidth]{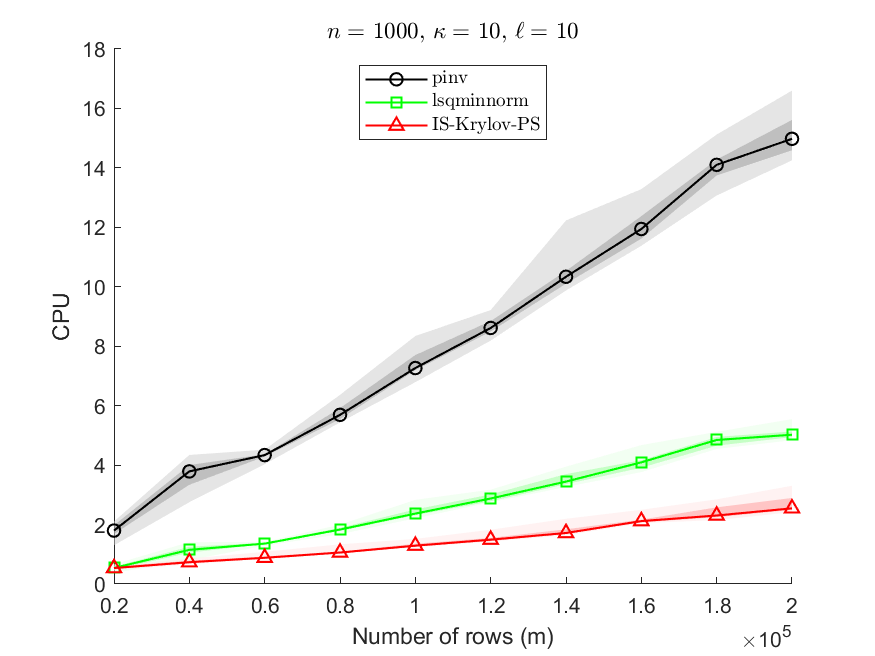}
	\includegraphics[width=0.32\linewidth]{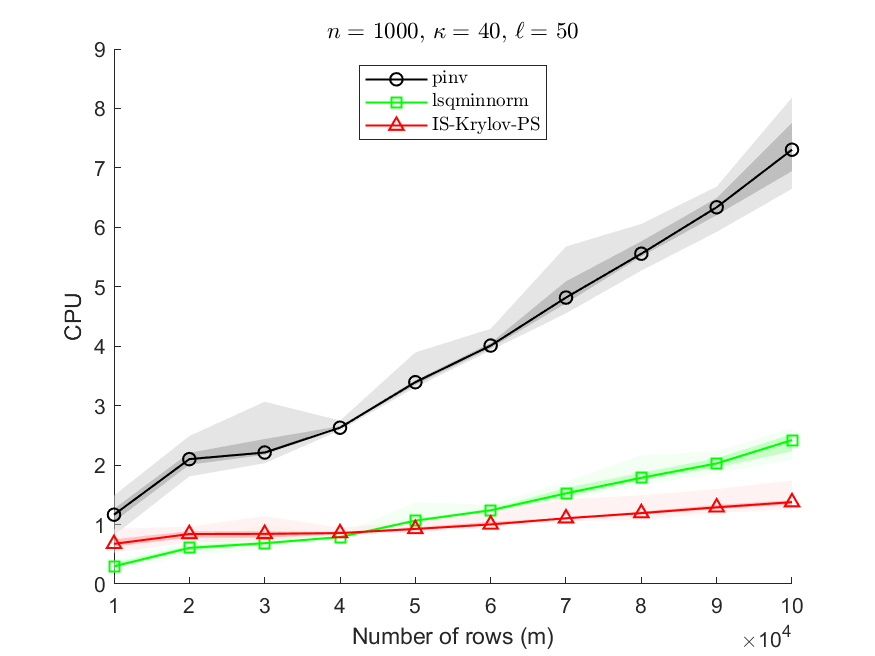}
	\caption{Figures depict the CPU time (in seconds) versus increasing number of rows. The title of each plot indicates the values of $n,\kappa$ and $\ell$. We set $q=50$. }
	\label{fig:5}
\end{figure}

\section{Concluding Remarks}\label{sec-concluding}

We introduced the affine subspace search into randomized iterative methods, establishing a new paradigm for solving linear systems. We proved that as the dimension of the affine subspace increases, the convergence factor decreases, leading to improved convergence results. When the subspace dimension becomes sufficiently large and the algorithm reduces to its deterministic form, it precisely coincides with Krylov subspace methods. Thus, we established a connection between randomized iterative methods and Krylov subspaces. We derived efficient implementation schemes for the algorithm and proposed a novel iterative-sketching-based Krylov subspace method. Numerical results confirmed the efficiency of our proposed method.

There are still many possible avenues for future research. For instance, theoretically determining the number of previous iterates \(\ell\). The randomized coordinate descent (RCD) method \cite{Lev10} could be seen as the dual form of the randomized Kaczmarz algorithm, capable of solving least squares problems regardless of whether the linear system is consistent or inconsistent. Establishing the connection between RCD and Krylov subspace methods is also a valuable topic.

\bibliographystyle{abbrv}
\bibliography{ref}

\begin{thebibliography}{10}

\bibitem{balabanov2022randomized}
Oleg Balabanov and Laura Grigori.
\newblock Randomized {Gram-Schmidt} process with application to {GMRES}.
\newblock {\em SIAM Journal on Scientific Computing}, 44(3):A1450--A1474, 2022.

\bibitem{bishop2006pattern}
Christopher~M Bishop and Nasser~M Nasrabadi.
\newblock {\em Pattern recognition and machine learning}.
\newblock Springer, 2006.

\bibitem{boyd2018introduction}
Stephen Boyd and Lieven Vandenberghe.
\newblock {\em Introduction to applied linear algebra: vectors, matrices, and
  least squares}.
\newblock Cambridge university press, 2018.

\bibitem{burke2025gmres}
Liam Burke, Stefan G{\"u}ttel, and Kirk~M Soodhalter.
\newblock Gmres with randomized sketching and deflated restarting.
\newblock {\em SIAM Journal on Matrix Analysis and Applications},
  46(1):702--725, 2025.

\bibitem{chang2011libsvm}
Chih-Chung Chang and Chih-Jen Lin.
\newblock Libsvm: a library for support vector machines.
\newblock {\em ACM transactions on intelligent systems and technology (TIST)},
  2(3):1--27, 2011.

\bibitem{charikar2004finding}
Moses Charikar, Kevin Chen, and Martin Farach-Colton.
\newblock Finding frequent items in data streams.
\newblock {\em Theor. Comput. Sci.}, 312(1):3--15, 2004.

\bibitem{chen2021regularized}
Xuemei Chen and Jing Qin.
\newblock Regularized {K}aczmarz algorithms for tensor recovery.
\newblock {\em SIAM Journal on Imaging Sciences}, 14(4):1439--1471, 2021.

\bibitem{clarkson2017low}
Kenneth~L Clarkson and David~P Woodruff.
\newblock Low-rank approximation and regression in input sparsity time.
\newblock {\em J. ACM}, 63(6):1--45, 2017.

\bibitem{cormode2005improved}
Graham Cormode and Shan Muthukrishnan.
\newblock An improved data stream summary: the count-min sketch and its
  applications.
\newblock {\em J. Algorithms}, 55(1):58--75, 2005.

\bibitem{derezinski2024recent}
Micha{\l} Derezi{\'n}ski and Michael~W Mahoney.
\newblock Recent and upcoming developments in randomized numerical linear
  algebra for machine learning.
\newblock In {\em Proceedings of the 30th ACM SIGKDD Conference on Knowledge
  Discovery and Data Mining}, pages 6470--6479, 2024.

\bibitem{fino1976unified}
Fino and Algazi.
\newblock Unified matrix treatment of the fast walsh-hadamard transform.
\newblock {\em IEEE Trans. Computers}, 100(11):1142--1146, 1976.

\bibitem{gazagnadou2022ridgesketch}
Nidham Gazagnadou, Mark Ibrahim, and Robert~M Gower.
\newblock {RidgeSketch}: a fast sketching based solver for large scale ridge
  regression.
\newblock {\em SIAM Journal on Matrix Analysis and Applications},
  43(3):1440--1468, 2022.

\bibitem{gearhart1989acceleration}
William~B Gearhart and Mathew Koshy.
\newblock Acceleration schemes for the method of alternating projections.
\newblock {\em Journal of Computational and Applied Mathematics},
  26(3):235--249, 1989.

\bibitem{golub2013matrix}
Gene~H Golub and Charles~F Van~Loan.
\newblock {\em Matrix computations}.
\newblock John Hopkins University Press, Baltimore, MD, 4th edition, 2012.

\bibitem{gordon1970algebraic}
Richard Gordon, Robert Bender, and Gabor~T. Herman.
\newblock Algebraic reconstruction techniques ({ART}) for three-dimensional
  electron microscopy and {X}-ray photography.
\newblock {\em J. Theor. Biol.}, 29(3):471--481, December 1970.

\bibitem{Gow15}
Robert~M. Gower and Peter Richt{\'a}rik.
\newblock Randomized iterative methods for linear systems.
\newblock {\em SIAM J. Matrix Anal. Appl.}, 36(4):1660--1690, 2015.

\bibitem{guttel2024sketch}
Stefan G{\"u}ttel and Igor Simunec.
\newblock A sketch-and-select {A}rnoldi process.
\newblock {\em SIAM Journal on Scientific Computing}, 46(4):A2774--A2797, 2024.

\bibitem{han2024randomized}
Deren Han, Yansheng Su, and Jiaxin Xie.
\newblock Randomized {Douglas--Rachford} methods for linear systems: {I}mproved
  accuracy and efficiency.
\newblock {\em SIAM Journal on Optimization}, 34(1):1045--1070, 2024.

\bibitem{Han2022-xh}
Deren Han, Yansheng Su, and Jiaxin Xie.
\newblock Randomized {Douglas-Rachford} methods for linear systems: {I}mproved
  accuracy and efficiency.
\newblock {\em SIAM J. Optim.}, 34(1):1045--1070, 2024.

\bibitem{han2022pseudoinverse}
Deren Han and Jiaxin Xie.
\newblock On pseudoinverse-free randomized methods for linear systems:
  {U}nified framework and acceleration.
\newblock {\em arXiv preprint arXiv:2208.05437}, 2022.

\bibitem{hastie2009elements}
Trevor Hastie, Robert Tibshirani, and Jerome Friedman.
\newblock The elements of statistical learning: data mining, inference, and
  prediction, 2009.

\bibitem{he2024inertial}
Songnian He, Ziting Wang, and Qiao-Li Dong.
\newblock Inertial randomized kaczmarz algorithms for solving coherent linear
  systems.
\newblock {\em Numerical Algorithms}, pages 1--31, 2024.

\bibitem{hefny2017rows}
Ahmed Hefny, Deanna Needell, and Aaditya Ramdas.
\newblock Rows versus columns: Randomized {K}aczmarz or {Gauss--Seidel} for
  ridge regression.
\newblock {\em SIAM Journal on Scientific Computing}, 39(5):S528--S542, 2017.

\bibitem{hegland2023generalized}
Markus Hegland and Janosch Rieger.
\newblock Generalized {Gearhart-Koshy} acceleration is a {K}rylov space method
  of a new type.
\newblock {\em arXiv preprint arXiv:2311.18305}, 2023.

\bibitem{herman1993algebraic}
Gabor~T Herman and Lorraine~B Meyer.
\newblock Algebraic reconstruction techniques can be made computationally
  efficient (positron emission tomography application).
\newblock {\em IEEE Trans. Medical Imaging}, 12(3):600--609, 1993.

\bibitem{huang2024randomized}
Longxiu Huang, Xia Li, and Deanna Needell.
\newblock Randomized {K}aczmarz in adversarial distributed setting.
\newblock {\em SIAM Journal on Scientific Computing}, 46(3):B354--B376, 2024.

\bibitem{huang2022linear}
Meng Huang and Yang Wang.
\newblock Linear convergence of randomized {K}aczmarz method for solving
  complex-valued phaseless equations.
\newblock {\em SIAM Journal on Imaging Sciences}, 15(2):989--1016, 2022.

\bibitem{jeong2023linear}
Halyun Jeong and Deanna Needell.
\newblock Linear convergence of reshuffling {K}aczmarz methods with sparse
  constraints.
\newblock {\em arXiv preprint arXiv:2304.10123, to appear in SIAM Journal on
  Scientific Computing}, 2023.

\bibitem{johnson1984extensions}
William~B Johnson, Joram Lindenstrauss, et~al.
\newblock Extensions of lipschitz mappings into a hilbert space.
\newblock {\em Contemporary mathematics}, 26(189-206):1, 1984.

\bibitem{Kac37}
S~Karczmarz.
\newblock Angen{\"{a}}herte aufl{\"{o}}sung von systemen linearer glei-chungen.
\newblock {\em Bull. Int. Acad. Pol. Sic. Let., Cl. Sci. Math. Nat.}, pages
  355--357, 1937.

\bibitem{kolodziej2019suitesparse}
Scott~P Kolodziej, Mohsen Aznaveh, Matthew Bullock, Jarrett David, Timothy~A
  Davis, Matthew Henderson, Yifan Hu, and Read Sandstrom.
\newblock The suitesparse matrix collection website interface.
\newblock {\em Journal of Open Source Software}, 4(35):1244, 2019.

\bibitem{lan2020first}
Guanghui Lan.
\newblock {\em First-order and stochastic optimization methods for machine
  learning}.
\newblock Springer, 2020.

\bibitem{Lev10}
Dennis Leventhal and Adrian~S Lewis.
\newblock Randomized methods for linear constraints: convergence rates and
  conditioning.
\newblock {\em Math. Oper. Res.}, 35(3):641--654, 2010.

\bibitem{liesen2013krylov}
J{\"o}rg Liesen and Zdenek Strakos.
\newblock {\em Krylov subspace methods: principles and analysis}.
\newblock Numerical Mathematics and Scie, 2013.

\bibitem{liu2016accelerated}
Ji~Liu and Stephen Wright.
\newblock An accelerated randomized {K}aczmarz algorithm.
\newblock {\em Math. Comp.}, 85(297):153--178, 2016.

\bibitem{liu2021subspace}
Xin Liu, Zaiwen Wen, and Ya-Xiang Yuan.
\newblock Subspace methods for nonlinear optimization.
\newblock {\em CSIAM Trans. Appl. Math.}, 2(4):585--651, 2021.

\bibitem{loizou2020momentum}
Nicolas Loizou and Peter Richt{\'a}rik.
\newblock Momentum and stochastic momentum for stochastic gradient, newton,
  proximal point and subspace descent methods.
\newblock {\em Computational Optimization and Applications}, 77(3):653--710,
  2020.

\bibitem{lorenz2023minimal}
Dirk~A Lorenz and Maximilian Winkler.
\newblock Minimal error momentum bregman-kaczmarz.
\newblock {\em Linear Algebra and its Applications}, 709:416--448, 2025.

\bibitem{ma2022randomized}
Anna Ma and Denali Molitor.
\newblock Randomized {K}aczmarz for tensor linear systems.
\newblock {\em BIT Numerical Mathematics}, 62(1):171--194, 2022.

\bibitem{ma2017stochastic}
Anna Ma and Deanna Needell.
\newblock Stochastic gradient descent for linear systems with missing data.
\newblock {\em Numer. Math. Theory Methods Appl.}, 12(1):1--20, 2019.

\bibitem{martinsson2020randomized}
Per-Gunnar Martinsson and Joel~A Tropp.
\newblock Randomized numerical linear algebra: Foundations and algorithms.
\newblock {\em Acta Numerica}, 29:403--572, 2020.

\bibitem{meng2014lsrn}
Xiangrui Meng, Michael~A Saunders, and Michael~W Mahoney.
\newblock {LSRN: A} parallel iterative solver for strongly over-or
  underdetermined systems.
\newblock {\em SIAM J. Sci. Comput.}, 36(2):C95--C118, 2014.

\bibitem{nakatsukasa2024fast}
Yuji Nakatsukasa and Joel~A Tropp.
\newblock Fast and accurate randomized algorithms for linear systems and
  eigenvalue problems.
\newblock {\em SIAM Journal on Matrix Analysis and Applications},
  45(2):1183--1214, 2024.

\bibitem{necoara2019faster}
Ion Necoara.
\newblock Faster randomized block kaczmarz algorithms.
\newblock {\em SIAM Journal on Matrix Analysis and Applications},
  40(4):1425--1452, 2019.

\bibitem{needell2014stochasticMP}
Deanna Needell, Nathan Srebro, and Rachel Ward.
\newblock Stochastic gradient descent, weighted sampling, and the randomized
  {K}aczmarz algorithm.
\newblock {\em Math. Program.}, 155:549--573, 2016.

\bibitem{nesterov2003introductory}
Yurii Nesterov.
\newblock {\em Introductory lectures on convex optimization: A basic course},
  volume~87.
\newblock Springer Science \& Business Media, Berlin, 2003.

\bibitem{nesterov1983method}
Yurii~E Nesterov.
\newblock A method for solving the convex programming problem with convergence
  rate {O}$(1/k^2)$.
\newblock In {\em Dokl. akad. nauk Sssr}, volume 269, pages 543--547, 1983.

\bibitem{nocedal1999numerical}
Jorge Nocedal and Stephen~J Wright.
\newblock {\em Numerical optimization}.
\newblock Springer, 1999.

\bibitem{pilanci2016iterative}
Mert Pilanci and Martin~J Wainwright.
\newblock Iterative hessian sketch: Fast and accurate solution approximation
  for constrained least-squares.
\newblock {\em Journal of Machine Learning Research}, 17(53):1--38, 2016.

\bibitem{polyak1964some}
Boris~T Polyak.
\newblock Some methods of speeding up the convergence of iteration methods.
\newblock {\em Comput. Math. Math. Phys.}, 4(5):1--17, 1964.

\bibitem{raskutti2016statistical}
Garvesh Raskutti and Michael~W Mahoney.
\newblock A statistical perspective on randomized sketching for ordinary
  least-squares.
\newblock {\em Journal of Machine Learning Research}, 17(213):1--31, 2016.

\bibitem{rieger2023generalized}
Janosch Rieger.
\newblock Generalized gearhart-koshy acceleration for the kaczmarz method.
\newblock {\em Mathematics of Computation}, 92(341):1251--1272, 2023.

\bibitem{robbins1951stochastic}
Herbert Robbins and Sutton Monro.
\newblock A stochastic approximation method.
\newblock {\em Ann. Math. Statistics}, pages 400--407, 1951.

\bibitem{saad2003iterative}
Yousef Saad.
\newblock {\em Iterative methods for sparse linear systems}.
\newblock SIAM, 2003.

\bibitem{schopfer2019linear}
Frank Sch{\"o}pfer and Dirk~A Lorenz.
\newblock Linear convergence of the randomized sparse {K}aczmarz method.
\newblock {\em Math. Program.}, 173(1):509--536, 2019.

\bibitem{strohmer2009randomized}
Thomas Strohmer and Roman Vershynin.
\newblock A randomized {K}aczmarz algorithm with exponential convergence.
\newblock {\em Journal of Fourier Analysis and Applications}, 15(2):262--278,
  2009.

\bibitem{su2024greedy}
Yansheng Su, Deren Han, Yun Zeng, and Jiaxin Xie.
\newblock On greedy multi-step inertial randomized kaczmarz method for solving
  linear systems.
\newblock {\em Calcolo}, 61(4):68, 2024.

\bibitem{tam2021gearhart}
Matthew~K Tam.
\newblock {Gearhart--Koshy} acceleration for affine subspaces.
\newblock {\em Operations Research Letters}, 49(2):157--163, 2021.

\bibitem{tan2019phase}
Yan~Shuo Tan and Roman Vershynin.
\newblock Phase retrieval via randomized {K}aczmarz: theoretical guarantees.
\newblock {\em Information and Inference: A Journal of the IMA}, 8(1):97--123,
  2019.

\bibitem{tropp2011improved}
Joel~A Tropp.
\newblock Improved analysis of the subsampled randomized hadamard transform.
\newblock {\em Adv. Adapt. Data Anal.,}, 3(01n02):115--126, 2011.

\bibitem{woolfe2008fast}
Franco Woolfe, Edo Liberty, Vladimir Rokhlin, and Mark Tygert.
\newblock A fast randomized algorithm for the approximation of matrices.
\newblock {\em Appl. Comput. Harmon. Anal.}, 25(3):335--366, 2008.

\bibitem{xiang2025randomized}
Ruike Xiang, Jiaxin Xie, and Qiye Zhang.
\newblock Randomized block kaczmarz with volume sampling: Momentum acceleration
  and efficient implementation.
\newblock {\em arXiv preprint arXiv:2503.13941}, 2025.

\bibitem{xie2024randomized}
Jiaxin Xie, Houduo Qi, and Deren Han.
\newblock Randomized iterative methods for generalized absolute value
  equations: {S}olvability and error bounds.
\newblock {\em arXiv preprint arXiv:2405.04091}, 2024.

\bibitem{yuan2014review}
Ya-xiang Yuan.
\newblock A review on subspace methods for nonlinear optimization.
\newblock In {\em Proceedings of the International Congress of Mathematics},
  pages 807--827, 2014.

\bibitem{yuan2022adaptively}
Zi-Yang Yuan, Lu~Zhang, Hongxia Wang, and Hui Zhang.
\newblock Adaptively sketched {B}regman projection methods for linear systems.
\newblock {\em Inverse Problems}, 38(6):065005, 2022.

\bibitem{zeng2023randomized}
Yun Zeng, Deren Han, Yansheng Su, and Jiaxin Xie.
\newblock Randomized {K}aczmarz method with adaptive stepsizes for inconsistent
  linear systems.
\newblock {\em Numer. Algor.}, 94:1403--1420, 2023.

\bibitem{zeng2024adaptive}
Yun Zeng, Deren Han, Yansheng Su, and Jiaxin Xie.
\newblock On adaptive stochastic heavy ball momentum for solving linear
  systems.
\newblock {\em SIAM Journal on Matrix Analysis and Applications},
  45(3):1259--1286, 2024.

\end{thebibliography}
\section{Appendix. Proof of the main results}
\label{sec:appd}

\subsection{Proof of Lemma \ref{lemma-effup}}

The following  lemma is useful for proving Lemma \ref{lemma-effup}.
\begin{lemma}[\cite{rieger2023generalized}, Lemma 12]
	\label{lemma-inv}
	Suppose the matrix $ B=\sum_{j=1}^n\alpha_j\mathbbm{1}_n^j(\mathbbm{1}_n^j)^\top$ with $\alpha_j\neq0$ for $j=1,\ldots,n$. Then the matrix \(C(\alpha_1, \ldots, \alpha_n)\), as defined in equation \eqref{matrixC}, is the inverse of \(B\).

\end{lemma}

Now, we are ready to prove  Lemma \ref{lemma-effup}.
\begin{proof}[Proof of Lemma \ref{lemma-effup}]
From Lemma \ref{lemma-inv}, we know that to prove the inverse matrix of \(V_k^\top V_k\) can be expressed as
$
C_k = C(\gamma_{j_k}\underline{s_{j_k}}, \ldots, \gamma_{k-1}\underline{s_{k-1}}) \in \mathbb{R}^{(k-j_k) \times (k-j_k)},
$
it is sufficient to prove that
	$$
			V_{k}^\top V_{k}=\sum_{i=1}^{k-j_{k}}\gamma_{j_{k}+i-1}\underline{s_{j_{k}+i-1}}\mathbbm{1}_{k-j_{k}}^{i}(\mathbbm{1}_{k-j_{k}}^{i})^\top.
		$$
Our proof can be classified into two cases: $j_k = j_{k+1} = 0$ and $j_k < j_{k+1}$.
		
{\bf Case 1} ($j_k = j_{k+1} = 0$). In this case, we know that now $k \leq \ell-2$. If $k=0$, we have $V_0=\emptyset$ and $V_{1}=\begin{pmatrix}
		x^0-x^1
	\end{pmatrix}$. Then $V_1^\top V_1=\lVert x^0-x^1\rVert_2^2=\lVert M_0s_0\rVert_2^2=\gamma_0\underline{s_0}$.
	
	If $1\leq k\leq \ell-2$, we have $V_k = (x^0 - x^k, x^1 - x^k, \ldots, x^{k-1} - x^k) \in \mathbb{R}^{n \times k}$, $M_k=\begin{pmatrix}
		V_k,-A^\top S_kS_k^\top(Ax^k-b)
	\end{pmatrix}\in \mathbb{R}^{n \times (k+1)}$, and $V_{k+1} = (x^0 - x^{k+1}, x^1 - x^{k+1}, \ldots, x^{k} - x^{k+1}) \in \mathbb{R}^{n \times (k+1)}$. 
	For the first $k$-order principal submatrix of $V_{k+1}^\top V_{k+1}$ $(i,j = 1,2,\ldots,k)$, we have
	\begin{equation*}
		\begin{aligned}
			(V_{k+1}^\top V_{k+1})_{i,j}&=\langle x^{i-1}-x^{k+1},x^{j-1}-x^{k+1}\rangle=\langle x^{i-1}-x^k-M_ks_k,x^{j-1}-x^k-M_ks_k\rangle\\
			&=\langle x^{i-1}-x^k,x^{j-1}-x^k\rangle-(x^{i-1}-x^k)^\top M_ks_k-(x^{j-1}-x^k)^\top M_ks_k+s_k^\top M_k^\top M_ks_k\\
			&=\langle x^{i-1}-x^k,x^{j-1}-x^k\rangle-(e^i_{k+1})^\top M_k^\top M_ks_k-(e^j_{k+1})^\top M_k^\top M_ks_k+\gamma_ks_k^\top e_{k+1}^{k+1}\\
			&\overset{(a)}{=}\langle x^{i-1}-x^k,x^{j-1}-x^k\rangle-\gamma_k(e^i_{k+1})^\top e_{k+1}^{k+1}-\gamma_k(e^j_{k+1})^\top e_{k+1}^{k+1}+\gamma_k\underline{s_k}\\
			&=\langle x^{i-1}-x^k,x^{j-1}-x^k\rangle+\gamma_k\underline{s_k},\\
		\end{aligned}
	\end{equation*}
	where  $(a)$ follows from  \eqref{eq-s_k}, i.e., $M_k^\top M_ks_k=\gamma_ke_{k+1}^{k+1}$.
	For the last row of $V_{k+1}^\top V_{k+1}$ $(j=1,2,\ldots,k)$, we have
	\begin{equation*}
		\begin{aligned}
			(V_{k+1}^\top V_{k+1})_{k+1,j}&=\langle x^k-x^{k+1},x^{j-1}-x^{k+1}\rangle=-\langle M_ks_k,x^{j-1}-x^k-M_ks_k\rangle\\
			&=-(x^{j-1}-x^k)^\top M_ks_k+s_k^\top M_k^\top M_ks_k=-(e^j_{k+1})^\top M_k^\top M_ks_k+\gamma_ks_k^\top e_{k+1}^{k+1}\\
			&\overset{(a)}{=}-\gamma_k(e^{j}_{k+1})^\top e_{k+1}^{k+1}+\gamma_k\underline{s_k}=\gamma_k\underline{s_k}.\\
		\end{aligned}
	\end{equation*}
	where $(a)$ follows from   \eqref{eq-s_k}, and  
	$$
	(V_{k+1}^\top V_{k+1})_{k+1,k+1}=\|x^{k+1}-x^k\|^2_2=s_k^\top M_k^\top M_ks_k=\gamma_k\underline{s_k}.
	$$
	Hence, we have 
	\begin{equation}\label{proof-xie-511}
		\begin{aligned}
			V_{k+1}^\top V_{k+1}&=\begin{pmatrix}
				V_k^\top V_k&0_{k\times 1}\\
				0_{1\times k}&0
			\end{pmatrix}+\gamma_k\underline{s_k}\mathbbm{1}_{k+1}^{k+1}(\mathbbm{1}_{k+1}^{k+1})^\top\\
			&=\begin{pmatrix}
				V_{k-1}^\top V_{k-1}&0_{(k-1)\times 2}\\
				0_{2\times (k-1)}&0_{2\times 2}
			\end{pmatrix}+\gamma_{k-1}\underline{s_{k-1}}\mathbbm{1}^{k}_{k+1}(\mathbbm{1}^{k}_{k+1})^\top+\gamma_k\underline{s_k}\mathbbm{1}_{k+1}^{k+1}(\mathbbm{1}_{k+1}^{k+1})^\top\\
			&=\begin{pmatrix}
				V_1^\top V_1&0_{1\times k}\\
				0_{k\times 1}&0_{k\times k}
			\end{pmatrix}+\gamma_1\underline{s_1}\mathbbm{1}_{k+1}^{2}(\mathbbm{1}_{k+1}^{2})^\top+\cdots+\gamma_k\underline{s_k}\mathbbm{1}_{k+1}^{k+1}(\mathbbm{1}_{k+1}^{k+1})^\top\\
			&=\sum_{i=1}^{k+1}\gamma_{i-1}\underline{s_{i-1}}\mathbbm{1}_{k+1}^{i}(\mathbbm{1}_{k+1}^{i})^\top
		\end{aligned}
	\end{equation}
	as desired.
	
{\bf Case 2} ($j_k < j_{k+1}$). In this case, we know that $k > \ell-2$, and  we have $V_k =\begin{pmatrix}
		x^{j_k} - x^k, \ldots, x^{k-1} - x^k
	\end{pmatrix} \in \mathbb{R}^{n \times (k-j_k)}$, 
	$M_k=\begin{pmatrix}
		V_k,-A^\top S_kS_k^\top(Ax^k-b)
	\end{pmatrix}\in \mathbb{R}^{n \times (k-j_k+1)}$
	and 
	$V_{k+1} = \begin{pmatrix}
		x^{j_{k+1}} - x^{k+1},\ldots, x^{k} - x^{k+1}
	\end{pmatrix} \in \mathbb{R}^{n \times (k-j_k)}$. 
	We define $\hat{V}_{k+1} := \begin{pmatrix}
		x^{j_k} - x^{k+1}, V_{k+1}
	\end{pmatrix} \in \mathbb{R}^{n \times (k-j_k+1)}$. 
	For the first $(k-j_k)$-order principal submatrix of $\hat{V}_{k+1}^\top \hat{V}_{k+1}$, $(i,j = 1,2,\ldots,k-j_k)$, we have
	\begin{equation*}
		\begin{aligned}
			&(\hat{V}_{k+1}^\top \hat{V}_{k+1})_{ij}=\langle x^{j_k+i-1}-x^{k+1},x^{j_k+j-1}-x^{k+1}\rangle\\
			&=\langle x^{j_k+i-1}-x^k-M_ks_k,x^{j_k+j-1}-x^k-M_ks_k\rangle\\
			&=\langle x^{j_k+i-1}-x^k,x^{j_k+j-1}-x^k\rangle-(x^{j_k+i-1}-x^k)^\top M_ks_k-(x^{j_k+j-1}-x^k)^\top M_ks_k+s_k^\top M_k^\top M_ks_k\\
			&=\langle x^{j_k+i-1}-x^k,x^{j_k+j-1}-x^k\rangle-(e^{i}_{k-j_k+1})^\top M_k^\top M_ks_k-(e^{j}_{k-j_k+1})^\top M_k^\top M_ks_k+\gamma_ks_k^\top e_{k-j_k+1}^{k-j_k+1}\\
			&\overset{(a)}{=}\langle x^{j_k+i-1}-x^k,x^{j_k+j-1}-x^k\rangle-\gamma_k(e^{i}_{k-j_k+1})^\top e_{k-j_k+1}^{k-j_k+1}-\gamma_k(e^{j}_{k-j_k+1})^\top e_{k-j_k+1}^{k-j_k+1}+\gamma_k\underline{s_k}\\
			&=\langle x^{j_k+i-1}-x^k,x^{j_k+j-1}-x^k\rangle+\gamma_k\underline{s_k},
		\end{aligned}
	\end{equation*}
		where  $(a)$ follows from  \eqref{eq-s_k}.
	For the last row of $\hat{V}_{k+1}^\top \hat{V}_{k+1}$ $(j=1,2,\ldots,k-j_k)$, we have
	\begin{equation*}
		\begin{aligned}
			(\hat{V}_{k+1}^\top \hat{V}_{k+1})_{k-j_k+1,j}&=\langle x^k-x^{k+1},x^{j_k+j-1}-x^{k+1}\rangle=-\langle M_ks_k,x^{j_k+j-1}-x^k-M_ks_k\rangle\\
			&=-(x^{j_k+j-1}-x^k)^\top M_ks_k+s_k^\top M_k^\top M_ks_k\\
			&=-(e^{j}_{k-j_k+1})^\top M_k^\top M_ks_k+\gamma_ks_k^\top e_{k-j_k+1}^{k-j_k+1}\\
			&=-\gamma_k(e^{j}_{k-j_k+1})^\top e_{k-j_k+1}^{k-j_k+1}+\gamma_k\underline{s_k}\\&
			=\gamma_k\underline{s_k}
		\end{aligned}
	\end{equation*}
	and 
	$$
	(\hat{V}_{k+1}^\top \hat{V}_{k+1})_{k-j_k+1,k-j_k+1}=\|x^{k+1}-x^k\|^2_2=s_k^\top M_k^\top M_ks_k=\gamma_k\underline{s_k}.
	$$
	
	 Consider the case where $k=\ell-1$. From \eqref{proof-xie-511}, we have $V_k^\top V_k=\sum_{i=1}^{k}\gamma_{i-1}\underline{s_{i-1}}\mathbbm{1}_{k}^{i}(\mathbbm{1}_{k}^{i})^\top$ Note that now $j_k=0$, we can get 
	\begin{equation*}
		\begin{aligned}
			\hat{V}_{k+1}^\top \hat{V}_{k+1}&=\begin{pmatrix}
				V_k^\top V_k&0_{(k-j_k)\times 1}\\
				0_{1\times (k-j_k)}&0
			\end{pmatrix}+\gamma_k\underline{s_k}\mathbbm{1}_{k-j_k+1}^{k-j_k+1}(\mathbbm{1}_{k-j_k+1}^{k-j_k+1})^\top\\
			&=\sum_{i=1}^{k-j_{k}}\gamma_{j_{k}+i-1}\underline{s_{j_{k}+i-1}}\begin{pmatrix}
				\mathbbm{1}_{k-j_{k}}^{i}\\
				0
			\end{pmatrix}\begin{pmatrix}
				(\mathbbm{1}_{k-j_{k}}^{i})^\top&0
			\end{pmatrix}+\gamma_k\underline{s_k}\mathbbm{1}_{k-j_k+1}^{k-j_k+1}(\mathbbm{1}_{k-j_k+1}^{k-j_k+1})^\top\\
			&=\sum_{i=1}^{k-j_{k}+1}\gamma_{j_{k}+i-1}\underline{s_{j_{k}+i-1}}\mathbbm{1}_{k-j_{k}+1}^{i}(\mathbbm{1}_{k-j_{k}+1}^{i})^\top,\\
		\end{aligned}
	\end{equation*}
	On the other hand, since $\hat{V}_{k+1}^\top \hat{V}_{k+1}$ can also be expressed as
	$$\hat{V}_{k+1}^\top \hat{V}_{k+1}=\begin{pmatrix}
		\lVert x^{j_k}-x^{k+1}\rVert_2^2&(x^{j_k}-x^{k+1})^\top V_{k+1}\\
		V_{k+1}^\top(x^{j_k}-x^{k+1})&V_{k+1}^\top V_{k+1},
	\end{pmatrix}$$ by component-wise comparison, we have
	\begin{equation*}
		\begin{aligned}
			V_{k+1}^\top V_{k+1}&=\sum_{i=2}^{k-j_{k}+1}\gamma_{j_{k}+i-1}\underline{s_{j_{k}+i-1}}\mathbbm{1}_{k-j_{k}}^{i-1}(\mathbbm{1}_{k-j_{k}}^{i-1})^\top,\\
			&=\sum_{i=1}^{k+1-j_{k+1}}\gamma_{j_{k+1}+i-1}\underline{s_{j_{k+1}+i-1}}\mathbbm{1}_{k+1-j_{k+1}}^{i}(\mathbbm{1}_{k+1-j_{k+1}}^{i})^\top
		\end{aligned}
	\end{equation*}
	as desired.
By  induction, we conclude that for all \(k \geq \ell\), the matrix \(V_{k+1}^\top V_{k+1}\) can also admit the desired representation.
	This completes the proof of this lemma.
\end{proof}

\subsection{Proof of Lemmas \ref{Thm-Pi_k-Krylov} and  \ref{xie-equ-krylov}}

 \begin{proof}[Proof of Lemma \ref{Thm-Pi_k-Krylov}]
	First, we prove that \(\Pi_k = \Theta_k\). By the iteration scheme \eqref{xk1-scheme}, we know that
	\[
	x^{k+1} = x^{k} + V_k \overline{s_k} - \underline{s_k} A^\top S_k S_k^\top r^k \in \Pi_k,
	\]
	which implies that all generating points of \(\Theta_k\) belong to \(\Pi_k\), thus \(\Theta_k \subseteq \Pi_k\).
	Conversely, for any \(x \in \Pi_k\), there exist coefficients \(\{\lambda_i\}_{i = 1}^{k-j_k+1}\) such that
	\[
	\begin{aligned}
		x &= x^k + \sum_{i=j_k}^{k-1} \lambda_{i-j_k+1}(x^i - x^k) - \lambda_{k-j_k+1} A^\top S_k S_k^\top r^k \\
		&= x^k + \sum_{i=j_k}^{k-1} \lambda_{i-j_k+1}(x^i - x^k) + \frac{\lambda_{k-j_k+1}}{\underline{s_k}}(x^{k+1} - x^k - V_k \overline{s_k}) \\
		&\in \Theta_k,
	\end{aligned}
	\]
	where the second equality follows from \eqref{xk1-scheme}. By the arbitrariness of \(x\), we conclude \(\Pi_k \subseteq \Theta_k\). Hence, \(\Pi_k = \Theta_k\).
	
	Next, we show that \(\Theta_k = \Lambda_k\) when \(\ell = \infty\). In this case, \(j_k = 0\). By the iteration scheme \eqref{xk1-scheme}, we know that for any \(i = 0, 1, \ldots, k\),
	\[
	\begin{aligned}
		A^\top S_i S_i^\top r^i &= \frac{1}{\underline{s_i}}\left(-x^{i+1} + x^i + V_i \overline{s_i}\right) \\
		&\in \text{span}\{x^1 - x^0, \ldots, x^{i+1} - x^0\} \\
		&\subseteq \text{span}\{x^1 - x^0, \ldots, x^{k+1} - x^0\},
	\end{aligned}
	\]
	which indicates that \(\Lambda_k \subseteq \Theta_k\).
	Thus, we only need to prove that \(\Theta_k \subseteq \Lambda_k\). When \(k = 0\), we know that 
	\[
	\Theta_0 = x^0 + \operatorname{span}\{x^1 - x^0\} = x^0 + \operatorname{span}\{A^\top S_0 S_0^\top (Ax^0 - b)\} = \Lambda_0.
	\]
	By induction, assume \(\Theta_i \subseteq \Lambda_i\) for \(i = 0, \ldots, k-1\). Hence, for any \(i = 0, \ldots, k-1\),
	\[
	\begin{aligned}
		x^{i+1} &\in x^0 + \text{span}\{A^\top S_0 S_0^\top r^0, \ldots, A^\top S_i S_i^\top r^i\} \\
		&\subseteq x^0 + \text{span}\{A^\top S_0 S_0^\top r^0, \ldots, A^\top S_{k-1} S_{k-1}^\top r^{k-1}\}.
	\end{aligned}
	\]
	Thus, there exist coefficients \(\{\lambda_j^i\}_{j=0}^i\) such that \(x^{i+1} = x^0 + \sum_{j=0}^{i} \lambda_{j}^{i} A^\top S_j S_j^\top r^j\). For any \(x \in \Pi_k=\Theta_k\), there exist coefficients \(\{\mu_i\}_{i=1}^{k+1}\) such that 
	\[
	\begin{aligned}
		x &= x^0 + \sum_{i=1}^{k} \mu_i (x^i - x^0) + \mu_{k+1} (x^k - x^0 - A^\top S_k S_k^\top r^k) \\
		&= x^0 + \sum_{i=1}^{k-1} \mu_i (x^i - x^0) + (\mu_k + \mu_{k+1})(x^k - x^0) - \mu_{k+1} A^\top S_k S_k^\top r^k \\
		&= x^0 + \sum_{i=1}^{k-1} \sum_{j=0}^{i-1} \mu_i \lambda_{j}^{i-1} A^\top S_j S_j^\top r^j + (\mu_k + \mu_{k+1}) \sum_{j=0}^{k-1} \lambda_{j}^{k-1} A^\top S_j S_j^\top r^j - \mu_{k+1} A^\top S_k S_k^\top r^k \\
		&\in \Lambda_k.
	\end{aligned}
	\]
	Hence, \(\Theta_k \subseteq \Lambda_k\). This completes the proof of the theorem.
\end{proof}

   \begin{proof}[Proof of Lemma \ref{xie-equ-krylov}]
	For convenience, let \(\mathcal{H}_k := \text{span}\{A^\top r^0, \ldots, A^\top r^k\}\). If \(k = 0\), then \(\mathcal{H}_0 = \operatorname{span}\{A^\top r^0\} = \mathcal{K}_{1}(A^\top A, A^\top r^0)\). Next, we consider the case where \(k \geq 1\).
	
	First, we demonstrate that \(\mathcal{H}_k \subseteq \mathcal{K}_{k+1}(A^\top A, A^\top r^0)\). By induction, assume \(\mathcal{H}_i \subseteq \mathcal{K}_{i+1}(A^\top A, A^\top r^0)\) for \(i = 0, \ldots, k-1\). According to Lemma \ref{Thm-Pi_k-Krylov}, we have 
	\[\mathcal{H}_i = \Theta_i - x^0 = \operatorname{span}\{x^1 - x^0, \ldots, x^{i+1} - x^0\}.\]
	Thus, \(x^{i+1} - x^0 \in \mathcal{K}_{i+1}(A^\top A, A^\top r^0)\). Hence, for any \(i = 0, \ldots, k-1\), there exist coefficients \(\{\lambda_j^i\}_{j=0}^i\) such that \(x^{i+1} - x^0 = \sum_{j=0}^{i}\lambda^i_j (A^\top A)^j A^\top r^0\), and
	\begin{equation}\label{proof-xie-0328-l1}
		\begin{aligned}
			A^\top r^{i+1} &= A^\top r^0 + A^\top A(x^{i+1} - x^0) = A^\top r^0 + \sum_{j=0}^i \lambda^i_j (A^\top A)^{j+1} A^\top r^0 \\
			&\in \mathcal{K}_{i+2}(A^\top A, A^\top r^0) \subseteq \mathcal{K}_{k+1}(A^\top A, A^\top r^0).
		\end{aligned}
	\end{equation}
	This implies that all generating vectors of \(\mathcal{H}_k\) belong to \(\mathcal{K}_{k+1}(A^\top A, A^\top r^0)\). Therefore, \(\mathcal{H}_k \subseteq \mathcal{K}_{k+1}(A^\top A, A^\top r^0)\).
	
	Next, we show that \(\mathcal{K}_{k+1}(A^\top A, A^\top r^0) \subseteq \mathcal{H}_k\). By induction, assume \(\mathcal{K}_{i+1}(A^\top A, A^\top r^0) \subseteq \mathcal{H}_i\) for \(i = 0, \ldots, k-1\). Since we have already proven that for any \(s \geq 0\), it holds that \(\mathcal{H}_s \subseteq \mathcal{K}_{s+1}(A^\top A, A^\top r^0)\). Therefore, for any \(x^{i+1} - x^0 \in \mathcal{H}_i\), we have \(x^{i+1} - x^0 \in \mathcal{K}_{i+1}(A^\top A, A^\top r^0)\). Thus, there exist coefficients \(\{\lambda_j^i\}_{j=0}^i\) such that \(x^{i+1} - x^0 = \sum_{j=0}^{i} \lambda^i_j (A^\top A)^j A^\top r^0\). Since \(x^1-x^0, \ldots, x^k-x^0\) are linearly independent, we conclude that \(\lambda_i^i \neq 0\) for \(i = 1, 2, \ldots, k-1\). Otherwise, if there exists an index \(i_0\) such that \(\lambda_{i_0}^{i_0} = 0\), then
	\[
	x^{i_0+1} - x^0 = \sum_{j=0}^{i_0-1} \lambda^{i_0}_j (A^\top A)^j A^\top r^0 \in \mathcal{K}_{i_0}(A^\top A, A^\top r^0) \subseteq \mathcal{H}_{i_0-1},
	\]
	which contradicts the linear independence of \(x^1-x^0, \ldots, x^{i_0}-x^0\). Here, we define \(\mathcal{H}_{-1} = \{0\}\).
	For any \(x \in \mathcal{K}_{k+1}(A^\top A, A^\top r^0)\), we know that there exist coefficients \(\{\mu_j\}_{j=0}^k\) such that
	\[
	\begin{aligned}
		x &= \sum_{j=0}^{k-1} \mu_j (A^\top A)^j A^\top r^0 + \mu_k (A^\top A)^k A^\top r^0 \\
		&= \sum_{j=0}^{k-1} \mu_j (A^\top A)^j A^\top r^0 + \frac{\mu_k}{\lambda_{k-1}^{k-1}} \left(A^\top r^k - A^\top r^0 - \sum_{j=0}^{k-2} \lambda^i_j (A^\top A)^{j+1} A^\top r^0\right) \\
		&\in \text{span}\{A^\top r^0, \ldots, A^\top r^k\},
	\end{aligned}
	\]
	where the second equality  follows from  \eqref{proof-xie-0328-l1} where we have $$(A^\top A)^k A^\top r^0=\frac{1}{\lambda_{k-1}^{k-1}}\left(A^\top r^k-A^\top r^0-\sum_{j=0}^{k-2}\lambda^i_j (A^\top A)^{j+1} A^\top r^0\right),$$
	which is well-defined as $\lambda_{k-1}^{k-1}\neq 0$.
	By the arbitrariness of \(x\), we conclude \(\mathcal{K}_{k+1}(A^\top A, A^\top r^0) \subseteq \mathcal{H}_k\). Hence, we can conclude that \(\mathcal{H}_k = \mathcal{K}_{k+1}(A^\top A, A^\top r^0)\). This completes the proof of the lemma.
\end{proof}

\subsection{Proof of Theorem \ref{thm-eq-0330} and Proposition \ref{prop-property-Krylov}}
  \begin{proof}[Proof of Theorem \ref{thm-eq-0330}]
	By the iteration scheme \eqref{xk1-scheme}, we know
	\begin{equation*}
		\begin{aligned}
			x^{k+1}&=x^k+V_k\overline{s_k}-\underline{s_k}d_k=x^k-\underline{s_k}V_k(V_k^\top V_k)^{-1}V_k^\top d_k+\underline{s_k}d_k.
		\end{aligned}
	\end{equation*}
	Let $\tilde{p}_k:=x^{k+1}-x^k$. Then we have 
	\begin{equation*}
		\begin{aligned}
			\tilde{p}_{k+1}&=-\underline{s_{k+1}}V_{k+1}(V_{k+1}^\top V_{k+1})^{-1}V_{k+1}^\top d_{k+1}+\underline{s_{k+1}}d_{k+1}.
		\end{aligned}
	\end{equation*}
	Next, we will examine the the matrix $V_{k+1}(V_{k+1}^\top V_{k+1})^{-1}V_{k+1}^\top$. Note that $(V_{k+1}^\top V_{k+1})^{-1}=C(\gamma_{j_{k+1}}\underline{s_{j_{k+1}}},\cdots,\gamma_{k}\underline{s_k})$, we have
	\begin{equation*}
		\begin{aligned}
			&V_{k+1}(V_{k+1}^\top V_{k+1})^{-1}V_{k+1}^\top=\begin{pmatrix}
				x^{j_{k+1}}-x^{k+1}&\ldots&x^{k}-x^{k+1}
			\end{pmatrix}(V_{k+1}^\top V_{k+1})^{-1}
			\begin{pmatrix}
				(x^{j_{k+1}}-x^{k+1})^\top\\
				\vdots\\
				(x^{k}-x^{k+1})^\top\\
			\end{pmatrix}\\
			=&\frac{(x^{j_{k+1}}-x^{k+1})(x^{j_{k+1}}-x^{k+1})^\top}{\gamma_{j_{k+1}}\underline{s_{j_{k+1}}}}-\frac{(x^{j_{k+1}+1}-x^{k+1})(x^{j_{k+1}}-x^{k+1})^\top}{\gamma_{j_{k+1}}\underline{s_{j_{k+1}}}}\\
			&+\sum_{i=j_{k+1}+1}^{k}\left(-\frac{x^{i-1}-x^{k+1}}{\gamma_{i-1}\underline{s_{i-1}}}+\frac{x^{i}-x^{k+1}}{\gamma_{i-1}\underline{s_{i-1}}}+\frac{x^{i}-x^{k+1}}{\gamma_{i}\underline{s_{i}}}-\frac{x^{i+1}-x^{k+1}}{\gamma_{i}\underline{s_{i}}}\right)(x^{i}-x^{k+1})^\top\\
			=&\frac{(x^{j_{k+1}}-x^{k+1})(x^{j_{k+1}}-x^{k+1})^\top}{\gamma_{j_{k+1}}\underline{s_{j_{k+1}}}}-\frac{(x^{j_{k+1}+1}-x^{k+1})(x^{j_{k+1}}-x^{k+1})^\top}{\gamma_{j_{k+1}}\underline{s_{j_{k+1}}}}\\
			&+\sum_{i=j_{k+1}+1}^{k}\left(\frac{(x^{i}-x^{i-1})(x^{i}-x^{k+1})^\top}{\gamma_{i-1}\underline{s_{i-1}}}-\frac{(x^{i+1}-x^{i})(x^{i}-x^{k+1})^\top}{\gamma_{i}\underline{s_{i}}}\right)\\
			=&\frac{(x^{j_{k+1}}-x^{j_{k+1}+1})(x^{j_{k+1}}-x^{k+1})^\top}{\gamma_{j_{k+1}}\underline{s_{j_{k+1}}}}+\frac{(x^{j_{k+1}+1}-x^{j_{k+1}})(x^{j_{k+1}+1}-x^{k+1})^\top}{\gamma_{j_{k+1}}\underline{s_{j_{k+1}}}}\\
			&+\sum_{i=j_{k+1}+1}^{k-1}\frac{(x^{i+1}-x^{i})(x^{i+1}-x^{i})^\top}{\gamma_{i}\underline{s_{i}}}-\frac{(x^{k+1}-x^{k})(x^{k}-x^{k+1})^\top}{\gamma_{k}\underline{s_{k}}}\\
			=&\sum_{i=j_{k+1}}^{k}\frac{(x^{i+1}-x^{i})(x^{i+1}-x^{i})^\top}{\gamma_{i}\underline{s_{i}}}
			=\sum_{i=j_{k+1}}^{k}\frac{\tilde{p}_i\tilde{p}_i^\top}{\gamma_{i}\underline{s_{i}}}.
		\end{aligned}
	\end{equation*}
	We conclude that $p_k$ in \eqref{iter-krylov} satisfies $p_k=\frac{\tilde{p}_k}{\underline{s_k}}$. Indeed, let $\overline{p}_k:=\frac{\tilde{p}_k}{\underline{s_k}}$, then $\overline{p}_k$ has the following recursive relationship
	\begin{equation*}
		\begin{aligned}
			\overline{p}_{k+1}-d_{k+1}&=-V_{k+1}(V_{k+1}^\top V_{k+1})^{-1}V_{k+1}^\top d_{k+1}
			=-\sum_{i=j_{k+1}}^{k}\frac{\langle\tilde{p}_i,d_{k+1}\rangle}{\gamma_{i}\underline{s_{i}}}\tilde{p}_i
			\\
			&=-\sum_{i=j_{k+1}}^{k}\frac{\langle\tilde{p}_i,d_{k+1}\rangle}{\lVert \tilde{p}_{i}\rVert^2_2}\tilde{p}_i
			=-\sum_{i=j_{k+1}}^{k}\frac{\langle\overline{p}_i,d_{k+1}\rangle}{\lVert \overline{p}_{i}\rVert^2_2}\overline{p}_i,
		\end{aligned}
	\end{equation*}
	where the third equality follows from 
	\begin{equation}\label{proof-thm-xie-0329}
		\lVert \tilde{p}_i\rVert_2^2=\underline{s_i}^2(\lVert d_i\rVert_2^2-d_i^\top V_i(V_i^\top V_i)^{-1}V_i^\top d_i)=\gamma_i\underline{s_i}.
	\end{equation}
	Note that \(\overline{p}_0 = \frac{\tilde{p}_0}{\underline{s_0}} = \frac{x^1 - x^0}{\underline{s_0}} = -A^\top S_0 S_0^\top (Ax^0 - b) = p_0\). Hence, for \(k \geq 1\), we also have \(\overline{p}_k = p_k\). Consequently, we obtain that \(\overline{p}_k = p_k\) for all \(k \geq 0\), i.e., \(p_k = \frac{\tilde{p}_k}{\underline{s_k}}\). Thus, we have 
	\[
	x^{k+1} = x^k + \tilde{p}_k = x^k + \underline{s_k} p_k.
	\]
	Next, we show that \(\underline{s_k} = \delta_k\). Indeed, we have 
	\[
	\lVert p_k \rVert_2^2 = \frac{\lVert \tilde{p}_k \rVert_2^2}{\underline{s_k}^2} = \frac{\gamma_k \underline{s_k}}{\underline{s_k}^2} = \frac{\gamma_k}{\underline{s_k}},
	\]
	where the second equality follows from \eqref{proof-thm-xie-0329}. Hence, 
	\[
	\underline{s_k} = \frac{\gamma_k}{\lVert p_k \rVert_2^2} = \delta_k,
	\]
	where the last equality follows from Step 3  in Algorithm \ref{Algo-3}. This completes the proof of the theorem.
\end{proof}

  \begin{proof}[Proof of Proposition \ref{prop-property-Krylov}] (i) Since  $j_k\leq t< v\leq k$, we have
	\begin{equation}\label{proof-prop-0329-1}
		\begin{aligned}
			\langle p_{v}, p_t\rangle=&\frac{\langle x^{v+1}-x^v,x^{t+1}-x^t\rangle}{\delta_{v}\delta_t}=\frac{\langle x^{v+1}-A^\dagger b,x^{t+1}-x^t\rangle}{\delta_{v}\delta_t}+\frac{\langle A^\dagger b-x^v, x^{t+1}-x^t\rangle}{\delta_{v}\delta_t}\\
			=&\frac{\langle x^{v+1}-A^\dagger b,x^{t+1}-x^{v+1}\rangle}{\delta_{v}\delta_t}+\frac{\langle x^{v+1}-A^\dagger b,x^{v+1}-x^{t}\rangle}{\delta_{v}\delta_t}\\
			&+\frac{\langle A^\dagger b-x^v,x^{t+1}-x^v\rangle}{\delta_{v}\delta_t}+\frac{\langle A^\dagger b-x^v, x^v-x^t\rangle}{\delta_{v}\delta_t}
			\\=&0,
		\end{aligned}
	\end{equation}
	where the last equality follows from \eqref{equality-opt}. 
	
	(ii) For any $j_k\leq t< v\leq k$, we have 
	\[
	r^v=Ax^v-b=Ax^{v-1}+\delta_{v-1}Ap_{v-1}-b=r^{v-1}+\delta_{v-1}Ap_{v-1}=\cdots=r^{t}+\sum_{w=t}^{v-1}\delta_wAp_w,
	\]
	which together with  \eqref{iter-krylov} implies that
	\begin{equation*}
		\begin{aligned}
			\langle S_t^\top r^v,S_t^\top r^t\rangle&=\lVert S_t^\top r^t\rVert_2^2+\sum_{w=t}^{v-1}\delta_w\langle S_t^\top Ap_w,S_t^\top r^t\rangle 			
			=\lVert S_t^\top r^t\rVert_2^2-\sum_{w=t}^{v-1}\delta_w\langle p_w,d_t\rangle\\
			&=\lVert S_t^\top r^t\rVert_2^2-\sum_{w=t}^{v-1}\delta_w\langle p_w,p_{t}+\sum_{i=j_{t}}^{t-1}\eta_{t}^ip_i\rangle\\
			&=\lVert S_t^\top r^t\rVert_2^2-\sum_{w=t}^{v-1}\delta_w\langle p_w,p_{t}\rangle-\sum_{w=t}^{v-1}\sum_{i=j_{t}}^{t-1}\delta_w\eta_{t}^i\langle p_w,p_i\rangle\\
			&=\lVert S_t^\top r^t\rVert_2^2-\delta_t\lVert p_t\rVert_2^2-\sum_{w=t}^{v-1}\sum_{i=j_{t}}^{t-1}\delta_w\eta_{t}^i\langle p_w,p_i\rangle
			=-\sum_{w=t}^{v-1}\sum_{i=j_{t}}^{t-1}\delta_w\eta_{t}^i\langle p_w,p_i\rangle.
		\end{aligned}
	\end{equation*}
	Therefore, if $v=t+1$, we have 
	$$
	\langle S_t^\top r^v,S_t^\top r^t\rangle=-\sum_{w=t}^{t}\sum_{i=j_{t}}^{t-1}\delta_w\eta_{t}^i\langle p_w,p_i\rangle=-\delta_t\sum_{i=j_{t}}^{t-1}\eta_{t}^i\langle p_t,p_i\rangle=0.
	$$
	If $j_t=j_k$, then from (i) we have
	$$
	\langle S_t^\top r^v,S_t^\top r^t\rangle=-\sum_{w=t}^{v-1}\sum_{i=j_{t}}^{t-1}\delta_w\eta_{t}^i\langle p_w,p_i\rangle=-\sum_{w=t}^{v-1}\sum_{i=j_{k}}^{t-1}\delta_w\eta_{t}^i\langle p_w,p_i\rangle=0.
	$$
	Otherwise, from (i) we have
	$$
	\langle S_t^\top r^v,S_t^\top r^t\rangle=-\sum_{w=t}^{v-1}\sum_{i=j_{t}}^{t-1}\delta_w\eta_{t}^i\langle p_w,p_i\rangle=-\sum_{w=t}^{v-1}\sum_{i=j_{t}}^{j_k-1}\delta_w\eta_{t}^i\langle p_w,p_i\rangle.
	$$
	This completes the proof of this proposition.
\end{proof}

\end{document}